\documentclass[11pt]{article}
\usepackage{fullpage}
\usepackage[numbers]{natbib}
\usepackage{hyperref}
\usepackage[svgnames]{xcolor}
\usepackage{overpic}
\usepackage{tikz}
\usepackage{rotating}

\hypersetup{colorlinks=true,citecolor=blue,linkcolor=MediumSlateBlue, urlcolor=DarkBlue}

\usepackage{macros}
\usepackage{color}
\usepackage{bbm}
\usepackage{nicefrac}
\usepackage{psfrag}
\usepackage{xifthen} 
\usepackage{xr}

\newcommand{\paragraphc}[1]{\paragraph{{#1}:}}

\makeatletter
\long\def\@makecaption#1#2{
  \vskip 0.8ex
  \setbox\@tempboxa\hbox{\small {\bf #1:} #2}
  \parindent 1.5em  
  \dimen0=\hsize
  \advance\dimen0 by -3em
  \ifdim \wd\@tempboxa >\dimen0
  \hbox to \hsize{
    \parindent 0em
    \hfil 
    \parbox{\dimen0}{\def\baselinestretch{0.96}\small
      {\bf #1.} #2
    } 
    \hfil}
  \else \hbox to \hsize{\hfil \box\@tempboxa \hfil}
  \fi
}
\makeatother

\long\def\comment#1{}

\newcommand{\thetatil}{\ensuremath{\thetahat}}
\newcommand{\widgraph}[2]{\includegraphics[keepaspectratio,width=#1]{#2}}
\newcommand{\perr}{\ensuremath{P_{\mbox{\tiny{err}}}}}

\newcommand{\PRFAM}{\ensuremath{\mc{\statprob}_{p, \radius}}}
\newcommand{\PRFAMINF}{\ensuremath{\mc{\statprob}_{\infty, \radius}}}

\newcommand{\DOUBLEHACK}[2]{\ensuremath{#1(#2)}}
\newcommand{\HACKI}[1]{\DOUBLEHACK{#1}{i}}
\newcommand{\usedim}{\ensuremath{d}}

\newcommand{\Xtil}{\ensuremath{\widetilde{X}}}
\newcommand{\inprod}[2]{\ensuremath{\langle #1 , \, #2 \rangle}}

\newcommand{\ncensored}{N_{\rm cens}}
\newcommand{\nuncensored}{N_{\rm un}}

\begin{document}

\begin{center}
  {\Large{\bf{Minimax Optimal Procedures for Locally Private
        Estimation}}}

  \vspace*{.3in}
  
  \begin{tabular}{ccc}
    John C.\ Duchi$^\dagger$ & Michael I.\ Jordan$^{\ast}$ &
    Martin J.\ Wainwright$^{\ast}$
    \\ \href{mailto:jduchi@stanford.edu}{jduchi@stanford.edu} &
    \href{mailto:jordan@stat.berkeley.edu}{jordan@stat.berkeley.edu} &
    \href{mailto:wainwrig@berkeley.edu}{wainwrig@berkeley.edu}
  \end{tabular}
  
  \vspace*{.5cm}
  
  \begin{tabular}{ccc}
    Stanford University$^\dag$ &&
    University of California, Berkeley$^\ast$ \\
    Stanford, CA 94305 &&
    Berkeley, CA 94720
  \end{tabular}
  
  \vspace*{.2in}
\end{center}

\begin{abstract}  
  Working under a model of privacy in which data remains private even from
  the statistician, we study the tradeoff between privacy guarantees and the
  risk of the resulting statistical estimators.  We develop private versions
  of classical information-theoretic bounds, in particular those due to
  Le Cam, Fano, and Assouad.  These inequalities allow for a precise characterization
  of statistical rates under local privacy constraints and the development
  of provably (minimax) optimal estimation procedures. We provide a
  treatment of several canonical families of problems: mean estimation and
  median estimation, generalized linear models, and nonparametric
  density estimation. For all of these families, we provide lower and upper
  bounds that match up to constant factors, and exhibit new (optimal)
  privacy-preserving mechanisms and computationally efficient estimators
  that achieve the bounds. Additionally, we present a variety of
  experimental results for estimation problems involving sensitive data,
  including salaries, censored blog posts and articles, and drug abuse;
  these experiments demonstrate the importance of deriving optimal
  procedures.
\end{abstract}

%
%
%

\section{Introduction}

A major challenge in statistical inference is that of characterizing
and balancing statistical utility with the privacy of individuals from
whom data is obtained~\cite{DuncanLa86,DuncanLa89,FienbergMaSt98}.
Such a characterization requires a formal definition of privacy, and
\emph{differential privacy} has been put forth as one such candidate
(see, e.g., the papers~\cite{DworkMcNiSm06, BlumLiRo08, DworkRoVa10,
  HardtRo10, HardtTa10} and references therein).  In the database and
cryptography literatures from which differential privacy arose, early
research was mainly algorithmic in focus, with researchers using
differential privacy to evaluate privacy-retaining mechanisms for
transporting, indexing, and querying data.  More recent work aims to
link differential privacy to statistical concerns~\cite{DworkLe09,
  WassermanZh10, HallRiWa11, Smith11, ChaudhuriMoSa11,
  RubinsteinBaHuTa12}; in particular, researchers have developed
algorithms for private robust statistical estimators, point and
histogram estimation, and principal components analysis.  Much of this
line of work is non-inferential in nature: as opposed to studying
performance relative to an underlying population, the aim instead
has been to approximate a class of statistics under privacy-respecting
transformations for a fixed underlying data set.  There has also been
recent work within the context of classification problems and the
``probably approximately correct'' framework of statistical learning
theory~\cite[e.g.,][]{KasiviswanathanLeNiRaSm11,BeimelKaNi10} that
treats the data as random and aims to recover aspects of the
underlying population.

In this paper, we take a fully inferential point of view on privacy,
bringing differential privacy into contact with statistical
decision theory.  Our focus is on the fundamental limits of
\mbox{differentially-private} estimation, and the identification of
optimal mechanisms for enforcing a given level of privacy.  By
treating differential privacy as an abstract constraint on estimators,
we obtain independence from specific estimation procedures and
privacy-preserving mechanisms.  Within this framework, we derive both
lower bounds and matching upper bounds on minimax risk.  We obtain our
lower bounds by integrating differential privacy into the classical
paradigms for bounding minimax risk via the inequalities of Le Cam,
Fano, and Assouad, while we obtain matching upper bounds by proposing
and analyzing specific private procedures.

Differential privacy provides one formalization of the notion of
``plausible deniability'': no matter what public data is released, it
is nearly equally as likely to have arisen from one underlying private
sample as another.  It is also possible to interpret differential privacy 
within a hypothesis testing framework~\citep{WassermanZh10},
where the differential privacy parameter $\alpha$ controls the error
rate in tests for the presence or absence of individual data points in
a dataset (see Figure~\ref{FigDisclosure} for more details).  Such
guarantees against discovery, together with the treatment of issues of
side information or adversarial strength that are problematic for
other formalisms, have been used to make the case for differential
privacy within the computer science literature; see, for example, the
papers~\cite{EvfimievskiGeSr03, DworkMcNiSm06, BarakChDwKaMcTa07,
GantaKaSm08}. In this paper we bring this approach into contact
with minimax decision theory; we view the minimax framework as natural
for this problem because of the tension between adversarial discovery and privacy
protection. Moreover, we study the setting of \emph{local
privacy}, in which providers do not even trust the statistician
collecting the data. Although local privacy is a relatively stringent
requirement, we view this setting as an important first step in formulating
minimax risk bounds under privacy constraints. Indeed, local privacy
is one of the oldest forms of privacy: its essential form dates back
to~\citet{Warner65}, who proposed it as a remedy for what he termed
``evasive answer bias'' in survey sampling.

Although differential privacy provides an elegant formalism for
limiting disclosure and protecting against many forms of privacy
breach, it is a stringent measure of privacy, and it is conceivably
overly stringent for statistical practice.  Indeed,
\citet{FienbergRiYa10} criticize the use of differential privacy in
releasing contingency tables, arguing that known mechanisms for
differentially private data release can give unacceptably poor
performance.  As a consequence, they advocate---in some
cases---recourse to weaker privacy guarantees to maintain the utility
and usability of released data.  There are results that are more
favorable for differential privacy; for example, \citet{Smith11} shows
that the non-local form of differential privacy~\cite{DworkMcNiSm06}
can be satisfied while yielding asymptotically optimal parametric
rates of convergence for some point estimators.  Resolving such
differing perspectives requires investigation into whether particular
methods have optimality properties that would allow a general
criticism of the framework, and characterizing the trade-offs between
privacy and statistical efficiency.  Such are the goals of the current
paper.


\subsection{Our contributions}

In this paper, we provide a formal framework for characterizing the
tradeoff between statistical utility and local differential privacy.
Basing our work on the classical minimax framework, our primary goals
are to characterize how, for various types of estimation problems, the
optimal rate of estimation varies as a function of the privacy level
and other problem parameters.  Within this framework, we develop a
number of general techniques for deriving minimax bounds under local
differential privacy constraints.

These bounds are useful in that they not only characterize the
``statistical price of privacy,'' but they also allow us to compare
different concrete procedures (or privacy mechanisms) for producing
private data.  Most importantly, our minimax theory can be used to
identify which mechanisms are optimal, meaning that they preserve the
maximum amount of statistical utility for a given privacy level.
In practice, we find that these optimal mechanisms often differ
from widely-accepted procedures from the privacy literature, and lead
to better statistical performance while providing the same privacy
guarantee. As one concrete example,
Figure~\ref{fig:linf-mean-estimation} provides an illustration of such
gains in the context of estimating proportions of drug users based
on privatized data.  (See Section~\ref{sec:drug-use} for a full 
description of the data set, and how these plots were produced.)
The black curve shows the average $\ell_\infty$-error of a non-private
estimator, based on access to the raw data; any estimator that
operates on private data must necessarily have larger error than this
gold standard.  The upper two curves show the performance of two types
of estimators that operate on privatized data: the blue curve is based
on the standard mechanism of adding Laplace-distributed noise to the data, 
whereas the green curve is based on the optimal privacy mechanism identified 
by our theory.  This optimal mechanism yields a roughly five-fold reduction
in the mean-squared error, with the same computational requirements as
the Laplacian-based procedure.

\begin{figure}
  \begin{center}
    \begin{overpic}[width=.55\columnwidth]{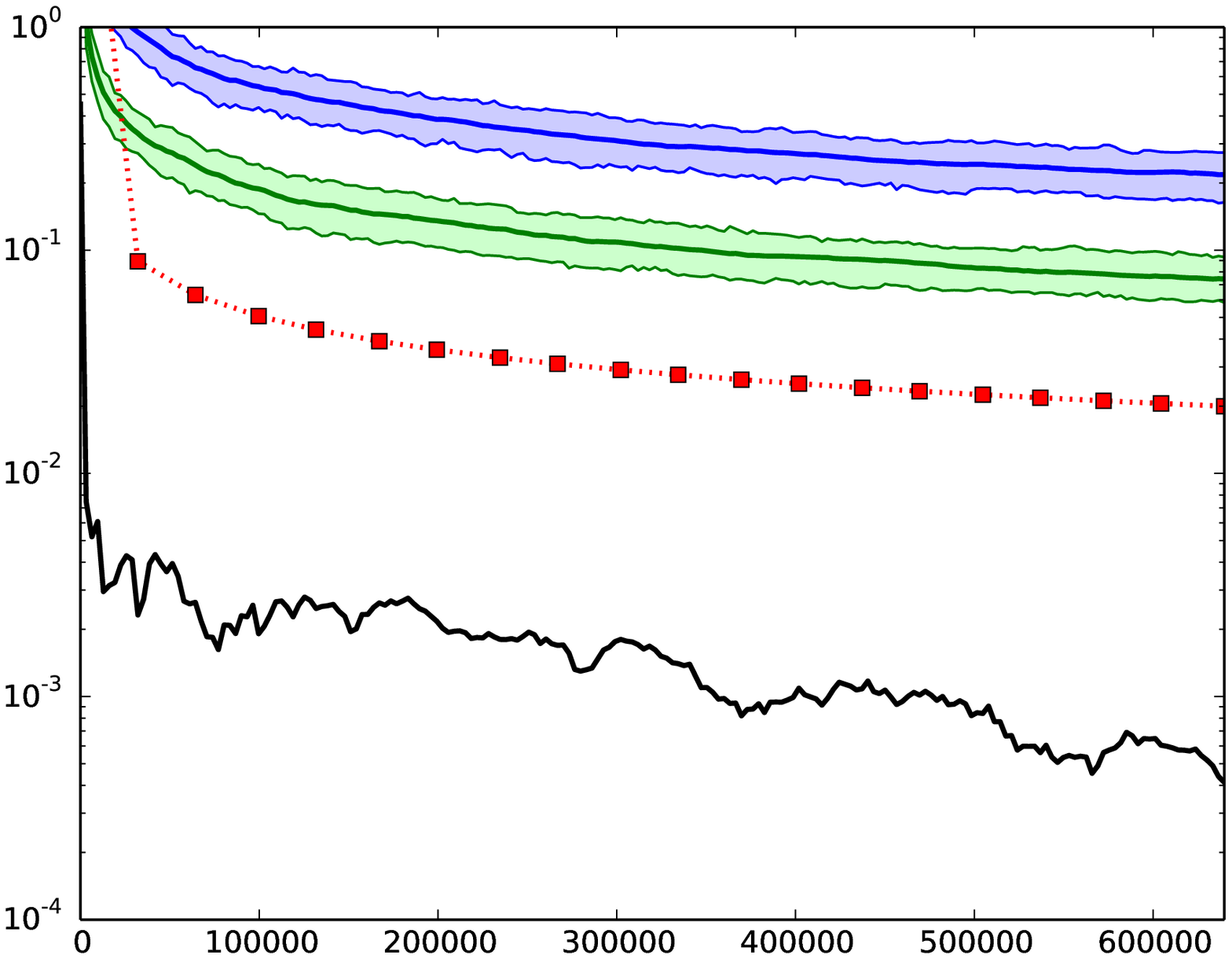}
      \put(-4,32){\rotatebox{90}{\small{$\linfs{\what{\theta} - \theta}$}}}
      \put(50,-1){\small{$n$}}
    \end{overpic}
    \caption{\label{fig:linf-mean-estimation} Estimating proportions
      of drug use, plotting the maximum error $\linfs{\what{\theta} -
        \theta}$ versus sample size. Top (blue) line: mean error of
      Laplace noise addition with 90\% coverage intervals. Middle top
      (green) line: mean error of minimax-optimal $\ell_\infty$
      sampling strategy~\eqref{eqn:linf-sampling} with 90\% coverage
      intervals. Middle bottom (red block) line:
      minimax lower bound $\sqrt{d \log (2d) / (n (e^\diffp - 1)^2)}$.
      Bottom (black) line: mean error of non-private
      estimate.}
  \end{center}
\end{figure}

In this paper, we analyze the private minimax rates of estimation for
several canonical problems: (a) mean estimation; (b) median estimation; (c)
high-dimensional and sparse sequence estimation; (d) generalized linear
model estimation; (e) density estimation. To do so, we expand upon several
canonical techniques for lower bounding minimax risk~\cite{Yu97},
establishing differentially private analogues of Le Cam's method in
Section~\ref{sec:le-cam} and concomitant optimality guarantees for mean and
median estimators; Fano's method in Section~\ref{sec:fano}, where we provide
optimal procedures for high-dimensional estimation; and Assouad's method in
Section~\ref{sec:assouad}, in which we investigate generalized linear models
and density estimation.
In accordance with our connections to statistical decision theory, we
provide minimax rates for estimation of \emph{population} quantities; by way
of comparison, most prior work in the privacy literature focuses on accurate
approximation of statistics in a conditional analysis in which the data are
treated as fixed (with some exceptions; e.g., the
papers~\cite{Smith11,KarwaSl16,BeimelNiOm08}, as well as preliminary
extended abstracts of our own work~\cite{DuchiJoWa13_focs,
  DuchiJoWa13_nips}).


\ifdefined\compareprevious
We now take a moment to situate the current paper in the context of our own
previous work~\cite{DuchiJoWa14}. The earlier paper presents results on
stochastic optimization, and to provide guarantees it assumes a particular
noise model for preventing disclosure that restricts its applicability
specifically to such optimization problems.  In contrast, here we provide a
suite of new results via new techniques; as noted in the preceding
paragraphs, we give private versions of three classical minimax
techniques---those due to Le Cam, Fano and Assouad---and show how to apply
them to several estimation problems. We also note that preliminary abstracts
of this paper have appeared in conferences~\cite{DuchiJoWa13_focs,
  DuchiJoWa13_nips}, though this paper also differs from these
abstracts. Briefly, the current manuscript provides results for
\emph{interactive estimation}, meaning that the channel used to prevent
disclosure of sensitive information may adapt to data seen
previously. Additionally, most of the examples---including regression,
high-dimensional mean estimation, and multinomial and density estimation in
interactive settings---are new. Finally, the current manuscript provides a
new, general form of Assouad's inequality appropriate for privacy problems,
which we view as a major contribution of the current manuscript.
\fi


\paragraphc{Notation}

For distributions $P$ and $Q$ defined on a space $\statdomain$, each
absolutely continuous with respect to a measure $\mu$ (with
corresponding densities $p$ and $q$), the KL divergence between $P$
and $Q$ is
\begin{equation*}
  \dkl{P}{Q} \defeq \int_\statdomain dP \log \frac{dP}{dQ} =
  \int_\statdomain p \log \frac{p}{q} d\mu.
\end{equation*}
Letting $\sigma(\statdomain)$ denote an appropriate $\sigma$-field 
on $\statdomain$, the total variation distance between $P$ and $Q$ is
\begin{equation*}
  \tvnorm{P - Q} \defeq \sup_{S \in \sigma(\statdomain)} |P(S) - Q(S)|
  = \half \int_\statdomain \left|p(\statsample) -
  q(\statsample)\right| d\mu(\statsample).
\end{equation*}
Given a pair of random variables $(X,Y)$ with joint distribution $P_{X, Y}$,
their mutual information is given by $\information(X; Y) = \dkl{P_{X,
    Y}}{P_X P_Y}$, where $P_X$ and $P_Y$ denote the marginal
distributions. A random variable $Y$ has the $\laplace(\alpha)$ distribution
if its density is $p_Y(y) = \frac{\alpha}{2} \exp\left(-\alpha |y|\right)$.
For matrices $A, B \in \R^{d \times d}$, the notation $A \preceq B$ means
that $B - A$ is positive semidefinite.  For sequences of real numbers
$\{a_n\}$ and $\{b_n\}$, we use $a_n \lesssim b_n$ to mean there is a
universal constant $C < \infty$ such that $a_n \leq C b_n$ for all $n$, and
$a_n \asymp b_n$ to denote that $a_n \lesssim b_n$ and $b_n \lesssim a_n$.
For a sequence of random variables $X_n$, we write $X_n \cd Y$ if $X_n$
converges in distribution to $Y$.


\section{Background and problem formulation}
\label{SecEstimationToTest}

We begin by setting up the classical minimax framework, and then
introducing the notion of an $\diffp$-private minimax rate that we
study in this paper.

\subsection{Classical minimax framework}

Let $\mc{P}$ denote a class of distributions on the sample space
$\statdomain$, and let $\optvar(\statprob) \in \parameterspace$ denote
a functional defined on $\mc{P}$.  The space $\Theta$ in which the
parameter $\optvar(\statprob)$ takes values depends on the underlying
statistical model.  For example, in the case of univariate mean
estimation, $\Theta$ is a subset of the real line, whereas for a density
estimation problem, $\Theta$ is some subset of the space of all possible
densities over $\statdomain$.  Let $\metric$ denote a semi-metric on
the space $\parameterspace$, which we use to measure the error of an
estimator for the parameter $\optvar$, and let $\Phi : \R_+
\rightarrow \R_+$ be a non-decreasing function with $\Phi(0) = 0$ (for
example, $\Phi(t) = t^2$).

In the non-private setting, the statistician is given direct access to
i.i.d.\ observations $\{ X_i\}_{i=1}^\numobs$ drawn according to some
distribution $\statprob \in \mc{P}$.  Based on the observations, the goal
is to estimate the unknown parameter $\theta(\statprob) \in \Theta$.
We define an estimator $\thetatil$ as a measurable function 
$\thetatil: \statdomain^n \rightarrow \optdomain$, and we assess 
the quality of the estimate $\thetatil(X_1, \ldots, X_\numobs)$
in terms of the risk
\begin{align*}
  \Exs_{\statprob} \left[ \Phi\big(\metric(\thetatil(X_1, \ldots,
    X_\numobs), \theta(\statprob))\big) \right].
\end{align*}
For instance, for a univariate mean problem with $\metric(\thetatil,
\theta) = |\thetatil - \theta|$ and $\Phi(t) = t^2$, this risk is the
mean-squared error.  The risk assigns a nonnegative number to each
pair $(\thetatil, \theta)$ of estimator and parameter.

The minimax risk is defined by the saddlepoint problem
\begin{align}
  \label{EqnClassicalMinimax}
  \minimax_n(\theta(\mc{P}), \Phi \circ \metric) \defeq
  \inf_{\thetatil} \sup_{\statprob \in \mc{P}}
  \E_{\statprob}\left[\Phi\big(\metric(\thetatil(\statrv_1, \ldots,
    \statrv_n), \theta(\statprob))\big)\right],
\end{align}
where we take the supremum over distributions $\statprob \in \mc{P}$ and the
infimum over all estimators $\thetatil$.  There is a substantial body of
literature focused on techniques for upper- and lower-bounding the minimax
risk for various classes of estimation problems.  Our goal in this paper is
to define and study a modified version of the minimax risk that accounts for
privacy constraints.

\subsection{Local differential privacy}

Let us now define the notion of local differential privacy in a
precise manner.  The act of transforming data from the raw samples
$\{X_i\}_{i=1}^n$ into a private set of samples
$\{\channelrv_i\}_{i=1}^n$ is modeled by a conditional distribution.
We refer to this conditional distribution as either a \emph{privacy
  mechanism} or a \emph{channel distribution}, as it acts as a conduit
from the original to the privatized data.  In general, we allow the
privacy mechanism to be \emph{sequentially interactive}, meaning the
channel has the conditional independence structure
\begin{align}
  \label{EqnCondInd}
  \{\statrv_i, \channelrv_1, \ldots, \channelrv_{i-1}\} \to
  \channelrv_i ~~~ \mbox{and} ~~~ \channelrv_i \perp \statrv_j \mid
  \{\statrv_i, \channelrv_1, \ldots, \channelrv_{i-1}\} ~ \mbox{for~}
  j \neq i.
\end{align}
See panel (a) of Figure~\ref{fig:interactive-channel} for a graphical model
that illustrates these conditional independence relationships.  Given these
independence relations~\eqref{EqnCondInd}, the full conditional distribution
(and privacy mechanism) can be specified in terms of the conditionals
$\channelprob_i(\channelrv_i \mid \statrv_i = \statsample_i,
\channelrv_{1:i-1} = \channelval_{1:i-1})$.  A special case is the
\emph{non-interactive} case, illustrated in panel (b) of
Figure~\ref{fig:interactive-channel}, in which each $\channelrv_i$ depends
only on $\statrv_i$.  In this case, the conditional distributions take the
simpler form $\channelprob_i(\channelrv_i \mid \statrv_i = \statsample_i)$.
While it is often simpler to think of the channel as being independent and
the (privatized) sample being i.i.d., in a number of scenarios it is
convenient for the output of the channel $\channel$ to depend on previous
computation. For example, stochastic approximation schemes~\cite{PolyakJu92}
require this type of dependence, and---as we demonstrate in
Sections~\ref{sec:median-estimation} (median estimation)
and~\ref{sec:glm-estimation} (generalized linear models)---this type of
conditional structure makes developing optimal estimators substantially
easier.

\begin{figure}[t]
  \begin{center}
    \begin{tabular}{ccc}
      \begin{overpic}[width=.36\columnwidth]
        {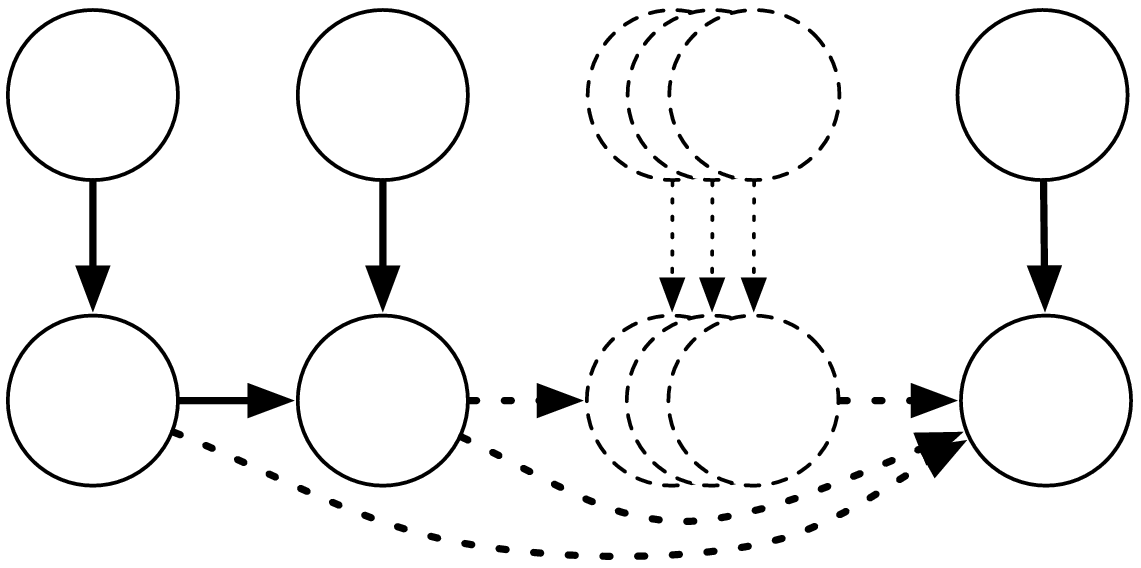} \put(5,13){$\channelrv_1$}
        \put(30,13){$\channelrv_2$} \put(88,13){$\channelrv_n$}
        \put(5,39.5){$\statrv_1$} \put(30,39.5){$\statrv_2$}
        \put(88,39.5){$\statrv_n$}
      \end{overpic} & \hspace{.4in}&
      \hspace{.5cm}
      \begin{overpic}[width=.36\columnwidth]
        {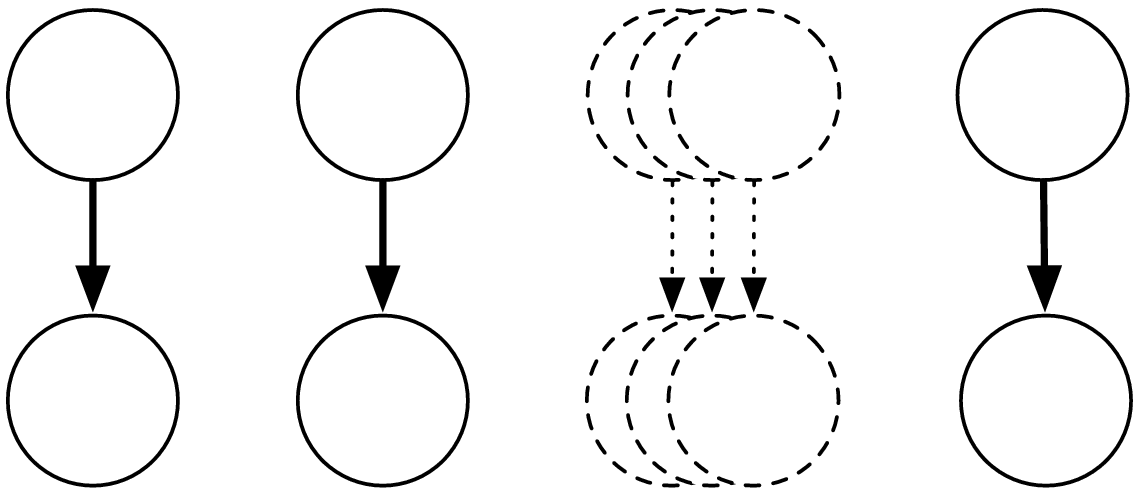} \put(5,13){$\channelrv_1$}
        \put(30,13){$\channelrv_2$} \put(88,13){$\channelrv_n$}
        \put(5,39.5){$\statrv_1$} \put(30,39.5){$\statrv_2$}
        \put(88,39.5){$\statrv_n$}
      \end{overpic} \\
      (a) && (b)
    \end{tabular}
    \caption{\label{fig:interactive-channel} (a) Graphical model
      illustrating the conditional independence relationships between
      the private data $\{\channelrv_i\}_{i=1}^\numobs$ and raw
      variables $\{\statrv_i\}_{i=1}^\numobs$ in the interactive case.
      (b) Simpler graphical model illustrating the conditional
      independence structure in the non-interactive case.}
  \end{center}
\end{figure}

Local differential privacy involves placing some
restrictions on the conditional distribution
$\channelprob_i$.
\begin{definition}
  For a given privacy parameter $\diffp \geq 0$, the random variable
  $\channelrv_i$ is an $\diffp$-\emph{differentially locally private
  view} of $\statrv_i$ if for all $\channelval_1, \ldots,
  \channelval_{i-1}$ and $\statsample, \statsample' \in \statdomain$ we
  have
  \begin{equation}
    \label{eqn:local-privacy}
    \sup_{S \in \sigma(\channeldomain)}
    \frac{\channelprob_i(\channelrv_i \in S \mid \statrv_i =
      \statsample, \channelrv_1 = \channelval_1, \ldots,
      \channelrv_{i-1} = \channelval_{i-1})}{
      \channelprob_i(\channelrv_i \in S \mid \statrv_i = \statsample',
      \channelrv_1 = \channelval_1, \ldots, \channelrv_{i-1} =
      \channelval_{i-1})}
    \le \exp(\diffp),
  \end{equation}
  where $\sigma(\channeldomain)$ denotes an appropriate $\sigma$-field
  on $\channeldomain$.  We say that the privacy mechanism $\channelprob$
  is $\diffp$-differentially locally private (DLP) if each variable
  $\channelrv_i$ is an $\diffp$-DLP view.
\end{definition}

In the non-interactive case, the bound~\eqref{eqn:local-privacy}
reduces to
\begin{equation}
  \label{eqn:local-privacy-simple}
  \sup_{S \in \sigma(\channeldomain)} \sup_{\statsample, \statsample'
    \in \statdomain} \frac{\channelprob(\channelrv_i \in S \mid \statrv_i =
    \statsample)}{
    \channelprob(\channelrv_i \in S \mid \statrv_i = \statsample')}
  \leq \exp(\diffp).
\end{equation}
The non-interactive version of local differential privacy dates back to the
work of~\citet{Warner65}; see also~\citet{EvfimievskiGeSr03}.  The more
general interactive model was put forth by~\citet*{DworkMcNiSm06}, and has
been investigated in a number of works since then. In the context of our
work on local privacy, relevant references include
\citeauthor{BeimelNiOm08}'s~\cite{BeimelNiOm08} investigation of
one-dimensional Bernoulli probability estimation under the
model~\eqref{eqn:local-privacy}, and \citeauthor{KairouzOhVi14}'s~\cite{KairouzOhVi14} 
study of channel constructions that maximize information-theoretic
measures of information content for various domains $\mc{X}$.

\begin{figure}[ht]
  \begin{center}
    \widgraph{.4\textwidth}{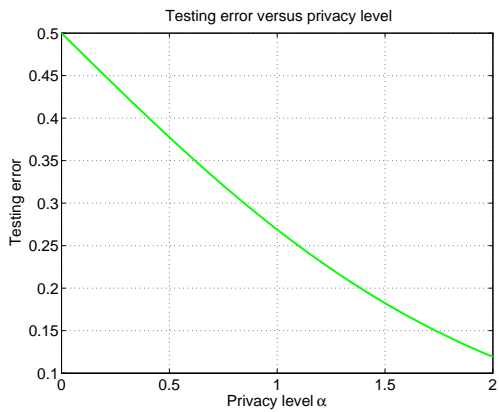}
    \caption{For any $\diffp$-DLP view $Z$ of $X$, the probability of
      error $\perr$ in distinguishing between the two hypotheses $\{X =
      x\}$ and $\{X = x'\}$ is lower bounded as $\perr \geq \frac{1}{1 +
        e^\diffp}$.  Thus, for $\diffp \in [0, \frac{1}{4}]$, we are
      guaranteed that $\perr \geq 0.43$.}
    \label{FigDisclosure}
  \end{center}
\end{figure}
As \citet{WassermanZh10} discuss, one intuitive interpretation of
differential privacy is in terms of disclosure risk. 
More concretely, suppose that given an $\diffp$-private view $Z$ of the
random variable $X$, our goal is to distinguish between the two
hypotheses $\{X = x \}$ versus $\{X = x'\}$, where $x, x' \in
\statdomain$ are two distinct possible values in the sample space.  A
calculation shows that the best possible probability of error
of any hypothesis test, with equal weights for each hypothesis, satisfies
\begin{equation*}
  \perr \defeq \half \cdot \inf_\test \left\{ \P(\test(Z) \neq x \mid X = x)
  + \P(\test(Z) \neq x' \mid X = x') \right\} \ge \frac{1}{1 + e^\diffp}.
\end{equation*}
Consequently, small values of $\diffp$ ensure that the performance of any
test is close to random guessing (see Figure~\ref{FigDisclosure}). We relate
this in passing to Warner's classical randomized response mechanism~\cite{Warner65} 
in a simple scenario with $X \in \{0, 1\}$, where we set $Z = X$ with 
probability $q_\diffp = \frac{e^\diffp}{1 + e^\diffp}$, and $Z
= 1 - X$ otherwise. Then $\channel(Z = z \mid X = x) / \channel(Z = z \mid X
= x') \in [e^{-\diffp}, e^\diffp]$, and the disclosure risk $\perr$ is
precisely $1 / (1 + e^\diffp)$.



\subsection{$\diffp$-private minimax risks}

Given our definition of local differential privacy (LDP), we are now
equipped to describe the notion of an $\diffp$-LDP minimax risk.  For
a given privacy level $\diffp > 0$, let $\QFAMA$ denote the set of all
conditional distributions have the $\diffp$-LDP
property~\eqref{eqn:local-privacy}.  For a given raw sample
$\{\statrv_i\}_{i=1}^\numobs$, any distribution $\channelprob \in
\QFAMA$ produces a set of private observations
$\{\channelrv_i\}_{i=1}^\numobs$, and we now restrict our attention to
estimators $\thetahat = \thetahat(\channelrv_1, \ldots,
\channelrv_\numobs)$ that depend purely on this private sample.
Doing so yields the modified minimax risk
\begin{align}
\label{EqnModifiedMinimax}
  \minimax_n(\theta(\mc{P}), \Phi \circ \metric; \channelprob) &
  \defeq \inf_{\thetahat} \sup_{\statprob \in \mc{P}}
  \E_{\statprob,\channel}\left[\Phi\big(\metric(\thetahat(\channelrv_1, \ldots,
    \channelrv_n), \theta(\statprob))\big)\right],
\end{align}
where our notation reflects the dependence on the choice of
privacy mechanism $\channelprob$.  By definition, any choice of
$\channelprob \in \QFAMA$ guarantees that the data
$\{\channelrv_i\}_{i=1}^\numobs$ are $\diffp$-locally differentially
private, so that it is natural to seek the mechanism(s) in $\QFAMA$
that lead to the smallest values of the minimax
risk~\eqref{EqnModifiedMinimax}.  This minimization problem leads to
the central object of study for this paper, a functional which
characterizes the optimal rate of estimation in terms of the privacy
parameter $\diffp$.
\begin{definition}
Given a family of distributions $\theta(\mc{P})$ and a privacy
parameter $\diffp > 0$, the \emph{$\diffp$-private minimax risk} in
the metric $\metric$ is
  \begin{align}
    \label{EqnPrivateMinimax}
    \minimax_n(\theta(\mc{P}), \Phi \circ \metric, \diffp) & \defeq
    \inf_{\channelprob \in \QFAMA} \inf_{\what{\theta}}
    \sup_{\statprob \in \mc{P}} \E_{\statprob,
      \channelprob}\left[\Phi(\metric(\what{\optvar}(\channelrv_1,
      \ldots, \channelrv_n), \optvar(\statprob)))\right].
  \end{align}
\end{definition}
Note that as $\diffp \rightarrow +\infty$, the constraint of membership
in $\QFAMA$ becomes vacuous, so that the $\diffp$-private minimax risk
reduces to the classical minimax risk~\eqref{EqnClassicalMinimax}.  Of
primary interest in this paper are settings in which $\diffp \in
\openleft{0}{1}$, corresponding to reasonable levels of the disclosure risk,
as illustrated in Figure~\ref{FigDisclosure}.


\section{Bounds on pairwise divergences: Le Cam's bound and variants}
\label{sec:le-cam}

Perhaps the oldest approach to bounding the classical minimax
risk~\eqref{EqnClassicalMinimax} is via Le Cam's
method~\cite{LeCam73}. Beginning with this technique, we develop a private
analogue of the Le Cam bound, and we show how it can be used to derive sharp
lower bounds on the $\diffp$-private minimax risk for one-dimensional mean
and median estimation problems. We also provide new optimal
procedures for each of these settings.

\subsection{A private version of Le Cam's bound}
\label{sec:pairwise-corollaries}

The classical version of Le Cam's method bounds the (non-private)
minimax risk~\eqref{EqnClassicalMinimax} in terms of a two-point hypothesis
testing problem~\cite{LeCam73,Yu97,Tsybakov09}.  For any distribution
$\statprob$, we use $\statprob^\numobs$ to denote the product distribution
corresponding to a collection of $\numobs$ i.i.d.\ samples.  Let us say that
a pair of distributions $\{\statprob_1, \statprob_2\}$ is $2
\delta$-separated with respect to $\theta$ if $\metric(\theta(\statprob_1),
\theta(\statprob_2)) \ge 2 \delta$.  With this terminology, a simple version
of Le Cam's lemma asserts that, for any $2\delta$-separated pair of
distributions, the classical minimax risk has lower bound
\begin{align}
  \label{EqnLeCam}
  \minimax_n(\theta(\mc{P}), \Phi \circ \metric)
  \geq \frac{\Phi(\delta)}{2} \; \Big \{ 1 -
  \tvnorm{\statprob_1^\numobs - \statprob_2^\numobs} \Big \} \;
  \stackrel{(i)}{\geq} \; \frac{\Phi(\delta)}{2} \; \Big \{ 1 -
  \frac{1}{\sqrt{2}} \sqrt{\numobs \dkl{\statprob_1}{\statprob_2}}
  \Big \}.
\end{align}
Here the bound (\emph{i}) follows as a
consequence of the Pinsker bound on the total variation norm in terms
of the KL divergence,
\begin{align*}
  \tvnorm{\statprob_1^\numobs - \statprob_2^\numobs}^2 \leq \frac{1}{2}
  \dkl{\statprob_1^\numobs}{\statprob_2^\numobs},
\end{align*}
along with the fact $\dkl{\statprob_1^\numobs}{\statprob_2^\numobs} =
\numobs \dkl{\statprob_1}{\statprob_2}$ because
$\statprob_1^n$ and $\statprob_2^n$ are product
distributions (i.e., we have $\statrv_i \simiid P$).

Let us now state a version of Le Cam's lemma that applies to the
$\diffp$-locally private setting in which the estimator $\thetahat$
depends only on the private variables $(Z_1, \ldots, Z_\numobs)$, and
our goal is to lower bound the $\diffp$-private minimax
risk~\eqref{EqnPrivateMinimax}.

\begin{proposition}[Private form of Le Cam bound]
  \label{PropPrivateLeCam}
  Suppose that we are given $\numobs$ i.i.d.\ observations from an
  $\diffp$-locally differential private channel for some $\diffp \in [0,
    \frac{23}{35}]$.  Then for any pair of
  distributions $(\statprob_1, \statprob_2)$
  that is $2 \delta$-separated w.r.t.\ $\optvar$, the $\diffp$-private
  minimax risk has lower bound
  \begin{align}
    \label{eqn:le-cam-private}
    \minimax_n(\theta(\mc{P}), \Phi \circ \metric, \diffp) & \geq
    \frac{\Phi(\delta)}{2} \; \Big \{ 1 - \sqrt{ 4 \diffp^2 \numobs
      \tvnorm{\statprob_1 - \statprob_2}^2} \Big \}
    \; \geq \;
    \frac{\Phi(\delta)}{2} \; \Big \{ 1 - \sqrt{2
      \diffp^2 \numobs \dkl{\statprob_1}{\statprob_2}} \Big \}.
  \end{align}
\end{proposition}
\noindent

\paragraph{Remarks:} A 
comparison with the original Le Cam bound~\eqref{EqnLeCam} shows that for
$\diffp \in [0, \frac{23}{35}]$, the effect of $\diffp$-local differential
privacy is to reduce the \emph{effective sample size} from $\numobs$ to at
most $4 \diffp^2 \numobs$.  Proposition~\ref{PropPrivateLeCam} is a
corollary of more general results, which we describe in
Section~\ref{sec:pairwise-kl-bounds}, that quantify the contraction in
KL divergence that arises from passing the data through an $\diffp$-private
channel. We also note that here---and in all subsequent bounds in the
paper---we may replace the term $\diffp^2$ with $(e^\diffp - 1)^2$,
which are of the same order for $\diffp = \order(1)$, while the latter
substitution always applies.


\subsection{Some applications of the private Le Cam bound}
\label{sec:initial-examples}

We now turn to some applications of the $\diffp$-private version of Le Cam's
inequality from Proposition~\ref{PropPrivateLeCam}.  First, we study the
problem of one-dimensional mean estimation.  In addition to demonstrating
how the minimax rate changes as a function of $\diffp$, we also reveal some
interesting (and perhaps disturbing) effects of enforcing $\diffp$-local
differential privacy: the effective sample size may be even polynomially
smaller than $\diffp^2 \numobs$.  Our second example studies median
estimation, which---as a more robust quantity than the mean---allows us to
always achieve parametric convergence rates with an effective sample size
reduction of $\numobs$ to $\diffp^2 \numobs$.  Our third example
investigates conditional probability estimation, which exhibits a more
nuanced dependence on privacy than the preceding estimates.  We state each
of our bounds assuming $\diffp \in [0, 1]$; the bounds hold (with different
numerical constants) whenever $\diffp \in [0, C]$ for some universal
\mbox{constant $C$.}


\subsubsection{One-dimensional mean estimation}
\label{sec:location-family}

Given a real number $k > 1$, consider the family
\begin{align*}
  \mc{P}_k \defeq \big \{ \mbox{distributions $P$ such that} ~
  \Exs_P[X] \in [-1, 1] ~ \mbox{and} ~ \Exs_P[|X|^k] \leq 1 \big \}.
\end{align*}
Note that the parameter $k$ controls the tail behavior of the random variable
$X$, with larger values of $k$ imposing more severe constraints.  As $k
\uparrow +\infty$, the random variable converges to one that is supported
almost surely on the interval $[-1, 1]$.  Suppose that our goal is to
estimate the mean $\theta(P) = \Exs_P[X]$, and that we adopt the squared
error to measure the quality of an estimator.  The classical
minimax risk for this problem scales as $\numobs^{-\min\{1, 2 -
  \frac{2}{k}\}}$ for all values of $k \ge 1$.  Our goal here is to
understand how the $\diffp$-private minimax risk~\eqref{EqnPrivateMinimax},
\begin{align*}
  \minimax_n(\theta(\mc{P}_k), (\cdot)^2, \diffp) \defeq
  \inf_{\channelprob \in \QFAMA} \inf_{\what{\optvar}} \sup_{\statprob
    \in \mc{P}_k} \Exs \left[ \big(\what{\optvar}(\channelrv_1,
    \ldots, \channelrv_n) - \theta(P) \big)^2\right],
\end{align*}
differs from the classical minimax risk.

\begin{corollary}
  \label{CorLocFamily}
  There exist universal constants $0 < c_\ell \leq c_u \le 9$ such
  that for all $k > 1$ and $\diffp \in [0,1]$, the $\diffp$-minimax
  risk $\minimax_n(\theta(\mc{\statprob}_k, (\cdot)^2, \diffp)$ is
  sandwiched as
  \begin{align}
    \label{eqn:location-family-bound}
    c_\ell \min\bigg\{1, \left(n\diffp^2\right)^{-\frac{k -
        1}{k}}\bigg\} \; \le \; \minimax_n(
    \theta(\mc{P}_k), (\cdot)^2, \diffp) \; \le
    \; c_u \min\left\{1, \left(n \diffp^2\right)^{-\frac{k -
        1}{k}}\right\}.
  \end{align}
\end{corollary}
\noindent We prove the lower bound using the $\diffp$-private
version~\eqref{eqn:le-cam-private} of Le Cam's inequality from
Proposition~\ref{PropPrivateLeCam}; see
Appendix~\ref{sec:proof-location-family} for the details.

In order to understand the bound~\eqref{eqn:location-family-bound}, it
is worthwhile considering some special cases, beginning with the usual
setting of random variables with finite variance ($k = 2$).  In the
non-private setting in which the original sample $(X_1, \ldots,
X_\numobs)$ is observed, the sample mean $\thetahat =
\frac{1}{\numobs} \sum_{i=1}^\numobs X_i$ has mean-squared error at
most $1/n$. When we require $\diffp$-local differential privacy,
Corollary~\ref{CorLocFamily} shows that the minimax rate worsens to $1
/ \sqrt{n \diffp^2}$. More generally, for any $k > 1$, the minimax
rate scales as $\minimax_n(\optvar(\mc{P}_k), (\cdot)^2, \diffp)
\asymp (n \diffp^2)^{-\frac{k-1}{k}}$.
As $k \uparrow \infty$, the moment condition $\Exs[|X|^k] \leq 1$
becomes equivalent to the boundedness constraint $|\statrv| \leq 1$
a.s., and we obtain the more standard parametric rate $(n
\diffp^2)^{-1}$, where there is no reduction in the exponent.

The upper bound is achieved by a variety of privacy mechanisms and
estimators. One of them is the following variant of the Laplace mechanism:
\begin{itemize}
\item Letting $\truncate{x}{T} = \max\{-T, \min\{x, T\}\}$ denote the
  projection of $x$ to the interval $[-T, T]$, output the private
  samples
  \begin{equation}
    \label{eqn:truncated-mean-estimator}
    \channelrv_i = \truncate{\statrv_i}{T} + W_i,
    ~~ \mbox{where} ~~ W_i \simiid \laplace(\diffp / (2T))
    ~~ \mbox{and} ~ T = (n \diffp^2)^{\frac{1}{2k}}.
  \end{equation}
\item Compute the sample mean $\what{\theta}_n \defeq \frac{1}{n}
  \sum_{i=1}^n \channelrv_i$ of the privatized data.
\end{itemize}
We present the analysis of this estimator in
Appendix~\ref{sec:proof-location-family}.  It is one case in which
the widely-used Laplacian mechanism is an optimal mechanism; later examples
show that this is \emph{not} always the case.\\

Summarizing our results thus far, Corollary~\ref{CorLocFamily} helps to
demarcate situations in which local differential privacy may or may
not be acceptable for location estimation problems.  In particular,
for bounded domains---where we may take $k \uparrow \infty$---local
differential privacy may be quite reasonable. However, in situations
in which the sample takes values in an unbounded space, local
differential privacy imposes more severe constraints.


\subsubsection{One-dimensional median estimation}
\label{sec:median-estimation}

Instead of attempting to privately estimate the mean---an inherently
non-robust quantity---we may also consider median estimation problems.
Median estimation for general distributions is impossible even in
non-private settings,\footnote{That is, the minimax error never converges to
  zero: consider estimating the median of the two distributions $P_0$ and
  $P_1$, each supported on $\{-\radius, \radius\}$, where $P_0(\radius) =
  \frac{1 + \delta}{2}$ and $P_1(\radius) = \frac{1 - \delta}{2}$, then take
  $\delta \downarrow 0$ as the sample size increases.} so we focus on
the median as an $M$-estimator.  Recalling that the minimizer(s) of
$\E[|X - \theta|]$ are the median(s) of $X$, we consider the gap between
mean absolute error of our estimator and that of the true median,
\begin{equation*}
  \E[\risk(\what{\theta})] - \inf_{\theta \in \R} \risk(\theta),
  ~~ \mbox{where} ~~
  \risk(\theta) \defeq \E[|X - \theta|].
\end{equation*}

We first give a proposition characterizing the minimax rate for this
problem by applying Proposition~\ref{PropPrivateLeCam}. Let
$\theta(P) = \median(P)$ denote the median of the distribution $P$,
and for radii $\radius > 0$, we consider the family of distributions
supported on $\R$ defined by
\begin{equation*}
  \mc{P}_\radius \defeq \left\{\mbox{distributions}~P~\mbox{such that}~
  |\median(P)| \le \radius, ~ \E_P[|X|] < \infty \right\}.
\end{equation*}
In this case, we consider the slight variant of the
minimax rate~\eqref{EqnPrivateMinimax} defined by the risk gap
\begin{equation*}
  \minimax_n(\theta(\mc{P}_\radius), \risk, \diffp)
  \defeq \inf_{\channel \in \QFAMA} \inf_{\what{\theta}}
  \sup_{P \in \mc{P}_\radius} \E_{P,Q} \left[
    \risk(\what{\theta}(\channelrv_1, \ldots, \channelrv_n))
    - \risk(\theta(P))\right].
\end{equation*}
We then have the following.
\begin{corollary}
  \label{CorMedEst}
  For the median estimation problem, there are universal constants
  $\frac{1}{20} \le c_\ell \le c_u < 6$ such that for all $\radius \ge
  0$ and $\diffp \in [0, 1]$, the minimax error satisfies
  \begin{equation*}
    c_\ell \cdot \radius \min\left\{1, (n \diffp^2)^{-\half} \right\}
    \le \minimax_n(\theta(\mc{P}_\radius), \risk, \diffp)
    \le c_u \cdot \radius \min\left\{1, (n \diffp^2)^{-\half} \right\}.
  \end{equation*}
\end{corollary}
\noindent
We present the proof of the lower bound in Corollary~\ref{CorMedEst}
to Section~\ref{sec:proof-median-estimation}, focusing our attention
here on a minimax optimal sequential procedure based on stochastic
gradient descent (SGD).

To describe our SGD procedure, let $\theta_0 \in [-\radius, \radius]$ be
arbitrary, and $W_i$ be an i.i.d.\ $\{-1, +1\}$ Bernoulli sequence with
$\P(W_i = 1) = \frac{e^\diffp}{e^\diffp + 1}$, and let $\statrv_i \simiid P$
be the observations of the distribution $P$ whose median we wish to estimate
(and which must be made private). We iterate according to the projected
stochastic gradient descent procedure
\begin{equation}
  \label{eqn:median-sgd}
  \theta_{i + 1} = \truncate{\theta_i - \stepsize_i \channelrv_i}{\radius},
  ~~ \mbox{where} ~~
  \channelrv_i = \frac{e^\diffp + 1}{e^\diffp - 1}
  \cdot W_i \cdot \sign(\theta_i - \statrv_i),
\end{equation}
where as in expression~\eqref{eqn:truncated-mean-estimator},
$\truncate{x}{\radius} = \max\{-\radius, \min\{x, \radius\}\}$ is the
projection onto the set $[-\radius, \radius]$, and the sequence $\stepsize_i
> 0$ are non-increasing stepsizes.  By inspection we see that $\channelrv_i$
is differentially private for $\statrv_i$, and we have the conditional
unbiasedness $\E[\channelrv_i \mid \statrv_i, \theta_i] = \sign(\theta_i -
\statrv_i) \in \partial_\theta |\theta_i - \statrv_i|$, where $\partial$
denotes the subdifferential operator. Standard results on stochastic
gradient descent methods~\cite{NemirovskiJuLaSh09} imply that for
$\what{\theta}_n = \frac{1}{n} \sum_{i=1}^n \theta_i$, we have
\begin{equation*}
  \E[\risk(\what{\theta}_n)] - \inf_{\theta \in [-\radius, \radius]}
  \risk(\theta)
  \le \frac{1}{n} \left[\frac{\radius^2}{\stepsize_n}
    + \half \sum_{i=1}^n \stepsize_i \left(\frac{e^\diffp + 1}{e^\diffp - 1}
    \right)^2 \right].
\end{equation*}
Under the assumption that $\diffp \le 1$, we take
$\stepsize_i = \diffp \cdot \radius / \sqrt{i}$, which immediately
implies the upper bound $\E[\risk(\what{\theta}_n)] - \risk(\median(P))
\le \frac{6 \radius}{\sqrt{n \diffp^2}}$.

We make two remarks on the procedure~\eqref{eqn:median-sgd}. First, it is
essentially a sequential variant of Warner's 1965 randomized
response~\cite{Warner65}, a procedure whose variants turn out to often be optimal, as
we show in the sequel. Secondly, while at first blush it is not clear that
the additional complexity of stochastic gradient descent is warranted, we
provide experiments comparing the SGD procedure with more naive estimators
in Section~\ref{sec:median-experiment} on a salary estimation task. These
experiments corroborate the improved performance of our
minimax optimal strategy.



\subsection{Pairwise upper bounds on Kullback-Leibler divergences}
\label{sec:pairwise-kl-bounds}

As mentioned previously, the private form of Le Cam's bound
(Proposition~\ref{PropPrivateLeCam}) is a corollary of more general
results on the contractive effects of privacy on pairs of
distributions, which we now state.  Given a pair of distributions
$(\statprob_1, \statprob_2)$ defined on a common space $\statdomain$,
any conditional distribution $\channelprob$ transforms such a pair of
distributions into a new pair $(\marginprob_1, \marginprob_2)$ via the
marginalization operation $\marginprob_j(S) = \int_{\statdomain}
\channelprob(S \mid x) d \statprob_j(x)$ for $j = 1, 2$.  Intuitively,
when the conditional distribution $\channelprob$ is
$\diffp$-locally differentially private, the two output
distributions $(\marginprob_1, \marginprob_2)$ should be closer
together.  The following theorem makes this intuition precise:
\begin{theorem}
  \label{theorem:master}
  For any $\diffp \geq 0$, let $\channelprob$ be a conditional
  distribution that guarantees $\diffp$-differential privacy.  Then,
  for any pair of distributions $\statprob_1$ and $\statprob_2$, the
  induced marginals $\marginprob_1$ and $\marginprob_2$ satisfy the
  bound
\begin{align}
\label{eqn:diffp-to-tv-bound}
\dkl{\marginprob_1}{\marginprob_2} +
\dkl{\marginprob_2}{\marginprob_1} \le \min\{4, e^{2\diffp}\}
(e^\diffp - 1)^2 \tvnorm{\statprob_1 - \statprob_2}^2.
\end{align}
\end{theorem} 

\paragraphc{Remarks} Theorem~\ref{theorem:master} is a type of
\emph{strong data processing} inequality~\cite{AnantharamGoKaNa13},
providing a quantitative relationship between the divergence
$\tvnorm{\statprob_1 - \statprob_2}$ and the KL-divergence
$\dkl{\marginprob_1}{\marginprob_2}$ that arises after applying the
channel $\channel$.  The result of Theorem~\ref{theorem:master} is
similar to a result due to \citet*[Lemma III.2]{DworkRoVa10}, who show
that $\dkl{\channel(\cdot \mid \statsample)}{\channel(\cdot \mid
  \statsample')} \le \diffp(e^\diffp - 1)$ for any $\statsample,
\statsample' \in \statdomain$, which implies
$\dkl{\marginprob_1}{\marginprob_2} \le \diffp(e^\diffp - 1)$ by
convexity.  This upper bound is weaker than
Theorem~\ref{theorem:master} since it lacks the term $
\tvnorm{\statprob_1 - \statprob_2}^2$.  This total variation term is
essential to our minimax lower bounds: more than providing a bound on
KL divergence, Theorem~\ref{theorem:master} shows that differential
privacy acts as a contraction on the space of probability measures.
This contractivity holds in a strong sense: indeed, the
bound~\eqref{eqn:diffp-to-tv-bound} shows that even if we start with a
pair of distributions $\statprob_1$ and $\statprob_2$ whose KL
divergence is infinite, the induced marginals $\marginprob_1$ and
$\marginprob_2$ always have finite KL divergence.  We provide the
proof of Theorem~\ref{theorem:master} in
Appendix~\ref{SecProofTheoremOne}.  \\

Let us now develop a corollary of Theorem~\ref{theorem:master} that
has a number of useful consequences, among them the private form of Le
Cam's method from Proposition~\ref{PropPrivateLeCam}.  Suppose that we
are given an indexed family of distributions $\{ \statprob_\packval,
\packval \in \packset \}$.  Let $\packrv$ denote a random variable
that is uniformly distributed over the finite index set $\packset$.
Conditionally on $\packrv = \packval$, suppose we sample a random
vector $(X_1, \ldots, X_\numobs)$ according to a product measure of
the form $\statprob_\packval(x_1, \ldots, x_\numobs) \defeq
\prod_{i=1}^\numobs \statprob_{\HACKI{\packval}}(x_i)$, where
$(\DOUBLEHACK{\packval}{1}, \ldots, \DOUBLEHACK{\packval}{\numobs})$
denotes some sequence of indices.  Now suppose that we draw an
$\diffp$-locally private sample $(\channelrv_1, \ldots, \channelrv_n)$
according to the channel $\channel(\cdot \mid \statrv_{1:n})$.
Conditioned on $\packrv = \packval$, the private sample is distributed
according to the measure $\marginprob_\packval^\numobs$ given by
\begin{align}
  \label{eqn:marginal-channel}
  \marginprob_\packval^n(S) \defeq \int \channelprob^n(S \mid
  \statsample_1, \ldots, \statsample_n) d
  \statprob_\packval(\statsample_1, \ldots, \statsample_n) ~~~
  \mbox{for}~ S \in \sigma(\channeldomain^n).
\end{align}
Since we allow interactive protocols, the distribution
$\marginprob_\packval^\numobs$ need not be a product distribution in
general.  Nonetheless, in this setup we have the following
tensorization inequality:
\begin{corollary}
  \label{corollary:idiot-fano}
  For any $\diffp$-locally differentially
  private conditional distribution~\eqref{eqn:local-privacy}
  $\channelprob$ and any paired sequences of distributions
  $\{\statprob_{\HACKI{\packval}}\}$ and
  $\{\statprob_{\HACKI{\altpackval}}\}$, we have
  \begin{align}
    \label{EqnWeakBound}
    \dkl{\marginprob_\packval^n}{\marginprob_{\altpackval}^n}
    & \leq 4
    (e^\diffp - 1)^2 \sum_{i = 1}^n \tvnorm{\statprob_{\HACKI{\packval}} -
      \statprob_{\HACKI{\altpackval}}}^2.
  \end{align}
\end{corollary}
\noindent
See Appendix~\ref{sec:proof-idiot-fano} for the proof, which requires
a few intermediate steps to obtain the additive inequality.  One
consequence of Corollary~\ref{corollary:idiot-fano} is the private
form of Le Cam's bound in Proposition~\ref{PropPrivateLeCam}. Given
the index set $\packset = \{1, 2 \}$, consider two paired sequences of
distributions of the form $\{\statprob_1, \ldots, \statprob_1\}$ and
$\{\statprob_2, \ldots, \statprob_2\}$.  With this choice, we have
\begin{align*}
  \tvnorm{\marginprob_1^\numobs - \marginprob_2^\numobs}^2 &
  \stackrel{(i)}{\leq} \half
  \dkl{\marginprob_1^n}{\marginprob_{2}^n}
  \;
  \stackrel{(ii)}{\leq} \; 2 (e^\diffp - 1)^2 \numobs
  \tvnorm{\statprob_{1} - \statprob_2},
\end{align*}
where step (\emph{i}) is Pinsker's inequality, and step (\emph{ii}) follows from
the tensorization inequality~\eqref{EqnWeakBound} and the i.i.d.\ nature of
the product distributions $\statprob_1^\numobs$ and $\statprob_2^\numobs$.
Noting that $(e^\diffp - 1)^2 \le 2 \diffp^2$ for $\diffp \le \frac{23}{35}$
and applying the classical Le Cam bound~\eqref{EqnLeCam} gives
Proposition~\ref{PropPrivateLeCam}.

In addition, inequality~\eqref{EqnWeakBound} can be used to derive a
bound on the mutual information.  Bounds of this type are useful in
applications of Fano's method, to be discussed at more length in the
following section. In particular, if we define the mixture
distribution $\meanmarginprob^n = \frac{1}{|\packset|} \sum_{\packval
  \in \packset} \marginprob^n_\packval$, then by the definition of
mutual information, we have
\begin{align}
  \information(\channelrv_1, \ldots, \channelrv_n; \packrv) \; = \;
  \frac{1}{|\packset|} \sum_{\packval \in \packset}
  \dkl{\marginprob_\packval^n}{ \meanmarginprob^n} & \le
  \frac{1}{|\packset|^2} \sum_{\packval, \altpackval}
  \dkl{\marginprob_\packval^n}{ \marginprob_{\altpackval}^n} \nonumber
  \\
  \label{eqn:idiot-fano}
  & \le 4(e^\diffp - 1)^2 \sum_{i=1}^n \frac{1}{|\packset|^2}
  \sum_{\packval, \altpackval \in \packset} \tvnorm{\statprob_{\HACKI{\packval}}
    - \statprob_{\HACKI{\altpackval}}}^2,
\end{align}
the first inequality following from the joint convexity of the KL
divergence and the final inequality from
Corollary~\ref{corollary:idiot-fano}.

\paragraphc{Remarks} Mutual information bounds under local privacy
have appeared previously. \citet{McGregorMiPiReTaVa10} study
relationships between communication complexity and differential
privacy, showing that differentially private schemes allow low
communication.  They provide a
result~\cite[Prop.~7]{McGregorMiPiReTaVa10} guaranteeing
$\information(\statrv_{1:n}; \channelrv_{1:n}) \le 3 \diffp n$; they
strengthen this bound to $\information(\statrv_{1:n};
\channelrv_{1:n}) \le (3/2)\diffp^2 n$ when the $\statrv_i$ are
i.i.d.\ uniform Bernoulli variables.  Since the total variation
distance is at most $1$, our result also implies this scaling (for
arbitrary $\statrv_i$); however, our result is stronger since it
involves the total variation terms $\tvnorms{\statprob_{\packval(i)} -
  \statprob_{\altpackval(i)}}$.  These TV terms are an essential part
of obtaining the sharp minimax results that are our focus. In
addition, Corollary~\ref{corollary:idiot-fano} allows for \emph{any}
sequentially interactive channel $\channelprob$; each random variable
$\channelrv_i$ may depend on the private answers $\channelrv_{1:i-1}$
of other data providers.



\section{Bounds on private mutual information: Fano's method}
\label{sec:fano}

We now turn to a set of techniques for bounding the private
minimax risk~\eqref{EqnPrivateMinimax} based on
Fano's inequality from information theory.  We begin by describing how
Fano's inequality is used in classical minimax theory, then
presenting some of its extensions to the private setting.

Recall that our goal is to lower bound the minimax risk associated
with estimating some parameter $\theta(\statprob)$ in a given metric
$\metric$.  Given a finite set $\packset$, a family of
distributions $\{ \statprob_\packval, \packval \in \packset \}$ is
said to be $2 \delta$-separated in the metric $\metric$ if
$\metric(\theta(\statprob_\packval), \theta(\statprob_\altpackval))
\geq 2 \delta$ for all distinct pairs $\packval, \altpackval \in
\packset$.  Given any such $2 \delta$-separated set, the classical
form of Fano's inequality~\cite[cf.][]{Yu97} asserts that the minimax
risk~\eqref{EqnClassicalMinimax} has lower bound
\begin{align*}
  \minimax_n(\theta(\mc{P}), \Phi \circ \metric) & \geq
  \frac{\Phi(\delta)}{2} \left \{ 1 - \frac{ I(\packrv; X_1^\numobs) +
    \log 2}{\log |\packset|} \right \}.
\end{align*}
Here $I(\packrv; X_1^n)$ denotes the mutual information between a random
variable $\packrv$ uniformly distributed over the set $\packset$ and
a random vector $X_1^n = (X_1, \ldots, X_\numobs)$ drawn from
the mixture distribution 
\begin{align}
  \label{EqnClassicalMixture}
  \meanstatprob \defeq \frac{1}{|\packset|} \sum_{\packval \in \packset}
  \statprob^\numobs_\packval,
\end{align}
so that $I(\packrv; X_1^n) = \frac{1}{|\packset|} \sum_{\packval}
\dkl{\statprob^n_\packval}{\meanstatprob}$; equivalently, the random
variables are drawn $\statrv_i \simiid \statprob_\packval$ conditional on
$\packrv = \packval$.  In the cases we consider, it is sometimes convenient
to use a slight generalization of the classical Fano method by
extending the
$2\delta$-separation above.  Let $\metric_\packset$ be a metric on the set
$\packset$, and for $t \ge 0$ define the \emph{neighborhood size} for the
set $\packset$ by
\begin{equation}
  \label{eqn:neighborhood-size}
  \hoodbig_t \defeq \max_{\packval \in \packset}
  \card\left\{\altpackval \in \packset \mid \metric_\packset(\packval,
  \altpackval) \le t\right\}
\end{equation}
and the separation function
\begin{equation}
  \label{eqn:separation}
  \delta(t) \defeq \half \min \left\{
  \metric(\theta(\statprob_\packval), \theta(\statprob_\altpackval))
  \mid \packval, \altpackval \in \packset
  ~ \mbox{and}~ \metric_\packset(\packval, \altpackval) > t
  \right\}.
\end{equation}
Then we have the following generalization~\cite[Corollary 1]{DuchiWa13} of
the Fano bound: for any $t \ge 0$,
\begin{equation}
  \label{eqn:user-friendly-fano}
  \minimax_n(\theta(\mc{P}), \bigloss \circ \metric)
  \ge \frac{\bigloss(\delta(t))}{2}
  \left\{1 - \frac{I(\packrv; X_1^\numobs) + \log 2}{
    \log |\packset| - \log \hoodbig_t}\right\}.
\end{equation}



\subsection{A private version of Fano's method}

We now turn to developing a version of Fano's lower bound that applies to
estimators $\thetahat$ that act on privatized samples $Z_1^\numobs = (Z_1,
\ldots, Z_\numobs)$, where the obfuscation channel $\channel$ is
non-interactive (Figure~\ref{fig:interactive-channel}(b)), meaning that
$\channelrv_i$ is conditionally independent of $\channelrv_{\setminus i}$
given $\statrv_i$. Our upper bound is variational: it involves optimization
over a subset of the space $L^\infty(\statdomain) \defn \big \{ f :
\statdomain \rightarrow \R \, \mid \, \|f\|_\infty < \infty \big \}$ of
uniformly bounded functions equipped with the supremum norm $\linf{f} =
\sup_x |f(x)|$ and the associated $1$-ball of the
supremum norm
\begin{align}
  \label{eqn:L-infty-set}
  \linfset(\statdomain) \defeq \left\{ \optdens \in
  L^\infty(\statdomain) \mid \|\optdens\|_\infty \leq 1 \right\}.
\end{align}
As the set $\statdomain$ is generally clear from context, we
typically omit this dependence (and adopt the shorthand $\linfset$).
As with the classical Fano method, we consider a $2 \delta$-separated
family of distributions $\{ \statprob_\packval, \packval \in \packset
\}$, and for each $\packval \in \packset$, we define the linear
functional $\varphi_\packval : L^\infty(\statdomain) \rightarrow \R$
by
\begin{align}
  \label{EqnLinearFunctional}
  \varphi_\packval(\optdens) = \int_{\statdomain} \optdens(x)
  (d\statprob_\packval(x) - d \meanstatprob(x)).
\end{align}
With this notation, we have the following private version of Fano's
method:
\begin{proposition}[Private Fano method]
  \label{PropPrivateFano}
  Given any set $\{\statprob_\packval, \packval \in \packset \}$, for any $t
  \ge 0$ the non-interactive $\diffp$-private minimax risk has lower bound
  \begin{align*}
    \minimax_n(\theta(\mc{P}), \Phi \circ \metric, \diffp) & \geq
    \frac{\Phi(\delta(t))}{2} \; \left \{ 1 - \frac{n(e^\diffp - 1)^2}{
      \log (|\packset| / \hoodbig_t)}
    \bigg[\frac{1}{|\packset|} \sup_{\optdens \in \linfset}
      \sum_{\packval \in \packset}
      \left(\varphi_\packval(\optdens)\right)^2\bigg]
    - \frac{\log 2}{\log
      (|\packset|/ \hoodbig_t)} \right \}.
  \end{align*}
\end{proposition}

Underlying Proposition~\ref{PropPrivateFano} is a variational bound on the
mutual information between a sequence $\channelrv_1^\numobs = (\channelrv_1,
\ldots, \channelrv_\numobs)$ of private random variables and a random index
$\packrv$ drawn uniformly on $\packset$, where 
$\channelrv_1^\numobs \sim \marginprob_\packval^\numobs$, conditional on
$\packrv = \packval$; that is, $\channelrv_1^n$ is marginally drawn
according to the mixture
\begin{align*}
  \meanmarginprob & \defeq \frac{1}{|\packset|} \sum_{\packval \in
    \packset} \marginprob_\packval^\numobs
  ~~ \mbox{where} ~
  \marginprob^\numobs_\packval(S) = \int \channel(S \mid \statval_1^n)
  d\statprob^\numobs_\packval(\statval_1^n).
\end{align*}
(Recall equation~\eqref{eqn:marginal-channel}).  When the conditional
distribution $Q$ is non-interactive, as considered in this section,
then $\marginprob^\numobs_\packval$ is also a product distribution.
By comparison with equation~\eqref{EqnClassicalMixture}, we see that
$\meanmarginprob$ is the private analogue of the mixture distribution
$\meanstatprob$ that arises in the classical Fano analysis.

Proposition~\ref{PropPrivateFano} is an immediate consequence of the Fano
bound~\eqref{eqn:user-friendly-fano} coupled with the following upper bound
on the mutual information between $\channelrv_1^\numobs$ and an index
$\packrv$ uniformly distributed over $\packset$:
\begin{align}
  \label{EqnSuperFano}
  \information(\packrv; \channelrv_1^\numobs) & \leq \numobs (e^\diffp -
  1)^2 \frac{1}{|\packset|} \sup_{\optdens \in \linfset} \sum_{\packval
    \in \packset} \left(\varphi_\packval(\optdens)\right)^2.
\end{align}
The inequality~\eqref{EqnSuperFano} is in turn an immediate consequence of
Theorem~\ref{theorem:super-master} to come; we provide the proof of this
inequality in Appendix~\ref{sec:proof-super-fano}.  We conjecture that it also
holds in the fully interactive setting, but given well-known difficulties of
characterizing multiple channel capacities with feedback~\cite[Chapter
  15]{CoverTh06}, it may be challenging to verify this conjecture.


\subsection{Some applications of the private Fano bound}
\label{sec:mean-estimation}

In this section, we show how Proposition~\ref{PropPrivateFano} leads to
sharp characterizations of the $\diffp$-private minimax rates for some
classical and high-dimensional mean estimation problems.  We consider
estimation of the $d$-dimensional mean $\theta(\statprob) \defeq
\E_\statprob[\statrv]$ of a random vector $\statrv \in \R^d$.  Due to the
difficulties associated with differential privacy on non-compact spaces
(recall Section~\ref{sec:location-family}), we focus on distributions with
compact support. We provide proofs of our mean estimation results
in Appendix~\ref{sec:proofs-big-mean-estimation}.

\subsubsection{Classical mean estimation in $d$ dimensions}

We begin by considering estimation of means for sampling distributions
supported on $\ell_p$ balls, where $p \in [1, 2]$. Indeed,
for a radius $\radius < \infty$, consider the family
\begin{align*}
  \PRFAM & \defeq \big \{ \mbox{distributions~} \statprob
  ~ \mbox{supported on}~ \ball_p(\radius) \subset \R^d \big\},
\end{align*}
where $\Ball_p(\radius) = \{ x \in \real^d \mid \norm{x}_p \le \radius
\}$ is the $\ell_p$-ball of radius $\radius$.  In the non-private
setting, the standard estimator $\what{\optvar} = \frac{1}{n}
\sum_{i=1}^n \statrv_i$ has mean-squared error at most $\radius^2 /
n$, since $\norm{\statrv}_2 \le \norm{\statrv}_p \le \radius$ by
assumption.  The following result shows that the private minimax MSE
is substantially different:
\begin{corollary}
  \label{CorDmean}
  For the mean estimation problem, for all $p \in [1, 2]$ and privacy
  levels $\diffp \in [0, 1]$, the non-interactive
  private minimax risk satisfies
  \begin{align*}
    \radius^2 \min\left\{1,
    \max\left\{\frac{1}{\sqrt{n \diffp^2}}
    \wedge \frac{d}{n \diffp^2},
    \frac{1}{(n \diffp^2)^\frac{2 - p}{p}} \wedge
    \frac{d^{2 \frac{p - 1}{p}}}{n \diffp^2}\right\}
    \right\}
    \lesssim \minimax_n(\optvar(\PRFAM),
    \ltwo{\cdot}^2, \diffp) \lesssim \radius^2 \min\left\{\frac{d}{n
      \diffp^2}, 1 \right\}.
  \end{align*}
\end{corollary}
\noindent See Appendix~\ref{sec:proof-d-dimensional-mean} for the
proof of this claim; the lower bound makes use of the private form of
Fano's method (Proposition~\ref{PropPrivateFano}), while the upper
bound is a consequence of the optimal mechanisms we develop
in Section~\ref{sec:attainability-means}.

Corollary~\ref{CorDmean} demonstrates the substantial difference
between $d$-dimensional mean estimation in private and non-private
settings: the privacy constraint leads to a
multiplicative penalty of $d/\diffp^2$ in terms of mean-squared error.
Thus, the effect of privacy is to reduce the effective sample size from
$n$ to $\diffp^2 \numobs / \usedim$. We remark in passing that
if $\diffp \ge 1$, our result still holds, though we replace
the quantity $\diffp$ in the lower bound with the quantity $e^\diffp - 1$
and in the upper bound with $1 - e^{-\diffp}$.
The lower bound as written is somewhat complex in its dependence on
$p \in [1, 2]$, so an investigation of the extreme cases is somewhat helpful.
Taking $p = 2$, the scaling in the lower bound simplifies to
$\min\{1, \frac{d}{n \diffp^2}\}$, identical to the upper bound;
in the case $p = 1$, it becomes $\min\{1, \frac{1}{\sqrt{n \diffp^2}},
\frac{d}{n \diffp^2}\}$. There is a gap in the regime
$d \ge \sqrt{n \diffp^2}$ in this case, though the asymptotic
regime for large $n$ shows that both the lower and upper bounds become
$d / (n \diffp^2)$, independent of $p \in [1, 2]$.


\subsubsection{Estimation of high-dimensional sparse vectors}

Recently, there has been substantial interest in high-dimensional
problems in which the dimension $d$ is larger than the sample size
$n$, but a low-dimensional latent structure makes
inference possible~\citep{BuhlmannGe11, NegahbanRaWaYu12}.  Here we consider a simple but canonical instance
of a high-dimensional problem, that of estimating a sparse mean
vector.  For an integer parameter $s \geq 1$, consider the class of
distributions
\begin{align}
  \mc{\statprob}_{\infty, \radius}^s
  \defeq \left\{\mbox{distributions}~ \statprob ~ \mbox{supported~on~}
  \ball_\infty(\radius) \subset \R^d ~ \mbox{with~}
  \norm{\E_\statprob[\statrv]}_0 \le s \right\}.
\end{align}
In the non-private case, estimation of such an $s$-sparse predictor in the
squared $\ell_2$-norm is possible at rate $\E[\ltwos{\what{\optvar} -
    \optvar}^2] \lesssim \radius^2 \frac{s \log (d/s)}{n}$, so that the
dimension $d$ can be exponentially larger than the sample size $n$. With
this context, the next result shows that local privacy can have a dramatic
impact in the high-dimensional setting.  For simplicity, we restrict
ourselves to the easiest case of a $1$-sparse vector ($s = 1$).


\begin{corollary}
  \label{CorSparseMean}
  For the $1$-sparse means problem, for all $\diffp \ge 0$, the
  non-interactive private minimax risk satisfies
  \begin{align*}
    \min\left\{\radius^2, \frac{\radius^2 d \log(2d)}{ n (e^\diffp - 1)^2}
    \right\}
    \; \lesssim \; \minimax_n\left(\optvar(\mc{\statprob}_{\infty,
      \radius}^1), \ltwo{\cdot}^2, \diffp\right) \; \lesssim \;
    \min\left\{\radius^2, \frac{\radius^2 d \log(2d)}{ n
      (1 - e^{-\diffp})^2}\right\}.
  \end{align*}
\end{corollary}
\noindent
See Appendix~\ref{sec:proof-mean-high-dim} for a proof. 

From the lower bound in Corollary~\ref{CorSparseMean}, we see that local
differential privacy has an especially dramatic effect for the sparse means
problem: due to the presence of the $\usedim$-term in the numerator,
estimation in high-dimensional settings ($\usedim \ge \numobs$) becomes
impossible, even for $1$-sparse vectors.  Contrasting this fact with the
scaling $\usedim \asymp e^{\numobs}$ that $1$-sparsity allows in the
non-private setting shows that local differential privacy is
a very severe constraint in this setting. We note in passing that
an essentially identical argument to that we provide
in Appendix~\ref{sec:proof-mean-high-dim} gives a lower bound
of $\radius \sqrt{\frac{d \log(2d)}{n (e^\diffp - 1)^2}}$
on estimation with $\linf{\cdot}$ error.
Corollary~\ref{CorSparseMean}
raises the question of whether high-dimensional estimation is possible
with local differential privacy. In non-interactive settings,
our result shows that there is a dimension-dependent penalty that must
be paid for estimation; in scenarios in which it is possible to
modify the privatizing mechanism $\channelprob$, it may be possible to
``localize'' in an appropriate sense once important variables
have been identified, providing some recourse against the negative
results of Corollary~\ref{CorSparseMean}. We leave such considerations
to future work.


\subsubsection{Optimal mechanisms: attainability for mean estimation}
\label{sec:attainability-means}

Our lower bounds for both $d$-dimensional mean estimation
(Corollary~\ref{CorDmean}) and $1$-sparse mean estimation
(Corollary~\ref{CorSparseMean}) are based on the private form of Fano's
method (Proposition~\ref{PropPrivateFano}).  On the other hand, the upper
bounds are based on direct analysis of specific privacy mechanisms and
estimators.  Here we discuss the optimal privacy mechanisms
for these two problems in more detail.

\paragraph{Sub-optimality of Laplacian mechanism:} For 
the $1$-dimensional mean estimation problem
(Corollary~\ref{CorLocFamily}), we showed that adding Laplacian noise
to (truncated versions of) the observations led to an optimal privacy
mechanism. The extension of this result to the $d$-dimensional
problems considered in Corollary~\ref{CorDmean}, however,
\emph{fails}.  More concretely, as a special case of the families in
Corollary~\ref{CorDmean}, consider the class $\mc{\statprob}_{2,1}$ of
distributions supported on the Euclidean ball $\ball_2(1) \subset
\R^d$ of unit norm.  In order to guarantee $\diffp$-differential
privacy, suppose that we output the additively corrupted random vector
$\channelrv \defeq \statsample + W$, where the noise vector $W \in
\R^d$ has i.i.d\ components following a $\laplace(\diffp / \sqrt{d})$
distribution.  With this choice, it can be verified that for $X$
taking values in $\Ball_2(1)$, the random vector $Z$ is an $\diffp$-DLP
view of $X$.  However, this privacy mechanism does \emph{not} achieve
the minimax risk over $\diffp$-private mechanisms.  In particular,
one must suffer the rate
\begin{equation}
  \label{eqn:mean-laplace-sucks}
  \inf_{\what{\optvar}} \sup_{\statprob \in \mc{\statprob}}
  \E_\statprob\left[ \ltwos{\what{\optvar}(\channelrv_1, \ldots,
      \channelrv_n) - \E_\statprob[\statrv]}^2\right] \gtrsim
  \min\left\{\frac{d^2}{n \diffp^2}, 1\right\},
\end{equation}
a quadratic ($\usedim^2$) dimension dependence, as
opposed to the linear scaling ($\usedim$) of the optimal result in
Corollary~\ref{CorDmean}. See Appendix~\ref{sec:mean-laplace-sucks}
for the proof of claim~\eqref{eqn:mean-laplace-sucks}. The poorer
dimension dependence of the Laplacian mechanism demonstrates that
sampling mechanisms must be chosen carefully.

\paragraph{Optimal mechanisms:}

Let us now describe some mechanisms that are optimal for the
$d$-dimensional and $1$-sparse mean estimation problems.  Both of them
are procedures that output a random variable $\channelrv$ that is an
$\diffp$-differentially-private view of $X$, and they are unbiased in the
sense that $\Exs[\channelrv \mid X = x] = x$. They require the
Bernoulli random variable
\begin{align*}
  T \sim \bernoulli(\pi_\diffp), \quad \mbox{where} \quad \pi_\diffp
  \defeq e^\diffp / (e^\diffp + 1).
\end{align*}

\noindent \underline{Privacy mechanism for $\ell_2$-ball:} Given a
vector $x \in \real^\usedim$ with $\ltwo{x} \le \radius$, define a
random vector
\begin{align*}
  \Xtil & \defeq \begin{cases} + \radius \frac{x}{\ltwo{x}} & \mbox{
      with probability $\half + \frac{\ltwo{x}}{2\radius}$} \\
    - \radius \frac{x}{\ltwo{x}} & \mbox{ with probability $\half -
      \frac{\ltwo{x}}{2\radius}$}.
  \end{cases}
\end{align*}
Then sample \mbox{$T \sim \bernoulli(\pi_\diffp)$} and set
\begin{align}
  \label{eqn:ltwo-sampling}
  \channelrv \sim \begin{cases} \uniform(\channelval \in \R^d \mid
    \inprod{\channelval}{\Xtil}> 0, \ltwo{\channelval} = \sbound) &
    \mbox{if~} T = 1 \\
    \uniform(\channelval \in \R^d \mid \inprod{\channelval}{\Xtil} \le 0,
    \ltwo{\channelval} = \sbound) & \mbox{if~} T = 0,
  \end{cases}
\end{align}
where $\sbound$ is chosen to equal
\begin{equation*}
  \sbound = \radius \frac{e^\diffp + 1}{e^\diffp - 1}
  \frac{\sqrt{\pi}}{2} \frac{d\, \Gamma(\frac{d - 1}{2} + 1)}{
    \Gamma(\frac{d}{2} + 1)}.
\end{equation*}

\noindent \underline{Privacy mechanism for $\ell_\infty$-ball:} Given
a vector $x \in \real^\usedim$ with $\linf{x} \le \radius$, construct
a random vector $\Xtil \in \R^d$ with independent coordinates of
the form
\begin{align*}
  \Xtil_j & = \begin{cases} + \radius & \mbox{with probability $\half +
      \frac{x_j}{2 \radius}$} \\
    -\radius & \mbox{with probability $\half -
      \frac{x_j}{2 \radius}$}.
  \end{cases}
\end{align*}
Then sample $T \sim \bernoulli(\pi_\diffp)$ and set
\begin{align}
  \label{eqn:linf-sampling}
  \channelrv \sim \begin{cases} \uniform(\channelval \in \{-\sbound,
    \sbound\}^d \mid \inprod{\channelval}{\Xtil} \ge 0) & \mbox{if~} T = 1
    \\
    \uniform(\channelval \in \{-\sbound, \sbound\}^d \mid
    \inprod{\channelval}{\Xtil} \le 0) & \mbox{if~} T = 0,
  \end{cases}
\end{align}
where the value $\sbound$ is chosen to equal
\begin{equation*}
  \sbound = \radius \frac{e^\diffp + 1}{e^\diffp - 1}
  C_d,
  ~~ \mbox{where} ~~
  C_d^{-1} = \begin{cases}
    \frac{1}{2^{d - 1}} \binom{d - 1}{(d - 1)/2} & 
    \mbox{if~} d ~ \mbox{is odd} \\
    \frac{1}{2^{d - 1} + \half \binom{d}{d/2}}
    \binom{d - 1}{d/2} & \mbox{if~} d ~ \mbox{is even.}
  \end{cases}
\end{equation*}

See Figure~\ref{fig:sampling} for visualizations of the geometry that
underlies these strategies.  By construction, each scheme guarantees that
$Z$ is an $\diffp$-private view of $\statrv$.  Each strategy is efficiently
implementable when combined with rejection sampling: the
$\ell_2$-mechanism~\eqref{eqn:ltwo-sampling} by normalizing a sample from
the $\normal(0, I_{d \times d})$ distribution, and the
$\ell_\infty$-strategy~\eqref{eqn:linf-sampling} by sampling the
hypercube $\{-1, 1\}^d$. Additionally, by Stirling's
approximation, we have that in each case $\sbound \lesssim \frac{\radius
  \sqrt{d}}{\diffp}$ for $\diffp \le 1$. Moreover, they are unbiased (see
Appendix~\ref{appendix:ltwo-sampling} for the unbiasedness of
strategy~\eqref{eqn:ltwo-sampling} and Appendix~\ref{appendix:linf-sampling}
for strategy~\eqref{eqn:linf-sampling}).

We now complete the picture of minimax optimal estimation schemes for
Corollaries~\ref{CorDmean} and~\ref{CorSparseMean}.  In the case of
Corollary~\ref{CorDmean}, the estimator $\what{\optvar} \defeq \frac{1}{n}
\sum_{i = 1}^n \channelrv_i$, where $\channelrv_i$ is constructed by
procedure~\eqref{eqn:ltwo-sampling} from the non-private vector $\statrv_i$,
attains the minimax optimal convergence rate.  In the case of
Corollary~\ref{CorSparseMean}, a slightly more complex estimator gives that
rate: in particular, we set
\begin{equation*}
  \what{\optvar}_n \defeq \argmin_{\optvar \in \R^d}
  \bigg\{\frac{1}{2n} \sum_{i = 1}^n \ltwo{\channelrv_i - \optvar}^2
  + \lambda_n \lone{\optvar}\bigg\},
  ~~~ \mbox{where} ~~
  \lambda_n = 2 \frac{\sqrt{d \log d}}{\sqrt{n \diffp^2}},
\end{equation*}
and the
$\channelrv_i$ are drawn according to strategy~\eqref{eqn:linf-sampling}.
See Appendix~\ref{sec:proof-mean-high-dim} for a rigorous argument.

\begin{figure}[t]
  \begin{center}
    \begin{tabular}{ccc}
      \psfrag{g}[][][1]{\large$x$} \psfrag{1}[][][1.4]{$\frac{1}{1 +
          e^\diffp}$} \psfrag{ea}[][][1.4]{$\frac{e^\diffp}{1 +
          e^{\diffp}}$}
      \includegraphics[width=.35\columnwidth]{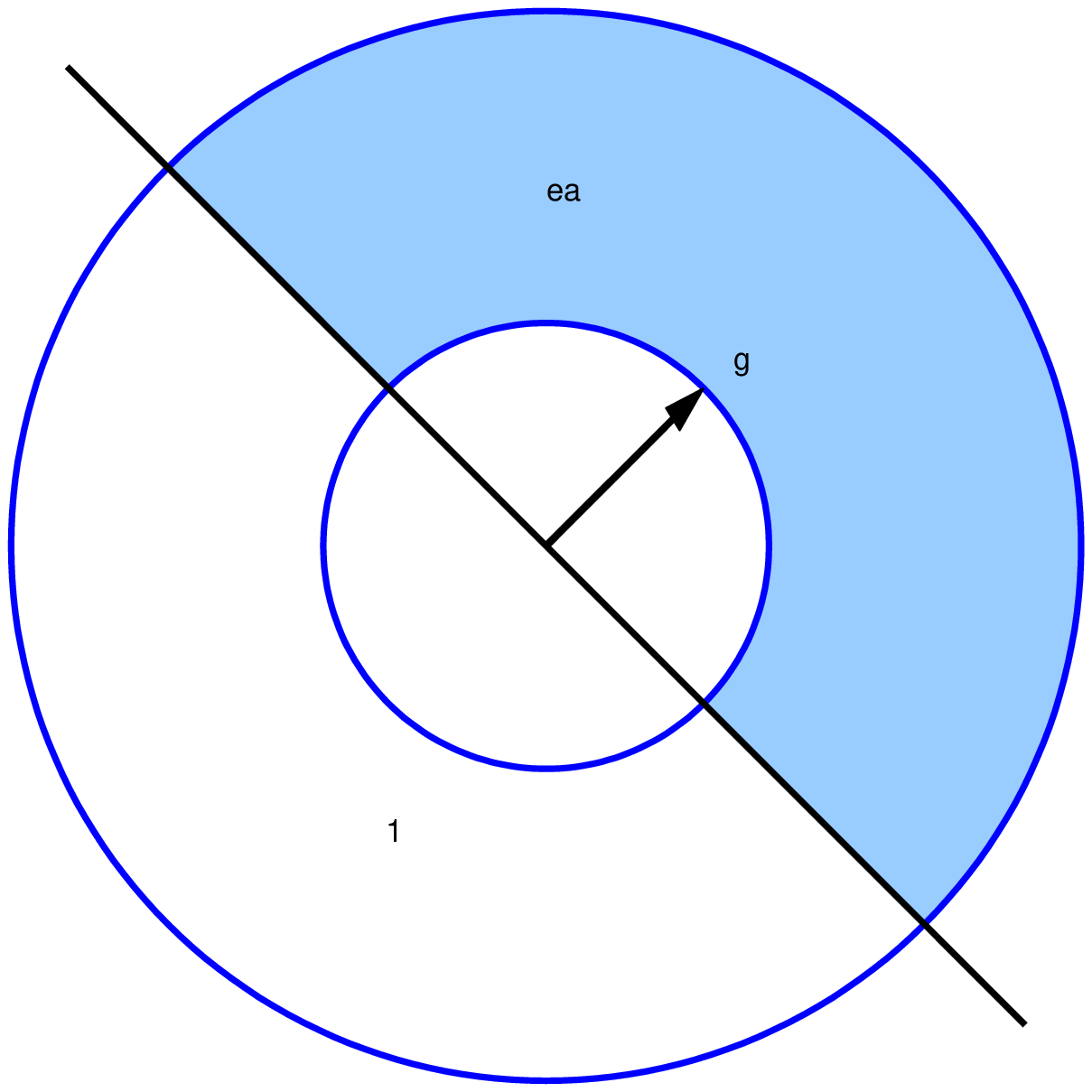} &
      \hspace{.5cm} & \psfrag{g}[][][1]{\large$x$}
      \psfrag{e0}[][][1.4]{$\frac{1}{1 + e^\diffp}$}
      \psfrag{ea}[][][1.4]{$\frac{e^\diffp}{1 + e^{\diffp}}$}
      \includegraphics[width=.4\columnwidth]{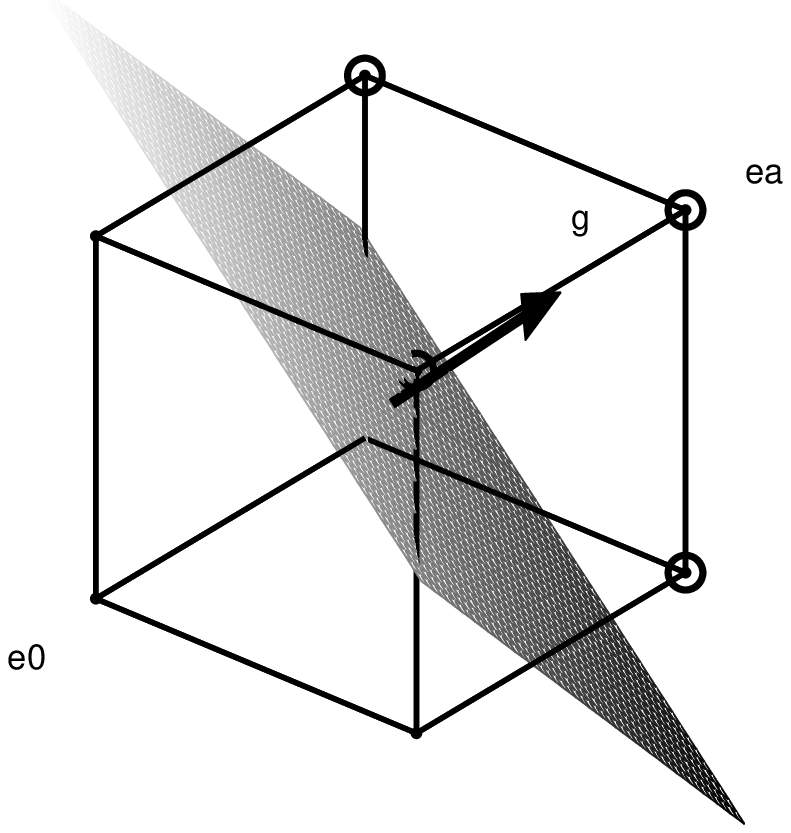} \\ (a)
      & & (b)
    \end{tabular}
  \end{center}
  \caption{\label{fig:sampling} Private sampling strategies for
    different types of $\ell_p$-balls. (a) A private sampling
    strategy~\eqref{eqn:ltwo-sampling} for data lying within an
    $\ell_2$-ball. Outer boundary of highlighted region sampled
    uniformly with probability $e^\diffp / (e^\diffp + 1)$.  (b) A
    private sampling strategy~\eqref{eqn:linf-sampling} for data lying
    within an \mbox{$\ell_\infty$-ball.} Circled point set sampled
    uniformly with probability $e^\diffp / (e^\diffp + 1)$.}
\end{figure}


\comment{
In Proposition~\ref{proposition:minimax-mean-linf}, we consider the family
$\PRFAMINF$ of distributions supported on the $\ell_\infty$-ball of radius
$\radius$. In our mechanism for attaining the
upper bound, we use the sampling scheme~\eqref{eqn:linf-sampling}
to generate the private $\channelrv_i$, so that for an observation $\statrv =
\statsample \in \R^d$ with $\linf{\statsample} \le \radius$, we resample
$\channelrv$ (from the initial vector $v = \statsample$) according to
strategy~\eqref{eqn:linf-sampling}.  Again, we would like to guarantee the
unbiasedness condition $\E[\channelrv \mid \statrv = \statsample] =
\statsample$, for which we use a result of \citet{DuchiJoWa12}. 
That paper shows that taking
\begin{equation}
  \label{eqn:linf-bound-size}
  \sbound = c \frac{\radius \sqrt{d}}{\diffp}
\end{equation}
for a (particular) universal constant $c$, yields the desired
unbiasedness~\cite[Corollary 3]{DuchiJoWa12}.  Since the random
variable $\channelrv$ satisfies $\channelrv \in \ball_\infty(\radius)$
with probability 1, each coordinate $[\channelrv]_j$ of $\channelrv$
is sub-Gaussian. As a consequence, we obtain via standard
bounds~\cite{BuldyginKo00} that
\begin{equation*}
  \E[\linfs{\what{\optvar} - \optvar}^2] \le \frac{\sbound^2 \log(2d)}{n}
  = c^2 \frac{\radius^2 d \log(2d)}{n \diffp^2}
\end{equation*}
for a universal constant $c$, proving the upper bound in
Proposition~\ref{proposition:minimax-mean-linf}.
}
%


\subsection{Variational bounds on private mutual information}

The private Fano bound in Proposition~\ref{PropPrivateFano} reposes on a
variational bound on the private mutual information that we describe
here. Recall the space $L^\infty(\statdomain) \defn \big \{ f : \statdomain
\rightarrow \R \, \mid \, \|f\|_\infty < \infty \big \}$ of uniformly
bounded functions, equipped with the usual sup-norm and unit norm
ball~\eqref{eqn:L-infty-set}, $\linfset$, as well as the linear functionals
$\varphi_\packval$ from Eq.~\eqref{EqnLinearFunctional}.  We then have the
following result.
\begin{theorem}
  \label{theorem:super-master}
  Let $\{\statprob_\packval\}_{\packval \in \packset}$ be an arbitrary
  collection of probability measures on $\statdomain$, and let
  $\{\marginprob_\packval\}_{\packval \in \packset}$ be the set of
  marginal distributions induced by an $\diffp$-differentially private
  distribution $\channelprob$.  Then
  \begin{align*}
    \frac{1}{|\packset|} \sum_{\packval \in \packset}
    \left[\dkl{\marginprob_\packval}{\meanmarginprob} +
      \dkl{\meanmarginprob}{\marginprob_\packval}\right] & \leq
    \frac{(e^\diffp - 1)^2}{|\packset|} \sup_{\optdens \in
      \linfset(\statdomain)} \sum_{\packval \in \packset}
    \left(\varphi_\packval(\optdens)\right)^2.
  \end{align*}
\end{theorem}

It is important to note that, at least up to constant factors,
Theorem~\ref{theorem:super-master} is never weaker than the results provided
by Theorem~\ref{theorem:master}.  By definition of the linear functional
$\varphi_\packval$, we have
\begin{align*}
  \sup_{\optdens \in \linfset(\statdomain)} \sum_{\packval \in
    \packset} \left(\varphi_\packval(\optdens)\right)^2
  & \stackrel{(i)}{\le}
  \sum_{\packval \in \packset} \sup_{\optdens \in \linfset(\statdomain)}
  \left(\varphi_\packval(\optdens)\right)^2
  = 4 \sum_{\packval \in \packset}
  \tvnorm{\statprob_\packval - \meanstatprob}^2,
\end{align*}
where inequality $(i)$ follows by interchanging the summation and
supremum.  Overall, we have
\begin{align*}
  I(\channelrv; \packrv)
  = \frac{1}{|\packset|} \sum_{\packval \in \packset}
  \dkl{\marginprob_\packval}{\meanmarginprob}
  \le 4 (e^\diffp - 1)^2 \frac{1}{|\packset|^2} \sum_{\packval,
    \altpackval \in \packset} \tvnorm{\statprob_\packval -
    \statprob_{\altpackval}}^2.
\end{align*}
The strength of Theorem~\ref{theorem:super-master} arises from the
fact that inequality $(i)$---the interchange of the order of supremum
and summation---may be quite loose.

We may extend Theorem~\ref{theorem:super-master} to sequences of random
variables, that is, to collections $\{\statprob_\packval^n\}_{\packval \in
  \packset}$ of product probability measures in the non-interactive
case. Indeed, we have by the standard chain rule for
mutual information~\cite[Chapter 5]{Gray90} that
\begin{equation*}
  I(\channelrv_1, \ldots, \channelrv_n; \packrv)
  = \sum_{i = 1}^n I(\channelrv_i; \packrv \mid \channelrv_{1:i-1})
  \le \sum_{i = 1}^n I(\channelrv_i; \packrv),
\end{equation*}
where the inequality is a consequence of the conditional independence of the
variables $\channelrv_i$ given $\packrv$, which holds when the channel
$\channel$ is non-interactive. Applying Theorem~\ref{theorem:super-master}
to the individual terms $I(\channelrv_i; \packrv)$ then yields
inequality~\eqref{EqnSuperFano}; see Appendix~\ref{sec:proof-super-fano} for
a fully rigorous derivation.




\section{Bounds on multiple pairwise divergences: Assouad's method}
\label{sec:assouad}

Thus far, we have seen how Le Cam's method and Fano's method, in the form of
Propositions~\ref{PropPrivateLeCam} and~\ref{PropPrivateFano}, can be used
to derive sharp minimax rates.  However, their application appears to be
limited to problems whose minimax rates can be controlled via reductions to
binary hypothesis tests (Le Cam's method) or for non-interactive privacy
mechanisms (Fano's method).  Another classical approach to deriving minimax
lower bounds is Assouad's method~\cite{Assouad83,Yu97}.  In this section, we 
show that a privatized form of Assouad's method can be be used to obtain sharp minimax
rates in interactive settings.  We illustrate by deriving bounds for several
problems, including multinomial probability estimation and nonparametric
density estimation.

Assouad's method transforms an estimation problem into multiple binary
hypothesis testing problems, using the structure of the problem in an
essential way. For some $d \in \N$, let $\packset = \{-1, 1\}^d$, and let us
consider a family of distributions $\{\statprob_\packval\}_{\packval \in
  \packset}$ indexed by the hypercube. We say that the family
$\{\statprob_\packval\}_{\packval \in \packset}$ induces a \emph{$2
  \delta$-Hamming separation} for the loss $\bigloss \circ \metric$ if there
exists a vertex mapping (a function $\maptocube : \optvar(\mc{\statprob})
\to \{-1, 1\}^d$) satisfying
\begin{equation}
  \label{eqn:risk-separation}
  \bigloss(\metric(\optvar, \optvar(\statprob_\packval)))
  \ge 2\delta \sum_{j=1}^d \indic{[\maptocube(\optvar)]_j
    \neq \packval_j}.
\end{equation}
As in the standard reduction from estimation to testing, we consider the
following random process: Nature chooses a vector $\packrv \in \{-1, 1\}^d$
uniformly at random, after which the sample $\statrv_1, \ldots, \statrv_n$
is drawn from the distribution $\statprob_\packval$ conditional on $\packrv
= \packval$. Letting $\P_{\pm j}$ denote the joint distribution over the
random index $\packrv$ and $\statrv$ conditional on the $j$th coordinate
$\packrv_j = \pm 1$, we obtain
the following sharper variant of Assouad's lemma~\cite{Assouad83}.
\begin{lemma}[Sharper Assouad method]
  \label{lemma:sharp-assouad}
  Under the conditions of the previous paragraph, we have
  \begin{align*}
    \minimax_n(\optvar(\pclass), \bigloss \circ \metric)
    & \ge \delta \sum_{j=1}^d
    \inf_{\test} \left[\P_{+j}(\test(\statrv_{1:n})
      \neq +1) + \P_{-j}(\test(\statrv_{1:n})
      \neq -1)\right].
  \end{align*}
\end{lemma}
\noindent We provide a proof of Lemma~\ref{lemma:sharp-assouad} in
Section~\ref{sec:proof-sharp-assouad} (see also the
paper~\cite{Arias-CastroCaDa13}).  We can also give a variant of
Lemma~\ref{lemma:sharp-assouad} after some minor
rewriting. For each $j \in [d]$ define the mixture
distributions
\begin{equation}
  \statprob_{+j}^n = \frac{1}{2^{d-1}}\sum_{\packval : \packval_j = 1}
  \statprob_\packval^n,
  ~~~
  \statprob_{-j}^n = \frac{1}{2^{d-1}} \sum_{\packval : \packval_j = -1}
  \statprob_\packval^n,
  ~~~
  \statprob_{\pm j} = \frac{1}{2^{d-1}} \sum_{\packval : \packval_j = \pm 1}
  \statprob_{\packval}
  \label{eqn:paired-mixtures}
\end{equation}
where $\statprob^n_\packval$ is the (product) distribution of $X_1, \ldots,
X_n$.  Then, by Le Cam's lemma, the following minimax lower bound is
equivalent to the Assouad bound of Lemma~\ref{lemma:sharp-assouad}:
\begin{equation}
  \minimax_n(\optvar(\pclass), \bigloss \circ \metric)
  \ge \delta \sum_{j = 1}^d
  \left[1 - \tvnorm{\statprob_{+j}^n - \statprob_{-j}^n}\right].
  \label{eqn:sharp-assouad}
\end{equation}


\subsection{A private version of Assoud's method}

As in the preceding sections, we extend
Lemma~\ref{lemma:sharp-assouad} to the locally differentially private
setting. In this case, we are able to provide a minimax lower bound
that applies to any locally differentially private channel $\channel$,
including in interactive settings
(Figure~\ref{fig:interactive-channel}(a)).  In this case, we again let
$\linfset(\statdomain)$ denote the collection of functions $f :
\statdomain \to \R$ with supremum norm bounded by $1$
(definition~\eqref{eqn:L-infty-set}).  Then we have the following
private version of Assouad's method.
\begin{proposition}[Private Assouad bound]
  \label{proposition:private-assouad}
  Let the conditions of Lemma~\ref{lemma:sharp-assouad} hold, that is, let
  the family $\{\statprob_\packval\}_{\packval \in \packset}$ induce a
  $2\delta$-Hamming separation for the loss $\bigloss \circ \metric$. Then
  \begin{equation*}
    \minimax_n(\optvar(\pclass), \bigloss \circ \metric, \diffp)
    \ge d \delta \left[1 - \Bigg(\frac{n(e^\diffp - 1)^2}{2d}
      \sup_{\optdens \in \linfset(\statdomain)}
      \sum_{j = 1}^d \left(\int_{\statdomain}
      \optdens(\statsample) (d\statprob_{+j}(\statsample)
      - d\statprob_{-j}(\statsample))\right)^2\Bigg)^\half
      \right].
  \end{equation*}
\end{proposition}
\noindent
As is the case for our private analogue of Fano's method
(Proposition~\ref{PropPrivateFano}), underlying
Proposition~\ref{proposition:private-assouad} is a variational bound that
generalizes the usual total variation distance to a variational quantity
applied jointly to multiple mixtures of distributions.
Proposition~\ref{proposition:private-assouad} is an immediate consequence of
Theorem~\ref{theorem:sequential-interactive} to come.


\subsection{Some applications of the private Assouad bound}

Proposition~\ref{proposition:private-assouad} allows sharp characterizations
of $\diffp$-private minimax rates for a number of classical
statistical problems. While it requires that there be a natural
coordinate-wise structure to the problem at hand because of the
Hamming separation condition~\eqref{eqn:risk-separation}, such conditions
are common in a number of estimation problems. Additionally,
Proposition~\ref{proposition:private-assouad} applies to interactive
channels $\channel$. As examples, we consider generalized linear
model and nonparametric density estimation.

\subsubsection{Generalized linear model estimation under local privacy}
\label{sec:glm-estimation}

For our first applications of
Proposition~\ref{proposition:private-assouad}, we consider a (somewhat
simplified) family of generalized linear models (GLMs), showing how to
perform inference for the parameter of the GLM under local
differential privacy, and arguing by an example using logistic
regression that---in a minimax sense---local differential
privacy again leads to an effective degradation in sample size of $n
\mapsto \frac{n \diffp^2}{d}$ for $\diffp = \order(1)$. In our GLM
setting, we model a target variable $y \in \mc{Y}$ conditional on
independent variables $\statrv = \statval$ as follows. Let $\mu$ be a
base measure on the space $\mc{Y}$, assume we represent the variables
$\statrv$ as a matrix $\statval \in \R^{d \times k}$ (we make implicit
any transformations performed on the data), and let $T : \mc{Y} \to
\R^k$ be the sufficient statistic for $Y$.  Then we model $Y \mid
\statrv = \statval$ according to
\begin{equation}
  \label{eqn:glm}
  p(y \mid \statval; \optvar) = \exp \Big(\<T(y), \statval^\top
  \optvar\> - A(\optvar, \statval) \Big), ~~~~ A(\optvar,
  \statval) \defeq \int_{\mc{Y}} \exp \Big (\<T(y), \statval^\top
  \optvar\> \Big) d\mu(y),
\end{equation}
so that $A(\cdot, \statval)$ is the cumulant function.

Developing a strategy for fitting GLMs~\eqref{eqn:glm} that allows
independent perturbation of data pairs $(\statrv, Y)$ appears
challenging, because most methods for fitting the model require
differentiating the cumulant function $A(\optvar, \statval)$,
which in turn generally requires knowing $\statval$. (In some special
cases, such as linear regression~\cite{LohWa12}, it is possible to
perturb the independent variables $\statrv$, but in general there is
no efficient standard methodology.)  That being said, there are
natural \emph{sequential} strategies based on stochastic gradient
descent---allowable in our interactive model of privacy (recall
Figure~\ref{fig:interactive-channel})---that provide local
differential privacy and efficient fitting of conditional
models~\eqref{eqn:glm}. Given the well-known difficulties of
estimation in perturbed (independent) variable models, we advocate
these types of sequential strategies for conditional models, which we
now describe in somewhat more care.

\paragraph{Stochastic gradient for private estimation of GLMs}

\newcommand{\subgrad}{\mathsf{g}}
\providecommand{\loss}{\ell}
\providecommand{\risk}{R}
\newcommand{\opt}{^\star}

The log loss $\loss(\optvar; \statval, y) = - \log p(y
\mid \statval; \optvar)$ for the model family~\eqref{eqn:glm} is
convex, and for each $\statval$, the function $A(\cdot, \statval)$ is
infinitely differentiable on its domain~\cite{Brown86}. Thus,
stochastic gradient descent
methods~\cite{NemirovskiJuLaSh09,PolyakJu92} are natural candidates
for minimizing the risk (population log-loss) $\risk(\optvar) \defeq
\E_\statprob[\loss(\optvar; \statrv, Y)]$. The first ingredient in
such a scheme, of which we give explicit examples presently, is an
unbiased gradient estimator. Let $\subgrad(\optvar; \statrv, Y)$ be a
random stochastic gradient vector, unbiased for the gradient of the
negative log-likelihood, constructed conditional on $\optvar, \statrv,
Y$ so that
\begin{align*}
  \E[\subgrad(\optvar; \statval, y)] = \nabla \loss(\optvar; \statval,
  y) = \statval T(y) - \nabla A(\optvar, \statval) \in \R^d,
\end{align*}
for fixed $\statval, y$. (Recall that $\statval \in \R^{d \times k}$
and $T(y) \in \R^k$.) Stochastic gradient descent proceeds
iteratively using stepsizes $\stepsize_i > 0$ as follows. Beginning
from a point $\optvar_0 \in \R^d$, at iteration $i$, we receive a pair
$(\statrv_i, Y_i) \simiid \statprob$, then perform the stochastic
gradient update
\begin{equation}
  \optvar_{i + 1} = \optvar_i -
  \stepsize_i \subgrad(\optvar_i; \statrv_i, Y_i),
  \label{eqn:sgd-update}
\end{equation}
where $\subgrad(\optvar; \statrv_i, Y_i)$ is unbiased for $\nabla
\loss(\optvar; \statrv_i, y_i)$.

We briefly review the (well-known) convergence properties of such stochastic
gradient procedures. Let us assume that $\optvar\opt \defeq
\argmin_\optvar \risk(\optvar)$ is such that $\nabla^2 \E[A(\optvar\opt;
  \statrv)] \succ 0$, that is, the Hessian of $\E[A(\optvar; \statrv)]$ at
$\optvar = \optvar\opt$ is positive definite, and that the random unbiased
estimates $\subgrad$ of $\nabla \loss$ are chosen in such a way that the
boundedness condition $\sup_{\optvar, \statval, y} \norm{\subgrad(\optvar,
  \statval, y)} < \infty$ holds with probability 1. For example, if
$\statrv$ is compactly supported, has a full-rank covariance matrix, and the
sufficient statistic $T(\cdot)$ is bounded, this holds.  Then the following
result is standard.
\begin{lemma}[Polyak and Juditsky~\cite{PolyakJu92}, Thm.~3]
  \label{lemma:polyak-juditsky}
  Let the conditions in the previous paragraph hold, let $\stepsize_i =
  \stepsize_0 i^{-\beta}$ for some $\beta \in (\half, 1)$, and let
  $\what{\optvar}_n = \frac{1}{n} \sum_{i = 1}^n \optvar_i$.  Let
  $A(\optvar) = \int A(\optvar, \statval) dP(\statval)$.  Then
  \begin{equation*}
    \sqrt{n}(\what{\optvar}_n - \optvar\opt)
    \cd \normal(0, \Sigma),
  \end{equation*}
  where
  \begin{equation*}
    \Sigma = (\nabla^2 A(\optvar\opt))^{-1}
    \E[\subgrad(\optvar\opt; \statrv, Y)
      \subgrad(\optvar\opt; \statrv, Y)^\top]
    (\nabla^2 A(\optvar\opt))^{-1}.
  \end{equation*}
\end{lemma}
\noindent
While Lemma~\ref{lemma:polyak-juditsky} is asymptotic, it provides an
exact characterization of the asymptotic distribution of the parameter and
allows inference for parameter values.

That the iteration~\eqref{eqn:sgd-update} and convergence guarantee of
Lemma~\ref{lemma:polyak-juditsky} allow unbiased (noisy) versions of
the gradient $\nabla \loss$ is suggestive of a private estimation
procedure: add sufficient noise to the gradient $\nabla \loss$ so as
to render it private while ensuring that the noise has sufficiently light
tails that the convergence conditions of Lemma~\ref{lemma:polyak-juditsky}
apply, and then perform stochastic gradient descent to estimate the
model~\eqref{eqn:glm}. To make this intuition concrete, we
now give an explicit recipe that yields locally differentially private
estimators with (asymptotically) minimax optimal convergence rates,
leveraging the optimal mechanisms for mean estimation in
Sec.~\ref{sec:attainability-means} to construct the unbiased gradients
$\subgrad$.

We assume the compactness condition
\begin{equation*}
  \norm{\statval T(y)} \le \radius
  ~~ \mbox{for~all~} (x, y) \in \supp \statprob,
\end{equation*}
where $\norm{\cdot}$ is either the $\ell_\infty$-norm or $\ell_2$-norm on
$\R^d$. Using $\nabla A(\optvar, \statval) =
\statval \E_{\rm glm}[T(Y) \mid \statval, \optvar]$, where $\E_{\rm glm}$
denotes expectation in the model~\eqref{eqn:glm}, we have $\norm{\statval
  T(y) - \nabla A(\optvar, \statval)} \le 2 \radius$.
Now let $\channel$ be the private channel
with mean $\nabla \loss(\theta;
\statval, y)$ using either of our half-space sampling
schemes~\eqref{eqn:ltwo-sampling} or~\eqref{eqn:linf-sampling}, and
draw the conditionally unbiased stochastic gradient
\begin{equation*}
  \subgrad(\optvar; \statval, y) \sim \channel(\cdot \mid \optvar, \statval,
  y).
\end{equation*}
Then we have
\begin{equation*}
  \tr(\E_{\statprob, \channel}[\subgrad(\optvar\opt; \statrv, Y)
    \subgrad(\optvar\opt; \statrv, Y)^\top])
  \le c \cdot \radius^2 \frac{(e^\diffp + 1)^2}{(e^\diffp - 1)^2} d
  \cdot \begin{cases} d & \mbox{if~} \norm{\cdot} = \linf{\cdot} \\
    1 & \mbox{if~} \norm{\cdot} = \ltwo{\cdot},
  \end{cases}
\end{equation*}
where $c$ is a numerical constant. In particular, for any finite number
$B < \infty$, we obtain
\begin{equation}
  \label{eqn:convergence-sgd}
  \E\left[B \wedge \ltwo{\what{\optvar}_n
      - \optvar\opt}^2\right]
  \lesssim B \wedge \opnorm{\nabla^2 A(\optvar\opt)^{-1}}^2
  \frac{\radius^2}{n (e^\diffp - 1)^2} d
  \cdot
  \begin{cases}
    d & \mbox{if~} \norm{\cdot} = \linf{\cdot} \\
    1 & \mbox{if~} \norm{\cdot} = \ltwo{\cdot}.
  \end{cases}
\end{equation}
That is, we have asymptotic mean-squared error (MSE)
of order $\opnorms{\nabla^2
  A(\optvar\opt)^{-1}}^2 \frac{d r^2}{n \diffp^2}$ if we use the
$\ell_2$-sampling scheme~\eqref{eqn:ltwo-sampling} and the data lie in an
$\ell_2$-ball of radius $\radius$, and asymptotic MSE of
order $\opnorms{\nabla^2 A(\optvar\opt)^{-1}}^2 \frac{d^2 r^2}{n \diffp^2}$
using the $\ell_\infty$-sampling scheme~\eqref{eqn:linf-sampling}, assuming
the data lie in an $\ell_\infty$-ball of radius $\radius$.

\paragraph{Minimax lower bounds for logistic regression}

To show the sharpness of our achievability guarantees for stochastic
gradient methods, we consider lower bounds for a binary logistic
regression problem; these lower bounds will show that in
general, it is impossible to outperform the convergence
guarantee~\eqref{eqn:convergence-sgd} of stochastic gradient descent
for conditionally-specified models.

Let $\mc{P}$ be the family of logistic distributions on
covariate-response pairs $(\statrv, Y) \in \{-1, 1\}^d \times \{-1,
1\}$; as we prove a lower bound, larger families can only increase the
bound.  We assume that
\begin{equation*}
  P(Y = y \mid \statrv = \statval)
  = \frac{1}{1 + e^{-y \optvar^\top \statval}}
  ~~~ \mbox{for~some}~ \optvar \in \R^d
  ~~ \mbox{with} ~ \ltwo{\optvar}^2 \le d,
\end{equation*}
meaning that $Y$ has a standard logistic distribution. We then have
the following corollary of
Proposition~\ref{proposition:private-assouad}, where $\optvar(P) \in
\R^d$ is the standard logistic parameter vector.  In stating the
corollary, we use the loss $d \wedge \ltwos{\what{\theta} -
  \theta(P)}^2$, as our construction guarantees that $\ltwo{\theta}^2
\le d$.
\begin{corollary}
  \label{corollary:full-logistic-lower-bound}
  For the logistic family $\mc{P}$ of distributions parameterized
  by $\theta \in \R^d, \ltwo{\theta}^2 \le d$, we have for
  all $\diffp \ge 0$ that
  \begin{equation}
    \label{eqn:full-logistic-lower-bound}
    \minimax_n(\optvar(\pclass),
    \ltwo{\cdot}^2 \wedge d, \diffp)
    \ge \min\left\{\frac{d}{4},
    \frac{d^2}{4 n(e^\diffp - 1)^2}\right\}.
  \end{equation}
\end{corollary}
\noindent
We provide the proof of the proposition in
Appendix~\ref{sec:proof-full-logistic-lower-bound}. 

To understand the sharpness of this prediction, we may consider a special
case of the logistic regression model. When the logistic model is true, then
standard results on exponential families~\cite{Lehmann99} show that the
\emph{non-private} maximum likelihood estimator $\what{\optvar}_{{\rm
    ML},n}$ based on a sample of size $n$ satisfies
\begin{equation*}
  \sqrt{n} (\what{\optvar}_{{\rm ML},n} - \optvar\opt)
  \cd \normal\left(0, \E\left[\statrv \statrv^\top
    p(Y \mid \statrv, \optvar\opt)
    (1 - p(Y \mid \statrv, \optvar \opt))\right]^{-1}
  \right),
\end{equation*}
where the covariance is the inverse of the expected conditional Fisher
Information.  In the ``best case'' (i.e., largest Fisher information) for
estimation when $\theta\opt = 0$, this quantity is simply $\frac{1}{4}
\cov(X)$. As our proof makes precise, our minimax lower
bound~\eqref{eqn:full-logistic-lower-bound} is a local bound that applies for
parameters $\theta$ shrinking to $\theta\opt = 0$, and when $\theta\opt = 0$
we have $\cov(X) = I_{d \times d}$. In particular, in the non-private case
we have
\begin{equation*}
  \E\left[d \wedge \ltwos{\what{\theta}_{{\rm ML}, n} - \theta\opt}^2\right]
  \lesssim \frac{d}{n}
  ~~~ \mbox{as~} n \to \infty
\end{equation*}
for $\theta\opt$ near zero (by continuity of the distribution that $\theta\opt$
parameterizes). Conversely, our minimax bound shows that no private
estimator can have risk better than $\frac{d}{(e^\diffp - 1)^2} \cdot
\frac{d}{4n}$ under this model, which our estimators achieve: recall
inequality~\eqref{eqn:convergence-sgd}, where we may take $\opnorms{\nabla
  A(\optvar\opt)^{-1}} = \order(1)$. As is typical for the locally private
setting, we see a sample size degradation of $n \mapsto
\frac{n (e^\diffp - 1)^2}{d}$.


\subsubsection{Density estimation under local privacy}
\label{sec:density-estimation}

In this section, we show that the effects of local differential
privacy are more severe for nonparametric density estimation: instead
of just a multiplicative loss in the effective sample size as in
previous sections, imposing local differential privacy leads to a
different convergence rate.  This result holds even though we solve a
problem in which both the function being estimated and the observations 
themselves belong to compact spaces.


\begin{definition}[Elliptical Sobolev space]
  \label{definition:sobolev-densities}
  For a given orthonormal basis $\{\basisfunc_j\}$ of $L^2([0, 1])$,
  smoothness parameter $\beta > 1/2$ and radius $\lipconst$, the
  Sobolev class of order $\beta$ is given by
  \begin{equation*}
    \densclass[\lipconst] \defeq \bigg\{f \in L^2([0, 1]) \mid f =
    \sum_{j=1}^\infty \optvar_j \basisfunc_j ~ \mbox{such~that} ~
    \sum_{j=1}^\infty j^{2\numderiv} \optvar_j^2 \le
    \lipconst^2\bigg\}.
  \end{equation*}
\end{definition}

If we choose the trignometric basis as our orthonormal basis,
membership in the class $\densclass[\lipconst]$ corresponds to
smoothness constraints on the derivatives of $f$.  More precisely, for
$j \in \N$, consider the orthonormal basis for $L^2([0, 1])$ of
trigonometric functions:
\begin{equation}
  \label{eqn:trig-basis}
  \basisfunc_0(t) = 1,
  ~~~ \basisfunc_{2j}(t) = \sqrt{2} \cos(2\pi j t),
  ~~~ \basisfunc_{2j + 1}(t) = \sqrt{2} \sin(2\pi j t).
\end{equation}
Let $f$ be a $\numderiv$-times almost everywhere differentiable
function for which $|f^{(\numderiv)}(x)| \le \lipconst$ for almost
every $x \in [0, 1]$ satisfying $f^{(k)}(0) = f^{(k)}(1)$ for $k \le
\numderiv - 1$. Then, uniformly over all such $f$, there is a
universal constant $c \le 2$ such that that $f \in \densclass[c \lipconst]$
(see, for instance,~\cite[Lemma A.3]{Tsybakov09}).  

Suppose our goal is to estimate a density function $f \in
\densclass[C]$ and that quality is measured in terms of the
squared error (squared $L^2[0,1]$-norm):
\begin{align*}
  \|\what{f} - f\|_2^2 \defeq \int_0^1 (\what{f}(x) - f(x))^2 dx.
\end{align*}
The well-known~\cite{Yu97,YangBa99,Tsybakov09} (non-private)
minimax mean squared error scales as
\begin{equation}
  \label{eqn:classical-density-estimation-rate}
  \minimax_n \left( \densclass, \ltwo{\cdot}^2, \infty\right) \asymp
  n^{-\frac{2 \numderiv}{2 \numderiv + 1}}.
\end{equation}
The goal of this section is to understand how this minimax rate
changes when we add an $\diffp$-privacy constraint to the problem.
Our main result is to demonstrate that the classical
rate~\eqref{eqn:classical-density-estimation-rate} is no longer
attainable when we require $\diffp$-local differential privacy.


\paragraph{Lower bounds on density estimation}

We begin by giving our main lower bound on the minimax rate of estimation of
densities when observations from the density are differentially private.  We
provide the proof of the following result in
Appendix~\ref{sec:proof-density-estimation}.
\begin{corollary}
  \label{CorDensity}
  For the class of densities $\densclass$ defined using the
  trigonometric basis~\eqref{eqn:trig-basis}, there exist constants $0 <
  c_\numderiv \leq c'_\numderiv < \infty$, dependent
  only on $\numderiv$, such that the $\diffp$-private minimax
  risk (for $\diffp \in [0,1]$) is sandwiched as
  \begin{align}
    \label{eqn:private-density-estimation-rate}
    c_\numderiv \left (n \diffp^2 \right)^{-\frac{2 \numderiv}{2
        \numderiv + 2}} \; \leq \; \minimax_n \left( \densclass[1],
    \ltwo{\cdot}^2, \diffp \right) \; \leq \; c'_\numderiv \left (n
    \diffp^2 \right)^{-\frac{2 \numderiv}{2 \numderiv + 2}}.
  \end{align}
\end{corollary}
\noindent 
The most important feature of the lower
bound~\eqref{eqn:private-density-estimation-rate} is that it involves
a \emph{different polynomial exponent} than the classical minimax
rate~\eqref{eqn:classical-density-estimation-rate}.  Whereas
the exponent in classical
case~\eqref{eqn:classical-density-estimation-rate} is $2 \numderiv /
(2 \numderiv + 1)$, it reduces to $2 \numderiv / (2 \numderiv +
2)$ in the locally private setting.  For example, when we estimate
Lipschitz densities ($\numderiv = 1$), the rate degrades from
$n^{-2/3}$ to $n^{-1/2}$. 


Interestingly, no estimator based on Laplace (or exponential)
perturbation of the observations $\statrv_i$ themselves can attain the
rate of convergence~\eqref{eqn:private-density-estimation-rate}.  This
fact follows from results of \citet{CarrollHa88} on nonparametric
deconvolution. They show that if observations $\statrv_i$ are
perturbed by additive noise $W$, where the characteristic function
$\phi_W$ of the additive noise has tails behaving as $|\phi_W(t)| =
\order(|t|^{-a})$ for some $a > 0$, then no estimator can deconvolve
$\statrv + W$ and attain a rate of convergence better than
$n^{-2 \numderiv / (2 \numderiv + 2a + 1)}$.
Since the characteristic function of the Laplace distribution has tails
decaying as $t^{-2}$, no estimator based on the Laplace mechanism
(applied directly to the observations) can attain rate of convergence
better than $n^{-2\numderiv / (2 \numderiv + 5)}$.  In order to attain
the lower bound~\eqref{eqn:private-density-estimation-rate}, we must
thus study alternative privacy mechanisms.

\paragraph{Achievability by orthogonal projection estimators}

For $\numderiv = 1$, histogram estimators with counts perturbed by Laplacian
noise achieve the optimal rate of
convergence~\eqref{eqn:private-density-estimation-rate}; this is a
consequence of the results of \citet[Section 4.2]{WassermanZh10} applied to
locally private mechanisms. For higher degrees of smoothness ($\numderiv >
1$), standard histogram estimators no longer achieve optimal rates in the
classical setting~\cite{Scott79}.  Accordingly, we now turn to developing
estimators based on orthogonal series expansion, and show that even in the
setting of local privacy, they can achieve the lower
bound~\eqref{eqn:private-density-estimation-rate} for all orders of
smoothness $\numderiv \ge 1$.

Recall the elliptical Sobolev space
(Definition~\ref{definition:sobolev-densities}), in which a function $f$ is
represented in terms of its basis expansion $f = \sum_{j = 1}^\infty
\optvar_j \basisfunc_j$.  This representation underlies the following classical
orthonormal series estimator. Given a sample $\statrv_{1:n}$
drawn i.i.d.\ according to a density $f \in L^2([0, 1])$, compute the
empirical basis coefficients
\begin{equation}
  \what{\optvar}_j = \frac{1}{n} \sum_{i=1}^n \basisfunc_j(\statrv_i)
  ~~~ \mbox{for~} j \in \{1, \ldots, \numbin\},
  \label{eqn:projection-estimator}
\end{equation}
where the value $\numbin \in \N$ is chosen either a priori based on
known properties of the estimation problem or adaptively, for example,
using cross-validation~\cite{Efromovich99,Tsybakov09}.  Using these
empirical coefficients, the density estimate is $\what{f} =
\sum_{j=1}^\numbin \what{\optvar}_j \basisfunc_j$.

In our local privacy setting, we consider a mechanism that employs a random
vector $\channelrv_i = (\channelrv_{i,1}, \ldots, \channelrv_{i, \numbin})$
satisfying the unbiasedness condition $\E [ \channelrv_{i,j} \mid \statrv_i]
= \basisfunc_j(\statrv_i)$ for each $j \in [\numbin]$.  We assume the basis
functions are $\orthbound$-uniformly bounded; that is, $\sup_j
\sup_\statsample |\basisfunc_j(\statsample)| \leq \orthbound < \infty$.
This boundedness condition holds for many standard bases, including the
trigonometric basis~\eqref{eqn:trig-basis} that underlies the classical
Sobolev classes and the Walsh basis.  We generate the random variables from
the vector $v \in \R^\numbin$ defined by $v_j = \basisfunc_j(\statrv)$ in
the hypercube-based sampling scheme~\eqref{eqn:linf-sampling}, where we
assume $\sbound > \orthbound$.  With this sampling
strategy, iteration of expectation yields
\begin{equation}
  \E[[\channelrv]_j \mid \statrv = \statsample] = c_\numbin
  \frac{\sbound}{\orthbound \sqrt{\numbin}}
  \left(\frac{e^\diffp}{e^\diffp + 1} - \frac{1}{e^\diffp + 1}\right)
  \basisfunc_j(\statsample),
  \label{eqn:size-infinity-channel}
\end{equation}
where $c_k > 0$ is a constant (which is bounded independently of $k$).
Consequently, it suffices to take $\sbound = \order(\orthbound \sqrt{\numbin}
/ \diffp)$  to guarantee the unbiasedness condition
$\E[[\channelrv_i]_j \mid \statrv_i] = \basisfunc_j(\statrv_i)$.

Overall, the privacy mechanism and estimator perform the
following steps:
\begin{enumerate}[(i)]
\item given a data point $\statrv_i$, 
  set the vector $v = [\basisfunc_j(\statrv_i)]_{j=1}^\numbin$;
\item sample $\channelrv_i$ according
  to the strategy~\eqref{eqn:linf-sampling}, starting from the
  vector $v$ and using the
  bound $\sbound = \orthbound \sqrt{\numbin} (e^\diffp + 1) /
  c_\numbin (e^\diffp - 1)$;
\item compute the density estimate
  \begin{equation}
    \what{f} \defeq \frac{1}{\numobs} \sum_{i=1}^\numobs
    \sum_{j=1}^\numbin \channelrv_{i,j} \basisfunc_j.
    \label{eqn:orthogonal-density-estimator}
  \end{equation}
\end{enumerate}
Whenever the underlying function $f$ belongs to the Sobolev space
$\densclass[\lipconst]$ and the orthonormal basis functions $\basisfunc_j$
are uniformly bounded by $\orthbound$, then the
estimator~\eqref{eqn:orthogonal-density-estimator} with the choice $\numbin
= (n \diffp^2)^{1 / (2 \numderiv + 2)}$ has mean squared error upper bounded
as
\begin{equation*}
  \E_f \left [\ltwos{f - \what{f}}^2 \right] \leq c_{\orthbound,\numderiv}
  \left (n \diffp^2 \right)^{-\frac{2 \numderiv}{2 \numderiv + 2}}.
\end{equation*}
This shows that the minimax
bound~\eqref{eqn:private-density-estimation-rate} is indeed sharp, and there
exist easy-to-compute estimators achieving the guarantee. See
Section~\ref{sec:proof-density-upper-bound} for a proof of this inequality.


Before concluding our exposition, we make a few remarks on other
potential density estimators.  Our orthogonal series
estimator~\eqref{eqn:orthogonal-density-estimator} and sampling
scheme~\eqref{eqn:size-infinity-channel}, while similar in spirit to
that proposed by~\citet[Sec.~6]{WassermanZh10}, is different in that
it is locally private and requires a different noise strategy to
obtain both $\diffp$-local privacy and the optimal convergence rate.
Lastly, similarly to our remarks on the insufficiency of standard
Laplace noise addition for mean estimation, it is worth noting that
density estimators that are based on orthogonal series and Laplace
perturbation are sub-optimal: they can achieve (at best) rates of $(n
\diffp^2)^{-\frac{2 \numderiv}{2 \numderiv + 3}}$.  This rate is
polynomially worse than the sharp result provided by
Corollary~\ref{CorDensity}.  Again, we see that
appropriately chosen noise mechanisms are crucial for obtaining
optimal results.


\subsection{Variational bounds on paired divergences}
\label{SecAssouadGeneral}

In this section, we provide a theorem that underpins the variational
inequality in the minimax lower bound of
Proposition~\ref{proposition:private-assouad}. We provide a slightly more
general bound than that required for
the proposition, showing how it implies the
earlier result. Recall that for some $d \in \N$, we consider collections of
distributions indexed using the Boolean hypercube $\packset = \{-1, 1\}^d$.
For each $i \in [n]$ and $\packval \in \packset$, let the distribution
$\statprob_{\packval,i}$ be supported on the fixed set $\statdomain$,
and define the product distribution $\statprob_\packval^n = \prod_{i=1}^n
\statprob_{\packval,i}$. Then in addition to
the definition~\eqref{eqn:paired-mixtures} of the
paired mixtures $\statprob_{\pm j}^n$, for each $i$ we let
\begin{equation*}
  \statprob_{+j,i} = \frac{1}{2^{d-1}} \sum_{\packval : \packval_j = 1}
  \statprob_{\packval,i}
  ~~ \mbox{and} ~~
  \statprob_{-j,i} = \frac{1}{2^{d-1}} \sum_{\packval : \packval_j = -1}
  \statprob_{\packval,i},
\end{equation*}
and, in analogy to the marginal channel~\eqref{eqn:marginal-channel}, we
define the marginal mixtures
\begin{equation*}
  \marginprob_{+j}^n(S) \defeq 
  \frac{1}{2^{d - 1}} \sum_{\packval : \packval_j = 1}
  \marginprob_\packval^n(S)
  = \int \channelprob^n(S \mid \statsample_{1:n})
  d \statprob_{+j}^n(\statsample_1, \ldots, \statsample_n)
  ~~~ \mbox{for~} j = 1, \ldots, d,
\end{equation*}
with the distributions $\marginprob_{-j}^n$ defined similarly.  For a given
pair of distributions $(M, M')$, we let \mbox{$\dklsym{M}{M'} = \dkl{M}{M'}
  + \dkl{M'}{M}$} denote the symmetrized KL-divergence.  Recalling the
supremum norm ball $\linfset(\statdomain) = \{f : \statdomain \to \R \mid
\linf{f} \le 1\}$ of Eq.~\eqref{eqn:L-infty-set}, we have the following
theorem.
\begin{theorem}
  \label{theorem:sequential-interactive}
  Under the conditions of the previous paragraph, for any
  $\diffp$-locally differentially private
  channel $\channel$~\eqref{eqn:local-privacy}, we have
  \begin{align*}
    \sum_{j=1}^d \dklsym{\marginprob_{+j}^n}{\marginprob_{-j}^n} & \le
    2 (e^\diffp - 1)^2 \sum_{i=1}^n
    \sup_{\optdens \in
      \linfset(\statdomain)} \sum_{j=1}^d \left(\int_{\statdomain}
    \optdens(\statsample)
    \left(d\statprob_{+j,i}(\statsample) -
    d\statprob_{-j,i}(\statsample)\right) \right)^2.
  \end{align*}
\end{theorem}
\noindent
Theorem~\ref{theorem:sequential-interactive} generalizes
Theorem~\ref{theorem:master}, which corresponds to the special case $d = 1$,
though it also has parallels with Theorem~\ref{theorem:super-master}, as
taking the supremum outside the summation is essential to obtaining sharp
results. We provide the proof of
Theorem~\ref{theorem:sequential-interactive} in
Appendix~\ref{sec:proof-sequential-interactive}.

Theorem~\ref{theorem:sequential-interactive} allows us to prove
Proposition~\ref{proposition:private-assouad} from the non-private
variant~\eqref{eqn:sharp-assouad} of Assouad's method. Indeed,
inequality~\eqref{eqn:sharp-assouad} immediately implies for any channel
$\channel$ that
\begin{equation*}
  \minimax_n(\optvar(\pclass), \bigloss \circ \metric, \channel)
  \ge \delta \sum_{j = 1}^d
  \left[1 - \tvnorm{\marginprob_{+j}^n - \marginprob_{-j}^n}\right].
\end{equation*}
Using a combination of Pinsker's inequality and Cauchy-Schwarz, we obtain
\begin{equation*}
  \sum_{j=1}^d \tvnorm{\marginprob_{+j}^n - \marginprob_{-j}^n}
  \le \half \sqrt{d} \bigg(\sum_{j=1}^d
  \dkl{\marginprob_{+j}^n}{\marginprob_{-j}^n}
  + \dkl{\marginprob_{-j}^n}{\marginprob_{+j}^n}\bigg)^\half.
\end{equation*}
Thus, whenever $\statprob_\packval$ induces a $2\delta$-Hamming separation
for $\Phi \circ \metric$ we have
\begin{equation}
  \minimax_n(\optvar(\mc{\statprob}), \Phi \circ \metric, \channel) \ge d \delta
  \Bigg[1 - \bigg(\frac{1}{4d}
    \sum_{j=1}^d\dklsym{\marginprob_{+j}^n}{\marginprob_{-j}^n}
    \bigg)^\half \Bigg].
  \label{eqn:sharp-assouad-kled}
\end{equation}
The combination of inequality~\eqref{eqn:sharp-assouad-kled} with
Theorem~\ref{theorem:sequential-interactive} yields
Proposition~\ref{proposition:private-assouad}.


\section{Experiments}
\label{SecExperiments}

In this section, we study three different datasets---each consisting
of data that are sensitive but are nonetheless available publicly---in
an effort to demonstrate the importance of minimax theory for practical 
private estimation. While the public availability of these datasets in 
some sense obviates the need for private analysis, they provide natural
proxies by which to evaluate the performance of privacy-preserving
mechanisms and estimation schemes. (The public availability also
allows us to make all of our experimental data freely available.)

\subsection{Salary estimation: experiments with one-dimensional data}
\label{sec:location-experiments}

Our first set of experiments investigates the performance of minimax optimal
estimators of one-dimensional mean and median statistics, as described in 
Sections~\ref{sec:location-family} and~\ref{sec:median-estimation}.  We use 
data on salaries in the University of California system~\cite{BerkeleyPayroll14}.
We perform our experiments with the 2010 salaries, which consists of a
population of $N = 252,\!540$ employees with mean salary \$39,531 and median
salary \$24,968. The data has reasonably long tails; 14 members of the
dataset have salaries above \$1,000,000 and two individuals have salaries
between two- and three-million dollars.

\subsubsection{Mean salary estimation}
\label{sec:mean-experiments}

\begin{figure}
  \begin{center}
    \begin{tabular}{cc}
      \begin{overpic}[width=.45\columnwidth]
        {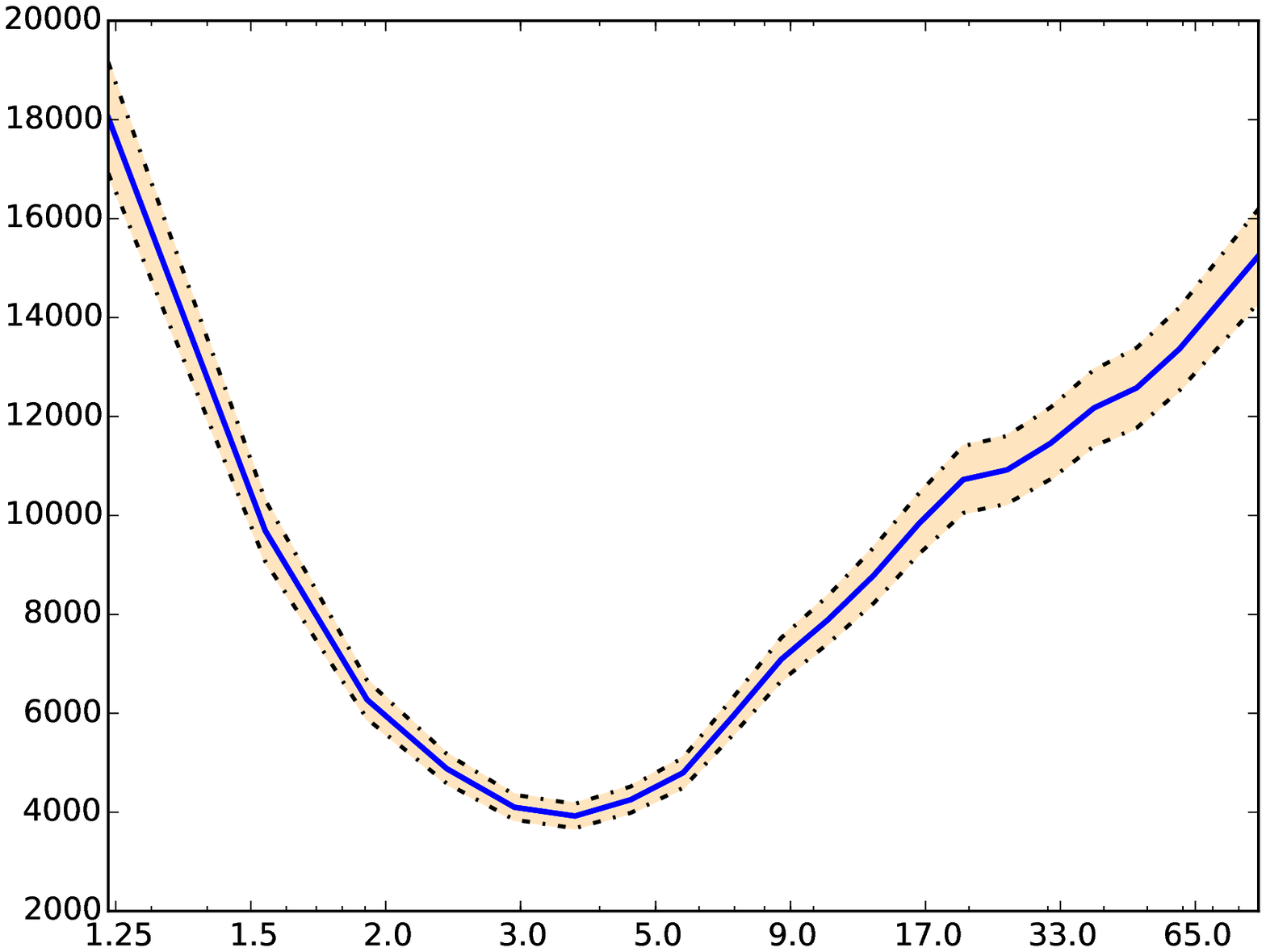}
        \put(-3,30){\scriptsize{\rotatebox{90}{$\E[|\what{\optvar} - \optvar|]$}}}
        \put(40,-1){\tiny{Assumed moment $k$}}
      \end{overpic} &
      \begin{overpic}[width=.45\columnwidth]
        {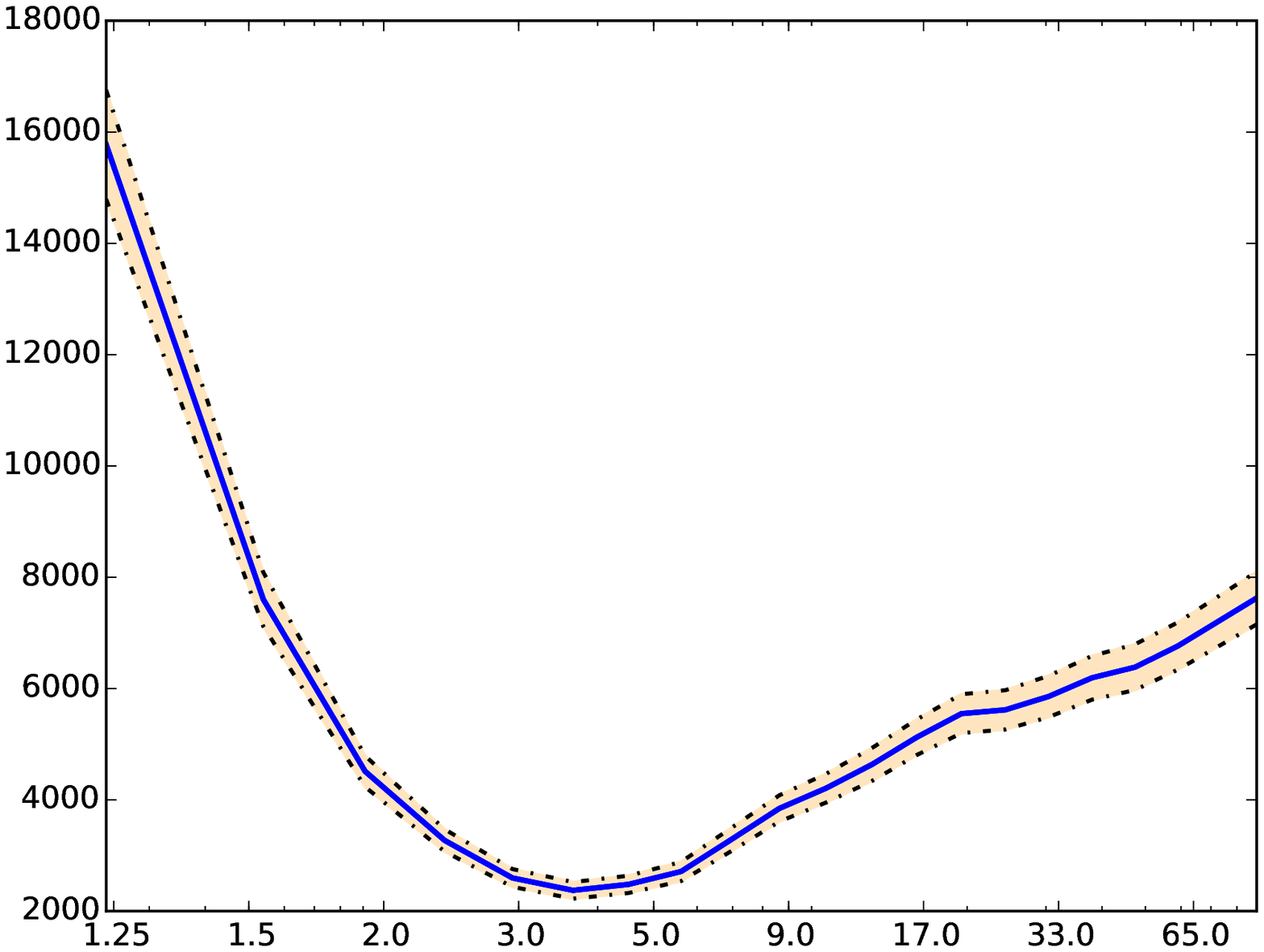}
        \put(-3,30){\scriptsize{\rotatebox{90}{$\E[|\what{\optvar} - \optvar|]$}}}
        \put(40,-1){\tiny{Assumed moment $k$}}
      \end{overpic} \\
      (a) & (b)
    \end{tabular}
    \caption{\label{fig:mean-salary} Mean salary estimation errors.
      Left (a): privacy $\diffp = 1/2$.  Right (b): privacy $\diffp = 1$.
      The horizontal axis corresponds to
      known assumed moment power $k$, that is, $\E[|X|^k]^{1/k} =
      \radius_k$, and vertical axis is mean absolute error. Confidence bands
      are simulated 90\% confidence intervals.}
  \end{center}
\end{figure}

We first explore the effect of bounds on moments in the mean estimation 
problem as discussed in Section~\ref{sec:location-family}. Recall that for the
family of distributions $\statprob$ satisfying $\E_P[|X|^k]^{1/k} \le
\radius_k$, a minimax optimal estimator is the truncated mean
estimator~\eqref{eqn:truncated-mean-estimator}
(cf.\ Sec.~\ref{sec:proof-location-family}) with truncation level $T_k =
\radius_k (n \diffp^2)^{\frac{1}{2k}}$; then $\what{\theta} = \frac{1}{n}
\sum_{i=1}^n \channelrv_i$ has convergence rate $\E[(\what{\theta} -
  \E[X])^2] \lesssim \radius_k \left(n \diffp^2\right)^{-\frac{k - 1}{k}}$.
For heavy-tailed data, it may be the case that $\E[|X|^{k_1}]^{1/k_1} \ll
\E[|X|^{k_2}]^{1/k_2}$ for some $k_1 < k_2$; thus, while the rate
$n^{-\frac{k-1}{k}}$ may be slower for $k = k_1$ than $k = k_2$, the leading
moment-based term $\radius_{k_1}$ may yield better finite sample
performance. We investigate this possibility in several experiments.

We perform each experiment by randomly subsampling one-half of the dataset
to generate samples $\{X_1, \ldots, X_{N/2}\}$, whose means we estimate
privately using the estimator~\eqref{eqn:truncated-mean-estimator}. Within
an experiment, we assume we have knowledge of the population
moment $\radius_k = \E[|X|^k]^{1/k}$ for large enough $k$---an unrealistic
assumption facilitating comparison. Fixing $m = 20$,
we use the $m - 1$ logarithmically-spaced powers
\begin{equation*}
  k \in \left\{10^{\frac{j}{m - 1} \log_{10}(3m)} : 
  j = 1, 2, \ldots, m-1 \right\}
  \approx \{1.24, 1.54, 1.91, \ldots, 60 \}
  ~~ \mbox{and} ~~ k = \infty.
\end{equation*}
We repeat the experiment 400 times for each moment value $k$ and privacy
parameters $\diffp \in \{.1, .5, 1, 2\}$.

\begin{table}
  \hspace{-.5cm}
  \begin{tabular}{lr}
    \begin{minipage}{.35\columnwidth}
      \caption{\label{table:mean-errors} Optimal mean absolute errors
        for each privacy level $\diffp$, as determined by mean
        absolute error over assumed known moment $k$.}
    \end{minipage} &
    \begin{minipage}{.64\columnwidth}
      \begin{tabular}{|c|c|c|c|c|c|}
        \hline
        Privacy $\diffp$
        & $\nicefrac{1}{10}$ & $\nicefrac{1}{2}$ & $1$ & $2$ & $\infty$ \\
        \hline \hline
        Mean absolute error & 11849 & 3923 & 2373 & 1439 & 82 \\
        Moment $k$ & 2.9 & 3.6 & 3.6 & 3.6 & N/A \\
        Standard error & 428 & 150 & 91 & 55 & 6.0 \\
        \hline
      \end{tabular}
    \end{minipage}
  \end{tabular}
\end{table}

In Figure~\ref{fig:mean-salary}, we present the behavior of the truncated
mean estimator for estimating the mean salary paid in fiscal year 2010 in
the UC system. We plot the mean absolute error of the private estimator for
the population mean against the moment $k$ for the experiments with $\diffp
= \half$ and $\diffp = 1$.  We see that for this (somewhat) heavy-tailed
data, an appropriate choice of moment---in this case, about $k \approx
3$---can yield substantial improvements. In particular, assumptions of data
boundedness $(k = \infty)$ give unnecessarily high variance and large radii
$\radius_k$, while assuming too few moments ($k < 2$) yields a slow convergence 
rate $n^{-\frac{k-1}{k}}$ that causes absolute errors in estimation
that are too large.  In Table~\ref{table:mean-errors}, for each privacy
level $\diffp \in \{\frac{1}{10}, \half, 1, 2, \infty\}$ (with $\diffp =
\infty$ corresponding to no privacy) we tabulate the best mean absolute
error achieved by any moment $k$ (and the moment achieving this error). The
table shows that, even performing a \emph{post-hoc} optimal choice of moment
estimator $k$, local differential privacy may be quite a strong constraint
on the quality of estimation, as even with $\diffp = 1$ we incur an error
of approximately 6\% on average, while a non-private estimator observing
half the population has error of .2\%.

\subsubsection{Median salary estimation}
\label{sec:median-experiment}

We now turn to an evaluation of the performance of the minimax optimal
stochastic gradient procedure~\eqref{eqn:median-sgd} for median estimation,
comparing it with more naive but natural estimators based on noisy truncated
versions of individuals' data.  We first motivate the alternative strategy
using the problem setting of Section~\ref{sec:median-estimation}. Recall
that $\project_{[-\radius, \radius]}(\statval)$ denotes the projection of
$\statval \in \R$ onto the interval $[-\radius, \radius]$.  If by some prior
knowledge we know that the distribution $P$ satisfies $|\median(P)| \le
\radius$, then for $\statrv \sim P$ the random variables
$\project_{[-\radius, \radius]}(\statrv)$ have identical median to $P$, and
for any symmetric random variable $W$, the variable $\channelrv = W +
\project_{[-\radius, \radius]}(\statrv)$ satisfies $\median(\channelrv) =
\median(\statrv)$. Now, let $W_i \simiid \laplace(\diffp / (2 \radius))$,
and consider the natural $\diffp$-locally differentially private estimator
\begin{equation}
  \channelrv_i = \project_{[-\radius, \radius]}(\statrv_i) + W_i,
  ~~~\mbox{with} ~~~
  \what{\theta}_n = \median(\channelrv_1, \ldots, \channelrv_n).
  \label{eqn:naive-median}
\end{equation}
The variables $\channelrv_i$ are locally private versions of the
$\statrv_i$, as we simply add noise to (truncated) versions of the true data
$\statrv_i$. Rather than giving a careful theoretical investigation of the
performance of the estimator~\eqref{eqn:naive-median}, we turn to
empirical results.
\begin{figure}[t!]
  \begin{center}
    \begin{tabular}{ccc}
      \begin{overpic}[width=.32\columnwidth]
        {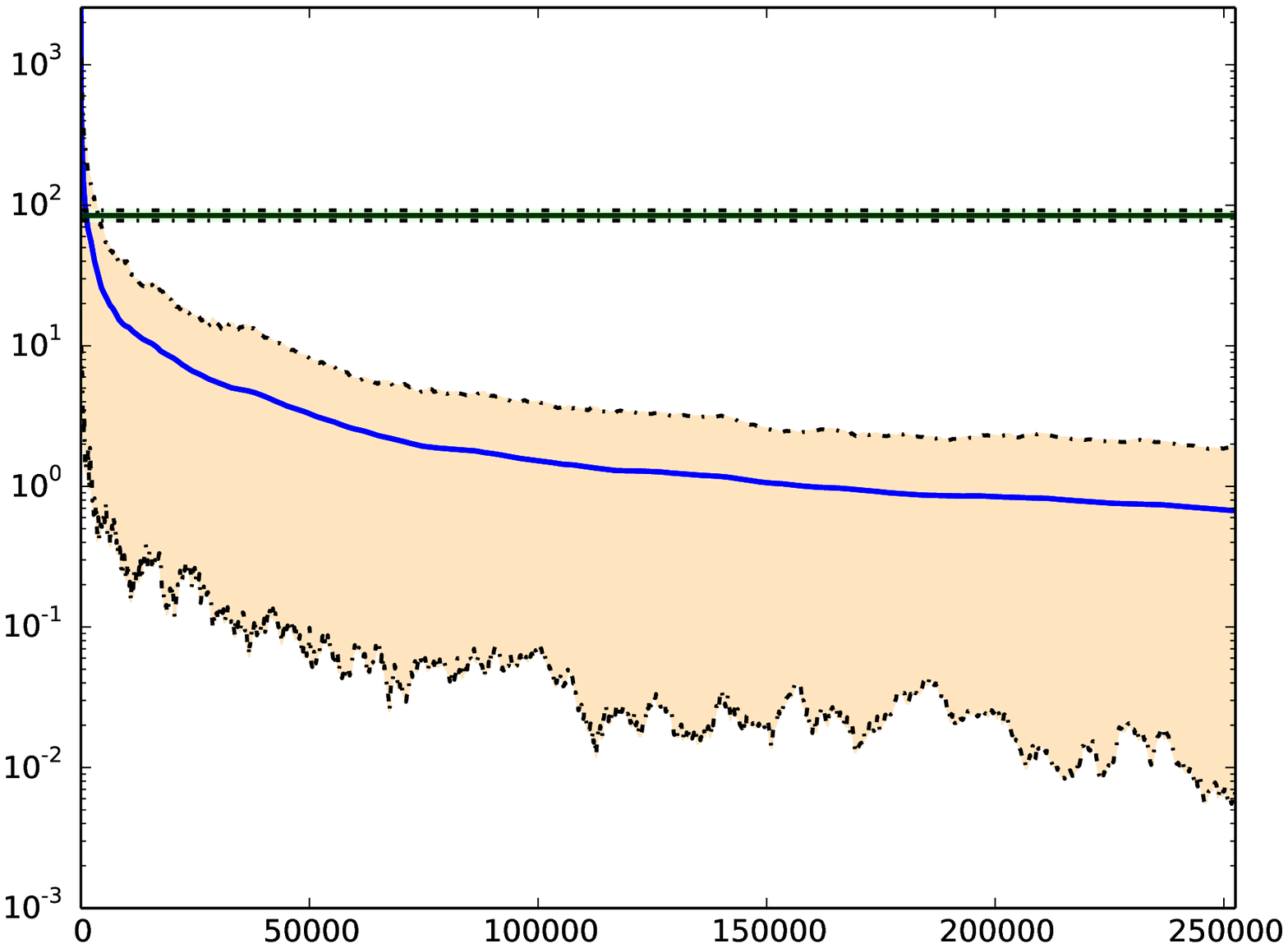}
        \put(-3,30){\tiny{\rotatebox{90}{$|\what{\optvar} - \optvar|$}}}
        \put(50,.5){\tiny{$n$}}
      \end{overpic} &
      \begin{overpic}[width=.32\columnwidth]
        {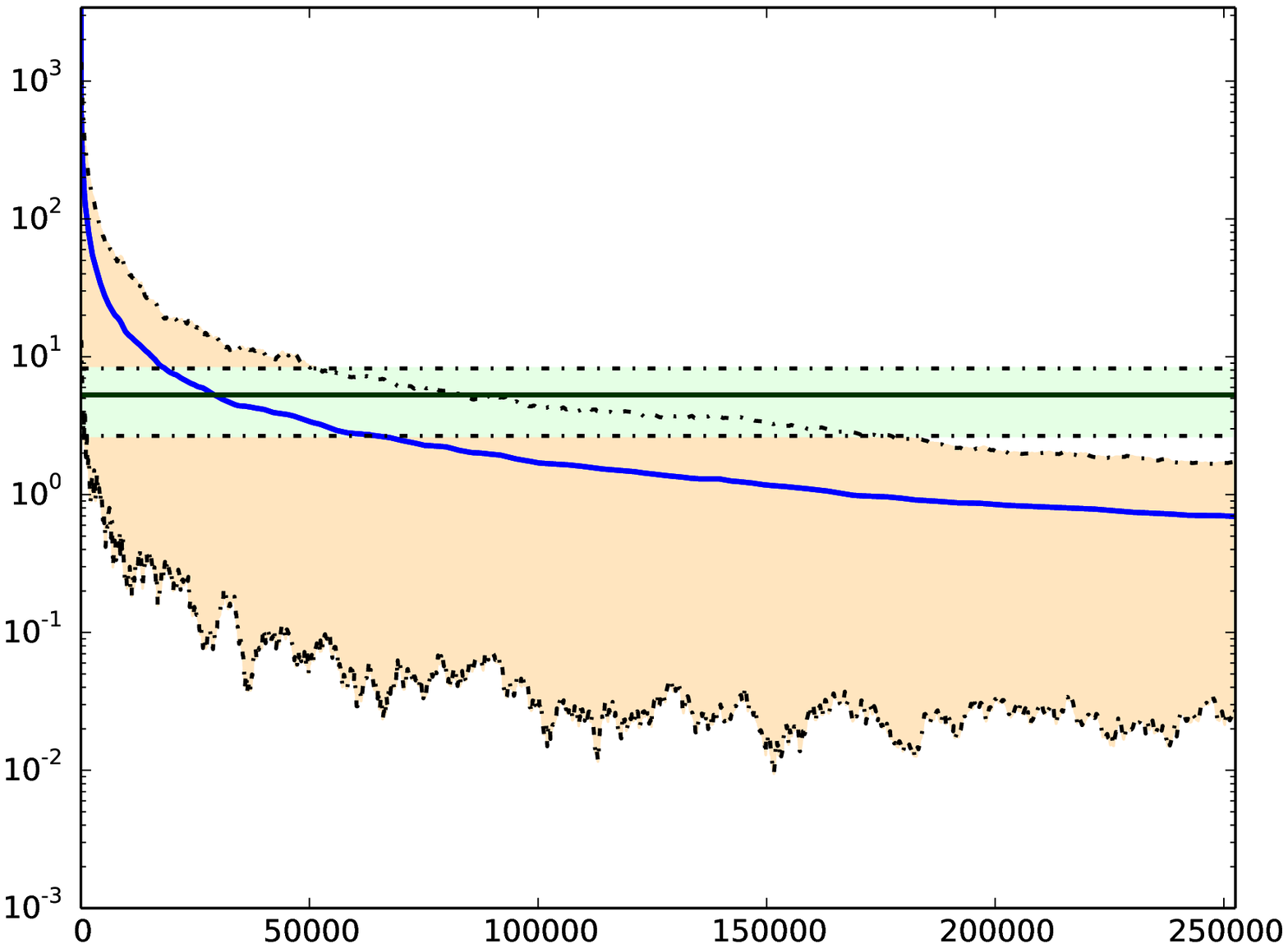}
        \put(-3,30){\tiny{\rotatebox{90}{$|\what{\optvar} - \optvar|$}}}
        \put(50,.5){\tiny{$n$}}
      \end{overpic} &
      \begin{overpic}[width=.32\columnwidth]
        {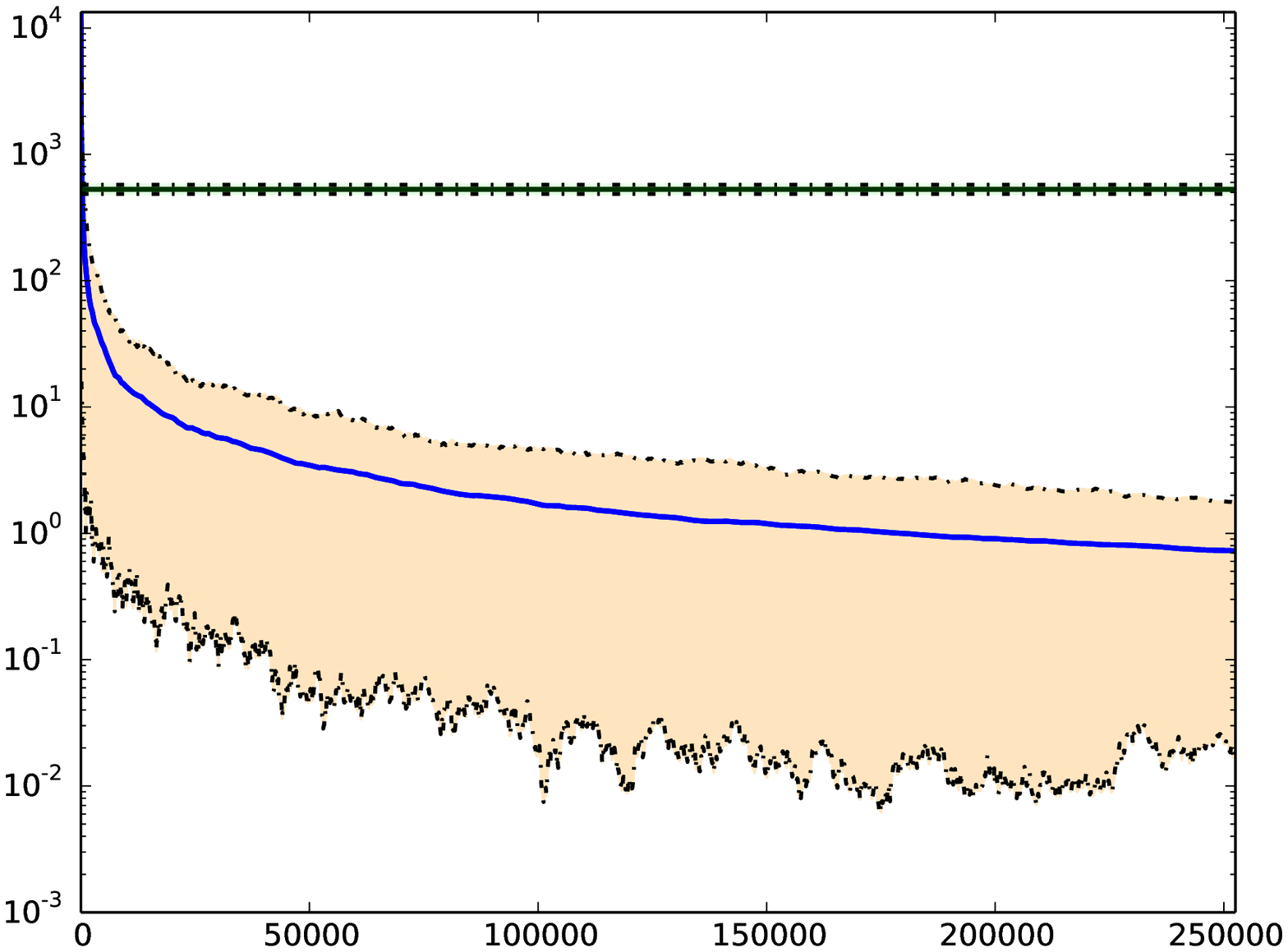}
        \put(-3,30){\tiny{\rotatebox{90}{$|\what{\optvar} - \optvar|$}}}
        \put(50,.5){\tiny{$n$}}
      \end{overpic} \\
      $\radius = 1.5 \median(P)$ & $\radius = 2 \median(P)$
      & $\radius = 4 \median(P)$
    \end{tabular}
    \begin{tabular}{cc}
      \begin{overpic}[width=.32\columnwidth]
        {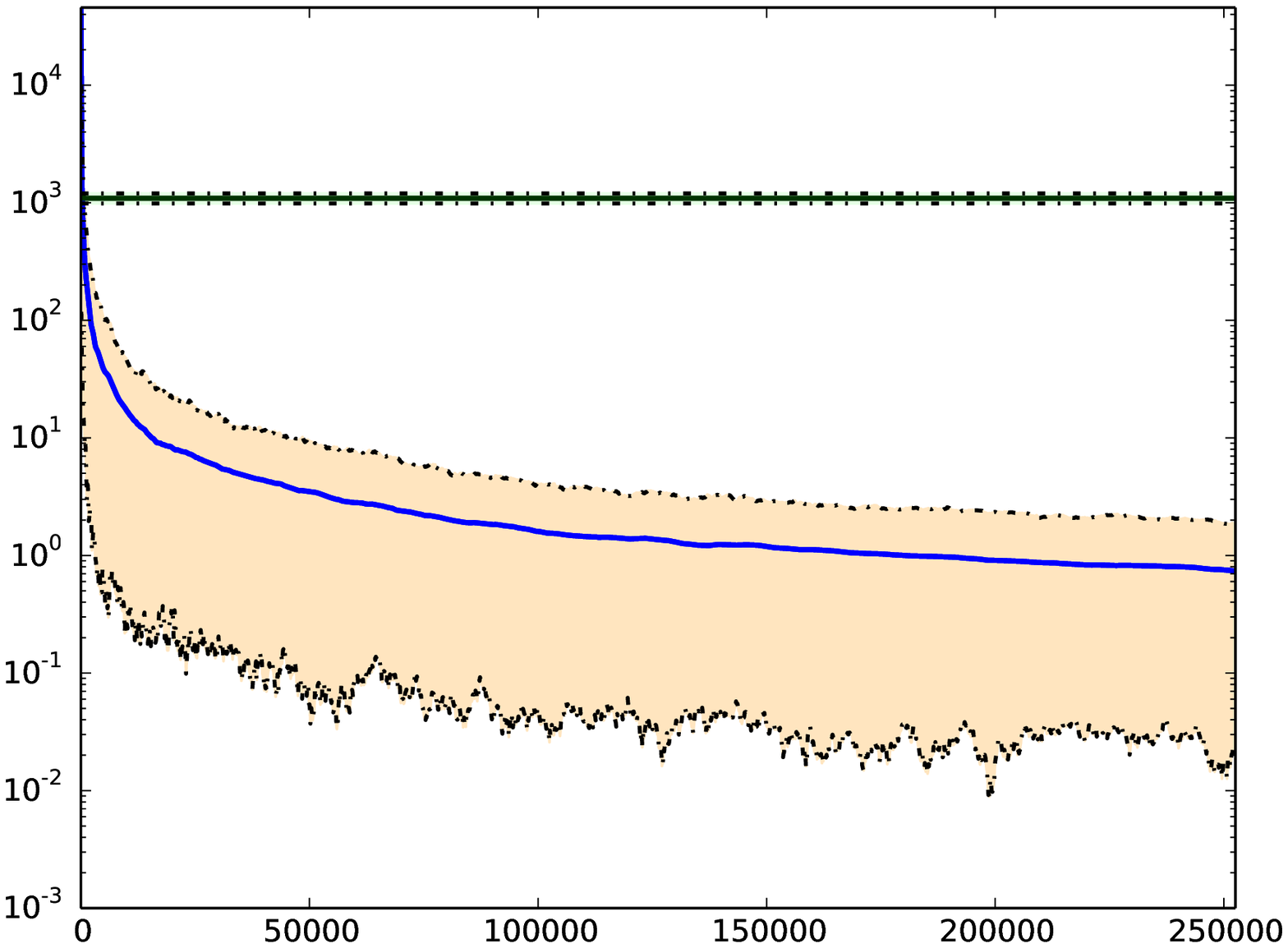}
        \put(-3,30){\tiny{\rotatebox{90}{$|\what{\optvar} - \optvar|$}}}
        \put(50,.5){\tiny{$n$}}
      \end{overpic} &
      \begin{overpic}[width=.32\columnwidth]
        {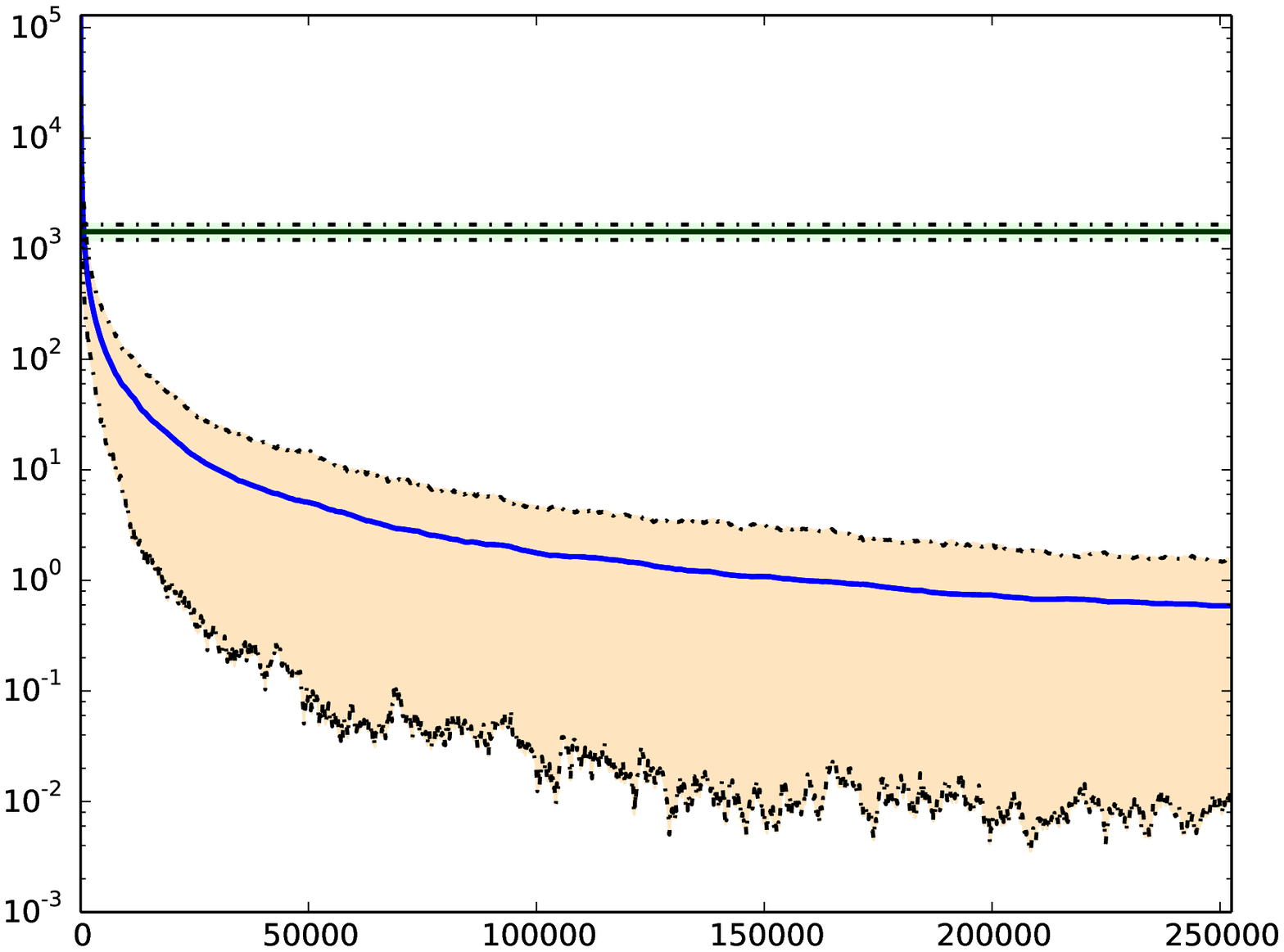}
        \put(-3,30){\tiny{\rotatebox{90}{$|\what{\optvar} - \optvar|$}}}
        \put(50,.5){\tiny{$n$}}
      \end{overpic} \\
      $\radius = 8 \median(P)$ & $\radius = 16 \median(P)$
    \end{tabular}
    \caption{\label{fig:median-salary} Performance of two median
      estimators---the SGD estimator~\eqref{eqn:median-sgd} and the naive
      estimator~\eqref{eqn:naive-median}---plotted against sample size for
      $\diffp = 1$ for privately estimating median salary in the UC
      system. The blue line (lower) gives the estimation error of SGD
      against iteration (number of observations) with 90\% confidence
      interval, the black (upper) line that of the naive median
      estimator~\eqref{eqn:naive-median} with 90\%
      confidence interval in with green shading (using the full
      sample). Each plot corresponds to a different estimated maximum radius
      $\radius$ for the true median.}
  \end{center}
\end{figure}

We again consider the salary data available from the UC system for fiscal
year 2010, where we know that the median salary is at least zero (so we
replace projections onto $[-\radius, \radius]$ with $[0, \radius]$ in the
methods~\eqref{eqn:median-sgd} and~\eqref{eqn:naive-median}).  We compare
the stochastic gradient estimator~\eqref{eqn:median-sgd} using the specified
stepsizes $\stepsize_i$ and averaged predictor $\what{\theta}_n$ as well as
the naive estimator~\eqref{eqn:naive-median}.  Treating the full population
(all the salaries) as the sampling distribution $P$, we perform five
experiments, each for a different estimate value of the radius of the median
$\radius$. Specifically, we use $\radius \in \{1.5, 2, 4, 8, 16\}
\median(P)$, where we assume that we know that the median must lie in some region
near the true median. (The true median salary is approximately \$25,000.)
For each experiment, we perform the following steps 400 times (with
independent randomness) on the population of size $n = 252,\!540$:
\begin{enumerate}[1.]
\item We perform $n$ steps of private stochastic gradient
  descent~\eqref{eqn:median-sgd}, beginning from a uniformly random
  intialization $\theta_0 \in [0, \radius]$. Each step consists of a random
  draw $\statrv_i$ without replacement from the population
  $\{\statrv_1, \ldots, \statrv_n\}$, then performing the private
  gradient step.
\item We compute the naive private median~\eqref{eqn:naive-median} with
  $W_i \simiid \laplace(\diffp / (2\radius))$.
\end{enumerate}

In Figure~\ref{fig:median-salary}, we plot the results of these
experiments.  Each plot contains two lines. The first is a descending
blue line (the lower of the two), which is the running gap
$\risk(\what{\theta}_t) - \risk(\median(P))$ of the mean SGD estimator
$\what{\theta}_t = \frac{1}{t} \sum_{i=1}^t \theta_i$ for each $t = 1,
\ldots, n$; the gap is averaged over the $400$ runs of SGD. We shade
the region between the 5th and 95th percentile of the gaps
$\risk(\what{\theta}_t) - \risk(\median(P))$ for the estimator over
all 400 runs of SGD; note that the upward deviations are fairly
tight. The flat line is the mean performance $\risk(\what{\theta}) -
\risk(\median(P))$ of the naive estimator~\eqref{eqn:naive-median}
over all 400 tests with 90\% confidence bands above and below (this
estimator uses the full sample of size $n$). Two aspects of these
results are notable.  First, the minimax optimal SGD estimator always
has substantially better performance---even in the best case for the
naive estimator (when $\radius = 2 \median(P)$), the gap between the
two is approximately a factor of 6---and usually gives a several order
of magnitude improvement. Secondly, the optimal SGD estimator is quite
robust, yielding mean performance $\E[\risk(\what{\theta}_n)] -
\risk(\median(P)) < 1$ for all experiments. By way of comparison, the
non-private stochastic gradient estimator---that with $\alpha =
+\infty$---attains error approximately $0.25$. More precisely, as our
problem has median salary approximately \$25,000, the relative error
in the gap, no matter which radius $\radius$ is chosen, is at most $4
\cdot 10^{-5}$ for private SGD.

\subsection{Drug use and hospital admissions}
\label{sec:drug-use}

In this section, we study the problem of estimating the proportions of
a population admitted to hospital emergency rooms for different types
of drug use; it is natural that admitted persons might want to keep
their drug use private, but accurate accounting of drug use is helpful
for assessment of public health. We apply our methods to drug use
data from the National Estimates of Drug-Related Emergency Department
Visits (NEDREDV)~\cite{HHS13} and treat this problem as a
mean-estimation problem.  First, we describe our data-generation
procedure, after which we describe our experiments in detail.

The NEDREDV data consists of tabulated triples of the form $({\rm
  drug}, {\rm year}, m)$, where $m$ is a count of the number of
hospital admissions for patients using the given drug in the given
year. We take admissions data from the year $2004$, which consists of
959,715 emergency department visits, and we include admissions for $d
= 27$ common drugs.\footnote{The drugs are Alcohol, Cocaine, Heroin,
  Marijuana, Stimulants, Amphetamines, Methamphetamine, MDMA
  (Ecstasy), LSD, PCP, Antidepressants, Antipsychotics, Miscellaneous
  hallucinogens, Inhalants, lithium, Opiates, Opiates unspecified,
  Narcotic analgesics, Buprenorphine, Codeine, Fentanyl, Hydrocodone,
  Methadone, Morphine, Oxycodone, Ibuprofen, Muscle relaxants.}  Given
these tuples, we generate a random dataset $\{X_1, \ldots, X_N\}
\subset \{0,1\}^d$, $N = 959,\!715$, with the property that for each
drug $j \in \{1, \ldots, d\}$ the marginal counts $\sum_{i=1}^N
X_{ij}$ yield the correct drug use frequencies.  Under this marginal
constraint, we set each coordinate $X_{ij} = 0$ or $1$ independently
and uniformly at random.  Thus, each non-private observation consists
of a vector $X \in \{0, 1\}^d$ representing a hospital admission,
where coordinate $j$ of $X$ is $1$ if the admittee abuses drug $j$ and
$0$ otherwise. Many admittees are users of multiple drugs (this is
true in the non-simulated data as there are substantially more drug
counts than total admissions), so we consider the problem of
estimating the mean $\frac{1}{N} \sum_{i=1}^N X_i$ of the population,
$\theta = \E[X]$, where all we know is that $X \in [0,1]^d$.

In each separate experiment, we draw a random sample of size $n =
\lceil \frac{2N}{3} \rceil$ from the population, replacing each
element $X_i$ of the sample with an $\diffp$-locally differentially
private view $Z_i$, and then construct an estimate $\what{\theta}$ of
the true mean $\theta$.  In this case, the minimax optimal
$\diffp$-differentially private strategy is based on
$\ell_\infty$-sampling strategy~\eqref{eqn:linf-sampling}.  We also
compare to a naive estimator that uses $\channelrv = X + W$, where $W
\in \R^d$ has independent coordinates each with $\laplace(\diffp / d)$
distribution, as well as a non-private estimator using the average of
the subsampled vectors $X$.


We displayed our results earlier in
Figure~\ref{fig:linf-mean-estimation} in the introduction, where we
plot the results of 100 independent experiments with privacy parameter
$\diffp = \half$. We show the mean $\ell_\infty$ error for estimating
the population proportions $\theta = \frac{1}{N} \sum_{i=1}^N
X_i$---based on a population of size $N = 959,\!715$---using a sample of
size $n$ as $n$ ranges from $1$ to $n = 600000$.  We consider
estimation of $d = 27$ drugs. The top-most (blue) line corresponds to
the Laplace estimator, the bottom (black) line a non-private estimator
based on empirical counts, and the middle (green) line the optimal
private estimator. The plot also shows 5th and 95th percentile
quantiles for each of the private experiments. From the figure, it is
clear that the minimax optimal sampling strategy outperforms the
equally private Laplace noise addition mechanism; even the worst
performing random samples of the optimal sampling scheme outperform
the best of the Laplace noise addition scheme. The mean error of the
optimal scheme is also roughly a factor of $\sqrt{d} \approx 5$ better
than the non-optimal Laplace scheme.

\subsection{Censorship, privacy, and logistic regression}
\label{sec:censorship}

In our final set of experiments, we investigate private estimation
strategies for conditional probability estimation---logistic
regression---for prediction of whether a document will be censored. We
applied our methods to a collection of $N = 190,\!000$ Chinese blog posts,
of which $\ncensored = 90,\!000$ have been censored and $\nuncensored =
100,\!000$ have been allowed to remain on Weibao (a Chinese blogging
platform) by Chinese authorities.\footnote{ We use data identical to that
  used in the articles~\cite{KingPaRo13,KingPaRo14}. The datasets were
  constructed as follows: all blog posts from
  \url{http://weiboscope.jmsc.hku.hk/datazip/} were downloaded (see
  \citet{FuChCh13}), and the Chinese text of each post is
  segmented using the
  Stanford Chinese language parser~\cite{LevyMa03}. Of these, a random
  subsample of $\ncensored$ censored blog posts and $\nuncensored$
  uncensored posts is taken.}  The goal is to find those words strongly
correlated with censorship decisions by estimation of a logistic model and
to predict whether a particular document will be censored.  We let $\statval
\in \{0, 1\}^d$ be a vector of variables representing a single document,
where $\statval_j = 1$ indicates that word $j$ appears in the document and
$\statval_j = 0$ otherwise. Then the task is to estimate the logistic model
$P(Y = y \mid X = x; \optvar) = 1 / (1 + \exp(-y \<x, \theta\>))$ for $y \in
\{-1, 1\}$ and $x \in \{0, 1\}^d$.

As the initial dimension is too large for private strategies to be
effective, we perform and compare results over two experiments.  In
the first, we use the $d = 458$ words appearing in at least 0.5\% of
the documents, and in the second we use the $d = 24$ words appearing
in at least 10\% of the documents. We repeat the following experiment
25 times with privacy parameter $\diffp \in \{1, 2, 4\}$. First, we
draw a subsample of $n = \lceil 0.75 N \rceil$ random documents, on
which we fit a logistic regression model using either (i) no privacy,
(ii) the minimax optimal stochastic gradient
scheme~\eqref{eqn:sgd-update} with optimal
$\ell_2$-sampling~\eqref{eqn:ltwo-sampling}, or (iii) the stochastic
gradient scheme~\eqref{eqn:sgd-update}, where the stochastic gradients
are perturbed by mean-zero independent Laplace noise sufficient to
guarantee $\diffp$-local-differential privacy (in the two stochastic
gradient cases, we present the examples in the same order). We then
evaluate the performance of the fit vector $\what{\theta}$ on the
remaining held-out 25\% of the data.  For numerical stability reasons,
we project our stochastic gradient iterates $\theta_i$ onto the
$\ell_2$-ball of radius $5$; this has no effect on the convergence
guarantees given in Lemma~\ref{lemma:polyak-juditsky} for stochastic
gradient descent, because the non-private solution for the logistic
regression problem on each of the samples satisfies the norm bound
$\norm{\theta} < 5$.

\begin{figure}
  \begin{tabular}{cc}
    \hspace{-.6cm}
    \begin{minipage}{.5\columnwidth}
      \begin{tabular}{|c|c|c|c|}
        \hline
        $\diffp$ & Non-private & Optimal & Laplace \\
        \hline \hline
        1 & $0.256 \pm 0.0$ & $0.443 \pm 0.004$ & $0.5 \pm 0.005$ \\
        \hline
        2 & $0.255 \pm 0.0$ & $0.43 \pm 0.003$ & $0.5 \pm 0.006$ \\
        \hline
        4 & $0.256 \pm 0.0$ & $0.409 \pm 0.003$ & $0.486 \pm 0.007$ \\
        \hline
      \end{tabular}
    \end{minipage} &
    \begin{minipage}{.5\columnwidth}
      \begin{tabular}{|c|c|c|c|}
        \hline
        $\diffp$ & Non-private & Optimal & Laplace \\
        \hline \hline
        1 & $0.35 \pm 0.001$ & $0.431 \pm 0.005$ & $0.5 \pm 0.009$ \\
        \hline
        2 & $0.35 \pm 0.001$ & $0.406 \pm 0.005$ & $0.483 \pm 0.008$ \\
        \hline
        4 & $0.35 \pm 0.001$ & $0.4 \pm 0.004$ & $0.449 \pm 0.005$ \\
        \hline
      \end{tabular}
    \end{minipage}
    \\
    (a) Test error using top .5\% of words & (b)
    Test error using top 10\% of words
  \end{tabular}
  \caption{\label{fig:logistic-experiment}
    Logistic regression experiment. Tables include mean test
    (held-out) error of different privatization schemes for
    privacy levels $\diffp \in \{1, 2, 4\}$, averaged over
    25 experimental runs using random held-out sets of
    size $N/4$ of the data. We indicate standard errors by
    the $\pm$ terms.}
\end{figure}

Figure~\ref{fig:logistic-experiment} provides a summary of our
results: it displays the mean test (held-out) error rate over the 25
experiments, along with the standard error over the experiments. The
tables show, perhaps most importantly, that there is a non-trivial
degradation in classification quality as a consequence of privacy. The
degradation in Figure~\ref{fig:logistic-experiment}(a), when the
dimension $d \approx 450$, is more substantial than in the lower $d =
24$-dimensional case (Figure~\ref{fig:logistic-experiment}(b)).  The
classification error rate of the Laplace mechanism is essentially
random guessing in the higher-dimensional case, while for the minimax
optimal $\ell_2$-mechanism, the classification error rate is more or
less identical for both the high and low-dimensional problems, in
spite of the substantially better performance of the non-private
estimator in the higher-dimensional problem. In our experiments, in
the high-dimensional case, the Laplace mechanism had better test error
rate than the optimal randomized-response-style scheme in three of the
tests with $\diffp = 1$, one test with $\diffp = 2$, and two with $\diffp = 4$;
for the $d = 24$-dimensional case, the Laplace scheme outperformed the
randomized response scheme in two, one, and zero experiments
for $\diffp \in \{1, 2, 4\}$, respectively.  Part of this
difference is explainable by the size of the tails of the privatizing
distributions: our optimal sampling schemes~\eqref{eqn:ltwo-sampling}
and~\eqref{eqn:linf-sampling} in Sec.~\ref{sec:attainability-means}
have compact support, and are thus sub-Gaussian, while the Laplace
distribution has heavier tails.
In sum, it is clear that the $\ell_2$-optimal randomized
response strategy dominates the more naive Laplace noise addition strategy,
while non-private estimation enjoys improvements over both.

\section{Conclusions}

The main contribution of this paper is to link minimax analysis from
statistical decision theory with differential privacy, bringing some
of their respective foundational principles into close contact. Our
main technique, in the form of the divergence inequalities in
Theorems~\ref{theorem:master}, \ref{theorem:super-master},
and~\ref{theorem:sequential-interactive}, and their associated
corollaries, shows that applying differentially private sampling
schemes essentially acts as a contraction on distributions.  These
contractive inequalities allow us to obtain results, presented in
Propositions~\ref{PropPrivateLeCam}, \ref{PropPrivateFano},
and~\ref{proposition:private-assouad}, that generalize the classical
minimax lower bounding techniques of Le Cam, Fano, and Assouad.  These
results allow us to give sharp minimax rates for estimation in locally
private settings.  With our examples in
Sections~\ref{sec:mean-estimation}, \ref{sec:glm-estimation},
and~\ref{sec:density-estimation}, we have developed a framework that
shows that, roughly, if one can construct a family of distributions
$\{\statprob_\packval\}$ on the sample space $\statdomain$ that is not
well ``correlated'' with any member of $f \in L^\infty(\statdomain)$
for which $f(\statsample) \in \{-1, 1\}$, then providing privacy is
costly: the contraction provided in
Theorems~\ref{theorem:super-master}
and~\ref{theorem:sequential-interactive} is strong.

By providing sharp convergence rates for many standard statistical
estimation procedures under local differential privacy, we have
developed and explored some tools that may be used to better
understand privacy-preserving statistical inference. We have
identified a fundamental continuum along which privacy may be traded
for utility in the form of accurate statistical estimates, providing a
way to adjust statistical procedures to meet the privacy or utility
needs of the statistician and the population being sampled. 

There are a number of open questions raised by our work. It is natural to
wonder whether it is possible to obtain tensorized inequalities of the form
of Corollary~\ref{corollary:idiot-fano} even for interactive mechanisms. One
avenue of attack for such an approach could be the work on directed
information, which is useful for understanding communication over channels
with feedback~\cite{Massey90, PermuterKiWe11}. Another important question is
whether the results we have provided can be extended to settings in which
standard (non-local) differential privacy, or another form of disclosure
limitation, holds. Such extensions could yield insights into optimal
mechanisms for a number of private procedures.

Finally, we wish to emphasize the pessimistic nature of several of our
results.  The strengths of differential privacy as a formalization of
privacy need to weighed against a possibly significant loss in inferential
accuracy, particularly in high-dimensional settings, and this motivates
further work on privacy-preserving mechanisms that retain the strengths of
differential privacy while mitigating some of its undesirable effects on
inference.


\subsection*{Acknowledgments}

We are very thankful to Shuheng Zhou for pointing out errors in
Corollary~\ref{corollary:idiot-fano} and
Proposition~\ref{PropPrivateFano} in an earlier version of this
manuscript.  We also thank Guy Rothblum for helpful discussions and
Margaret Roberts for providing us tokenized data on the censorship
task. Our work was supported in part by the U.S.\ Army Research Office
under grant number W911NF-11-1-0391, Office of Naval Research MURI
grant N00014-11-1-0688, and National Science Foundation (NSF)
grants CCF-1553086 and CAREER-1553086.




\bibliographystyle{abbrvnat}
\bibliography{bib}

\appendix


\section{Proof of Theorem~\ref{theorem:master} and related results}
\label{SecProofTheoremOne}

We now collect proofs of our main results, beginning with
Theorem~\ref{theorem:master}.

\subsection{Proof of Theorem~\ref{theorem:master}}
\label{sec:proof-diffp-to-tv}

Observe that $\marginprob_1$ and $\marginprob_2$ are absolutely
continuous with respect to one another, and there is a measure
$\basemeasure$ with respect to which they have densities
$\margindens_1$ and $\margindens_2$, respectively. The channel
probabilities $\channelprob(\cdot \mid \statsample)$ and
$\channelprob(\cdot \mid \statsample')$ are likewise absolutely
continuous, so that we may assume they have densities
$\channeldensity(\cdot \mid \statsample)$ and write
$\margindens_i(\channelval) = \int \channeldensity(\channelval \mid
\statsample) d \statprob_i(\statsample)$.  In terms of these
densities, we have
\ifdefined\useaosstyle
\begin{align*}
  \dklsym{\marginprob_1}{\marginprob_2} +
  \dkl{\marginprob_2}{\marginprob_1}
  & = \int
  \big(\margindens_1(\channelval) - \margindens_2(\channelval)\big)
  \log \frac{\margindens_1(\channelval)}{\margindens_2(\channelval)}
  d\basemeasure(\channelval).
\end{align*}
\else
\begin{align*}
  \dkl{\marginprob_1}{\marginprob_2} +
  \dkl{\marginprob_2}{\marginprob_1}
  & = \int
  \margindens_1(\channelval) \log \frac{\margindens_1(\channelval)}{
    \margindens_2(\channelval)} d\basemeasure(\channelval) + \int
  \margindens_2(\channelval) \log \frac{\margindens_2(\channelval)}{
    \margindens_1(\channelval)} d\basemeasure(\channelval) \\
  & = \int
  \big(\margindens_1(\channelval) - \margindens_2(\channelval)\big)
  \log \frac{\margindens_1(\channelval)}{\margindens_2(\channelval)}
  d\basemeasure(\channelval).
\end{align*}
\fi
Consequently, we must bound both the difference $\margindens_1 -
\margindens_2$ and the log ratio of the marginal densities.  The
following two auxiliary lemmas are useful:
\begin{lemma}
\label{lemma:channel-diff-tv}
For any $\diffp$-locally differentially private channel $\channelprob$, we have
\begin{align}
    \label{eqn:channel-diff-tv}
    \left|\margindens_1(\channelval) -
    \margindens_2(\channelval)\right| & \leq c_\diffp \inf_\statsample
    \channeldensity(\channelval \mid \statsample) \left(e^\diffp -
    1\right) \tvnorm{\statprob_1 - \statprob_2},
\end{align}
where $c_\diffp = \min\{2, e^\diffp\}$.
\end{lemma}

\begin{lemma}
  \label{lemma:log-ratio-inequality}
  Let $a, b \in \R_+$. Then $\left|\log\frac{a}{b}\right| \le \frac{|a
    - b|}{\min\{a, b\}}$.
\end{lemma}
\noindent We prove these two results at the end of this section.

With the lemmas in hand, let us now complete the proof of the
theorem.  From Lemma~\ref{lemma:log-ratio-inequality}, the log ratio
is bounded as
\begin{equation*}
\left|\log
\frac{\margindens_1(\channelval)}{\margindens_2(\channelval)} \right|
\le \frac{|\margindens_1(\channelval) - \margindens_2(\channelval)|}{
  \min\left\{\margindens_1(\channelval), \margindens_2(\channelval)
  \right\}}.
\end{equation*}
Applying Lemma~\ref{lemma:channel-diff-tv} to the numerator yields
\begin{align*}
  \left|\log \frac{\margindens_1(\channelval)}{
    \margindens_2(\channelval)}\right|
  & \leq \frac{c_\diffp
    \left(e^\diffp - 1\right) \tvnorm{\statprob_1 - \statprob_2} \;
    \inf_\statsample \channeldensity(\channelval \mid \statsample)}{
    \min\{\margindens_1(\channelval), \margindens_2(\channelval)\}} \\
  & \le
  \frac{c_\diffp \left(e^\diffp - 1\right) \tvnorm{\statprob_1 -
      \statprob_2} \; \inf_\statsample \channeldensity(\channelval \mid
    \statsample)}{ \inf_\statsample \channeldensity(\channelval \mid
    \statsample)},
\end{align*}
where the final step uses the inequality
$\min\{\margindens_1(\channelval), \margindens_2(\channelval)\} \ge
\inf_\statsample \channeldensity(\channelval \mid \statsample)$.
Putting together the pieces leads to the bound
\begin{equation*}
  \left|\log \frac{\margindens_1(\channelval)}{
    \margindens_2(\channelval)}\right|
  \le c_\diffp (e^\diffp - 1) \tvnorm{\statprob_1 - \statprob_2}.
\end{equation*}
Combining with inequality~\eqref{eqn:channel-diff-tv} yields
\begin{equation*}
  \dkl{\marginprob_1}{\marginprob_2} +
  \dkl{\marginprob_2}{\marginprob_1} \le c_\diffp^2 \left(e^\diffp -
  1\right)^2 \tvnorm{\statprob_1 - \statprob_2}^2 \int
  \inf_{\statsample} \channeldensity(\channelval \mid \statsample) d
  \basemeasure(\channelval).
\end{equation*}
The final integral is at most one, which completes the proof of the
theorem.\\

It remains to prove Lemmas~\ref{lemma:channel-diff-tv}
and~\ref{lemma:log-ratio-inequality}. We begin
with the former.  For any $\channelval \in
\channeldomain$, we have
\begin{align*}
  \margindens_1(\channelval) - \margindens_2(\channelval) & =
  \int_\statdomain \channeldensity(\channelval \mid \statsample)
  \left[d\statprob_1(\statsample) - d\statprob_2(\statsample)\right]
  \\
& = \int_\statdomain \channeldensity(\channelval \mid \statsample)
  \hinge{d\statprob_1(\statsample) - d\statprob_2(\statsample)} +
  \int_\statdomain \channeldensity(\channelval \mid \statsample)
  \neghinge{d\statprob_1(\statsample) - d\statprob_2(\statsample)}
  \\ 
& \le \sup_{\statsample \in \statdomain}
  \channeldensity(\channelval \mid \statsample) \int_\statdomain
  \hinge{d\statprob_1(\statsample) - d\statprob_2(\statsample)} +
  \inf_{\statsample \in \statdomain} \channeldensity(\channelval \mid
  \statsample) \int_\statdomain \neghinge{d\statprob_1(\statsample) -
    d\statprob_2(\statsample)} \\ & = \left(\sup_{\statsample \in
    \statdomain} \channeldensity( \channelval \mid \statsample) -
  \inf_{\statsample \in \statdomain} \channeldensity(\channelval \mid
  \statsample)\right) \int_\statdomain \hinge{d
    \statprob_1(\statsample) - d\statprob_2(\statsample)}.
\end{align*}
By definition of the total variation norm, we have $\int
\hinge{d\statprob_1 - d\statprob_2} = \tvnorm{\statprob_1 -
  \statprob_2}$, and hence
\begin{equation}
\label{EqnAuxTwo}
  |\margindens_1(\channelval) - \margindens_2(\channelval)|
  \le \sup_{\statsample, \statsample'} \left|\channeldensity(\channelval \mid
  \statsample) - \channeldensity(\channelval \mid \statsample')\right|
  \tvnorm{\statprob_1 - \statprob_2}.
\end{equation}
For any $\hat{\statsample} \in \statdomain$, we may add and subtract
$\channeldensity(\channelval \mid \hat{\statsample})$ from the quantity inside the
supremum, which implies that
\begin{align*}
  \sup_{\statsample, \statsample'}
  \left|\channeldensity(\channelval \mid \statsample)
  - \channeldensity(\channelval \mid \statsample')\right|
  & = \inf_{\hat{\statsample}}
  \sup_{\statsample, \statsample'}
  \left|\channeldensity(\channelval \mid \statsample)
  - \channeldensity(\channelval \mid \hat{\statsample})
  + \channeldensity(\channelval \mid \hat{\statsample})
  - \channeldensity(\channelval \mid \statsample')\right| \\
  & \le 2 \inf_{\hat{\statsample}}
  \sup_\statsample \left|\channeldensity(\channelval \mid \statsample)
  - \channeldensity(\channelval \mid \hat{\statsample})\right| \\
  & = 2 \inf_{\hat{\statsample}}
  \channeldensity(\channelval \mid \hat{\statsample})
  \sup_\statsample \left|\frac{
    \channeldensity(\channelval \mid \statsample)}{
    \channeldensity(\channelval \mid \hat{\statsample})}
  - 1\right|.
\end{align*}
Similarly, we have for any $\statsample, \statsample'$
\begin{equation*}
  |\channeldensity(\channelval \mid \statsample)
  - \channeldensity(\channelval \mid \statsample')|
  = \channeldensity(\channelval \mid \statsample')
  \left|\frac{\channeldensity(\channelval \mid \statsample)}{
    \channeldensity(\channelval \mid \statsample')} - 1\right|
  \le e^\diffp \inf_{\what{\statsample}}
  \channeldensity(\channelval \mid \what{\statsample})
  \left|\frac{\channeldensity(\channelval \mid \statsample)}{
    \channeldensity(\channelval \mid \statsample')} - 1\right|.
\end{equation*}
Since for any choice of $\statsample, \hat{\statsample}$, we have
$\channeldensity(\channelval \mid \statsample)/ \channeldensity(\channelval
\mid \hat{\statsample}) \in [e^{-\diffp}, e^\diffp]$, we
find that (since $e^\diffp - 1 \ge 1 - e^{-\diffp}$)
\begin{equation*}
  \sup_{\statsample, \statsample'} \left|\channeldensity(\channelval \mid
  \statsample) - \channeldensity(\channelval \mid \statsample')\right|
  \le \min\{2, e^\diffp\}
  \inf_\statsample \channeldensity(\channelval \mid \statsample) \left(e^\diffp
  - 1\right).
\end{equation*}
Combining with the earlier inequality~\eqref{EqnAuxTwo} yields the
claim~\eqref{eqn:channel-diff-tv}.

To see Lemma~\ref{lemma:log-ratio-inequality}, note that for any $x > 0$, the
concavity of the logarithm implies that
\begin{equation*}
    \log(x) \le x - 1.
\end{equation*} 
Setting alternatively $x = a/b$ and $x = b/a$, we obtain the inequalities
\begin{equation*}
  \log\frac{a}{b}
  \le \frac{a}{b} - 1
  = \frac{a - b}{b}
  ~~~ \mbox{and} ~~~
  \log \frac{b}{a}
  \le \frac{b}{a} - 1
  = \frac{b - a}{a}.
\end{equation*}
Using the first inequality for $a \ge b$ and the second for $a < b$ completes
the proof.

\subsection{Proof of Corollary~\ref{corollary:idiot-fano}}
\label{sec:proof-idiot-fano}

Let us recall the definition of the induced marginal
distribution~\eqref{eqn:marginal-channel}, given by
\begin{align*}
  \marginprob_\packval(S) = \int_{\Xspace} \channelprob(S \mid
  \statsample_{1:n}) d \statprob_\packval^n(\statsample_{1:n}) \quad
  \mbox{for $S \in \sigma(\channeldomain^n)$.}
\end{align*}
For each $i = 2, \ldots, \numobs$, we let
$\marginprob_{\HACKI{\packval}}(\cdot \mid \channelrv_1 = \channelval_1,
\ldots, \channelrv_{i-1} = \channelval_{i-1}) =
\marginprob_{\HACKI{\packval}}(\cdot \mid \channelval_{1:i-1})$ denote the
(marginal over $\statrv_i$) distribution of the variable $\channelrv_i$
conditioned on \mbox{$\channelrv_1 = \channelval_1, \ldots, \channelrv_{i-1}
  = \channelval_{i-1}$}. In addition, use the shorthand notation
\begin{align*}
 \dkl{\marginprob_{\HACKI{\packval}} }{ \marginprob_{\HACKI{\altpackval}}} &
 \defeq \int_{\channeldomain^{i-1}} \dkl{\marginprob_{\HACKI{\packval}}(
   \cdot \mid \channelval_{1:i-1})}{
   \marginprob_{\HACKI{\altpackval}}(\cdot \mid \channelval_{1:i-1})}
 d \marginprob_\packval^{i-1}(\channelval_1, \ldots, \channelval_{i-1})
\end{align*}
to denote the integrated KL divergence of the conditional distributions on
the $\channelrv_i$. By the chain rule for KL divergences~\cite[Chapter
  5.3]{Gray90}, we obtain
\begin{equation*}
  \dkl{\marginprob_\packval^n}{\marginprob_{\altpackval}^n} =
  \sum_{i=1}^n \dkl{\marginprob_{\HACKI{\packval}}}{
    \marginprob_{\HACKI{\altpackval}}}.
\end{equation*}

By assumption~\eqref{eqn:local-privacy}, the distribution
$\channelprob_i(\cdot \mid \statrv_i, \channelrv_{1:i-1})$ for $\channelrv_i$
is $\diffp$-differentially private for the sample $\statrv_i$.  As a
consequence, if we let \mbox{$\statprob_{\HACKI{\packval}}(\cdot \mid
  \channelrv_1 = \channelval_1, \ldots, \channelrv_{i-1} =
  \channelval_{i-1})$} denote the conditional distribution of $\statrv_i$
given the first $i-1$ values $\channelrv_1, \ldots, \channelrv_{i-1}$ and
the packing index $\packrv = \packval$, then from the chain rule and
Theorem~\ref{theorem:master} we obtain
\begin{align*}
  \dkl{\marginprob_\packval^n}{\marginprob_{\altpackval}^n}
  & = \sum_{i = 1}^n \int_{\channeldomain^{i-1}}
  \dkl{\marginprob_{\HACKI{\packval}}(\cdot \mid \channelval_{1:i-1})}{
    \marginprob_{\HACKI{\altpackval}}(\cdot \mid \channelval_{1:i-1})}
  d \marginprob_{\packval}^{i-1}(\channelval_{1:i-1}) \\
  & \le \sum_{i=1}^n 4(e^\diffp - 1)^2 \int_{\channeldomain^{i-1}}
  \tvnorm{\statprob_{\HACKI{\packval}}(\cdot \mid \channelval_{1:i-1})
    - \statprob_{\HACKI{\altpackval}}(\cdot \mid \channelval_{1:i-1})}^2
  d \marginprob_\packval^{i-1}(\channelval_1, \ldots, \channelval_{i-1}).
\end{align*}
By the construction of our sampling scheme, the random variables $\statrv_i$
are conditionally independent given $\packrv = \packval$; thus the
distribution $\statprob_{\HACKI{\packval}}(\cdot \mid \channelval_{1:i-1}) =
\statprob_{\HACKI{\packval}}$, where $\statprob_{\HACKI{\packval}}$ denotes the
distribution of $\statrv_i$ conditioned on $\packrv = \packval$.
Consequently, we have
\begin{equation*}
  \tvnorm{\statprob_{\HACKI{\packval}}(\cdot \mid \channelval_{1:i-1})
    - \statprob_{\HACKI{\altpackval}}(\cdot \mid \channelval_{1:i-1})}
  = \tvnorm{\statprob_{\HACKI{\packval}} - \statprob_{\HACKI{\altpackval}}},
\end{equation*}
which gives the claimed result.


\section{Proof of minimax bounds associated with Le Cam's method}
\label{sec:proof-le-cam-associates}
In this appendix, we collect proofs of the various minimax lower bounds for
specific problems in Section~\ref{sec:le-cam}.

\subsection{Proof of Corollary~\ref{CorLocFamily}}
\label{sec:proof-location-family}

The minimax rate characterized by equation~\eqref{eqn:location-family-bound}
involves both a lower and an upper bound, and we divide our proof
accordingly.  We provide the proof for $\diffp \in \openleft{0}{1}$, but
note that a similar result (modulo different constants) holds for any finite
value of $\diffp$.

\paragraphc{Lower bound} We use Le Cam's method to prove the lower
bound in equation~\eqref{eqn:location-family-bound}.  Fix a given
constant $\delta \in (0,1]$, with a precise value to be specified
  later.  For $\packval \in \packset \in \{-1, 1\}$, define the
  distribution $\statprob_\packval$ with support $\{-\delta^{-1/k}, 0,
  \delta^{1/k}\}$ by
\begin{equation*}
  \statprob_\packval(\statrv = \delta^{-1/k}) = \frac{\delta(1 +
    \packval)}{2}, ~~~ \statprob_\packval(\statrv = 0) = 1 - \delta,
  \quad \mbox{and} \quad \statprob_\packval(\statrv = -\delta^{-1/k})
  = \frac{\delta(1 - \packval)}{2}.
\end{equation*}
By construction, we have $\E[|\statrv|^k] = \delta (\delta^{-1/k})^k =
1$ and $\optvar_\packval = \E_\packval[\statrv] =
\delta^{\frac{k-1}{k}} \packval$, whence the mean difference is given
by \mbox{$\optvar_1 - \optvar_{-1} = 2 \delta^{\frac{k-1}{k}}$.}
Applying Le Cam's method~\eqref{EqnLeCam}  yields
\begin{equation*}
  \minimax_n(\optdomain, (\cdot)^2, \channelprob) \ge
  \left(\delta^{\frac{k-1}{k}}\right)^2
  \left(\half - \half \tvnorm{\marginprob_1^n -
    \marginprob_{-1}^n}\right),
\end{equation*}
where $\marginprob_\packval^n$ denotes the marginal distribution of
the samples $\channelrv_1, \ldots, \channelrv_n$ conditioned on
$\optvar = \optvar_\packval$.  Now Pinsker's inequality implies that
$\tvnorm{\marginprob_1^n - \marginprob_{-1}^n}^2 \le \half
\dkl{\marginprob_1^n}{\marginprob_{-1}^n}$, and
Corollary~\ref{corollary:idiot-fano} yields
\begin{equation*}
  \dkl{\marginprob_1^n}{\marginprob_{-1}^n} \le 4 (e^\diffp - 1)^2 n
  \tvnorm{\statprob_1 - \statprob_{-1}}^2 = 4 (e^\diffp - 1)^2 n
  \delta^2.
\end{equation*}
Putting together the pieces yields $\tvnorm{\marginprob_1^n -
  \marginprob_{-1}^n} \le (e^\diffp - 1) \delta \sqrt{2 n}$.  For
$\diffp \in (0,1]$, we have $e^\diffp - 1 \le 2 \diffp$, and thus our
  earlier application of Le Cam's method implies
\begin{equation*}
  \minimax_n(\optdomain, (\cdot)^2, \diffp) \ge
  \left(\delta^{\frac{k-1}{k}}\right)^2 \left(\half - \diffp \delta
  \sqrt{2n} \right).
\end{equation*}
Substituting $\delta = \min\{1, 1 / \sqrt{32 n \diffp^2}\}$ yields the
claim~\eqref{eqn:location-family-bound}.


\paragraphc{Upper bound} We must demonstrate an $\diffp$-locally
private conditional distribution $\channelprob$ and an estimator that
achieves the upper bound in
equation~\eqref{eqn:location-family-bound}.  We do so via a
combination of truncation and addition of Laplacian noise.  Define the
truncation function $\truncate{\cdot}{T} : \R \rightarrow [-T, T]$ by
\begin{equation*}
  \truncate{\statsample}{T} \defeq \max\{-T, \min\{\statsample, T\}\},
\end{equation*}
where the truncation level $T$ is to be chosen.  Let $W_i$ be
independent $\laplace(\diffp / (2T))$ random variables, and for each
index $i = 1, \ldots, \numobs$, define
$\channelrv_i \defeq \truncate{\statrv_i}{T} + W_i$.  By construction,
the random variable $\channelrv_i$ is $\diffp$-differentially private
for $\statrv_i$.  For the mean estimator $\what{\optvar} \defeq
\frac{1}{n} \sum_{i=1}^n \channelrv_i$, we have
\begin{equation}
  \label{eqn:location-mse}
  \E\left[(\what{\optvar} - \optvar)^2\right] = \var(\what{\optvar}) +
  \left(\E[\what{\optvar}] - \optvar\right)^2 = \frac{4T^2}{n
    \diffp^2} + \frac{1}{n} \var(\truncate{\statrv_1}{T}) +
  \left(\E[\channelrv_1] - \optvar\right)^2.
\end{equation}
We claim that
\begin{equation}
  \label{eqn:truncated-mean}
  \E[\channelrv] = \E\left[\truncate{\statrv}{T}\right] \in
  \left[\E[\statrv] - \frac{2^{1/k}}{T^{k - 1}}, \E[\statrv] +
    \frac{2^{1/k}}{T^{k - 1}}\right].
\end{equation}
Indeed, by the assumption that $\E[|\statrv|^k] \le 1$, if we take $k' =
\frac{k}{k-1}$ so that $1/k' + 1/k = 1$, we have that
\begin{align*}
  \E[|\statrv - \truncate{\statrv}{T}|]
  & = \E[|\statrv - T| \indic{\statrv > T}]
  + \E[|\statrv + T| \indic{\statrv < -T}] \\
  & \le \E[|\statrv| \indic{\statrv > T}]
  + \E[|\statrv| \indic{\statrv < -T}] \\
  & \le \E[|\statrv|^k]^{1/k} \P(\statrv > T)^{1/k'}
  + \E[|\statrv|^k]^{1/k} \P(\statrv < -T)^{1/k'}
  \le 2^{1/k} T^{1 - k}.
\end{align*}
The final inequality follows from Markov's inequality, as $\P(X > T) + \P(X <
-T) = \P(|X| > T) \le T^{-k}$, and for any $k' \in [1, \infty]$, we have
$\sup_{a + b \le c} \{a^{1/k'} + b^{1/k'} : a, b \ge 0\} = 2^{1 - 1/k'}
c^{1/k'} = 2^{1/k} c^{1/k'}$.


From the bound~\eqref{eqn:location-mse} and the inequalities that
since $\truncate{\statrv}{T} \in [-T, T]$ and $\diffp^2 \le 1$, we
have
\begin{equation*}
  \E \left[(\what{\optvar} - \optvar)^2\right] \le \frac{5 T^2}{n
  \diffp^2} + \frac{4}{T^{2k - 2}} \quad \mbox{valid for any $T > 0$.}
\end{equation*}
Choosing $T = (n \diffp^2)^{1/(2k)}$ yields the upper
bound~\eqref{eqn:location-family-bound} immediately.


\subsection{Proof of Corollary~\ref{CorMedEst}}
\label{sec:proof-median-estimation}

We have already given the proof of the upper bound in the proposition
in the course of our discussion of the stochastic gradient descent
estimator~\eqref{eqn:median-sgd}. We thus focus on the lower bound.
Fix $\radius > 0$ and let $\delta \in [0, 1]$ be a constant to be
chosen subsequently. For $\packval \in \{-1, 1\}$ consider the
distributions $\statprob_\packval$ supported on $\{-\radius,
\radius\}$ defined by
\begin{equation*}
  P_\packval(\statrv = \radius) = \frac{1 + \delta \packval}{2}
  ~~ \mbox{and} ~~
  P_\packval(\statrv = -\radius) = \frac{1 - \delta \packval}{2}.
\end{equation*}
We have $\median(P_1) = \radius$ and $\median(P_{-1}) = -\radius$, and using
an extension of the classical reduction from minimax estimation to testing
to gaps in function values (recall
inequality~\eqref{EqnLeCam}; see also
\cite{AgarwalBaRaWa12}), we claim that
\begin{equation}
  \minimax_n([-\radius, \radius], \risk, \channel)
  \ge \frac{\radius \delta}{2}
  \inf_\test \P\left(\test(\channelrv_1, \ldots, \channelrv_n) \neq \packrv
  \right),
  \label{eqn:risk-lower-bound}
\end{equation}
where $\packrv \in \{-1, 1\}$ is chosen at random and conditional on
$\packrv = \packval$, we draw $\channelrv_i \simiid \marginprob_\packval$
for $\marginprob_\packval(\cdot) = \int \channel(\cdot \mid \statval)
dP_\packval(\statval)$. Indeed, under this model, by using the shorthand
$\risk_\packval(\theta) = \E_{\statprob_\packval}[|\theta - \statrv|]$, we
have $\inf_\theta \risk_\packval(\theta) = \frac{1 - \delta}{2} \radius$;
and for any estimator $\what{\theta}$ we have
\begin{align*}
  \half \sum_{\packval \in \{\pm 1\}}
  \left(\E_{\statprob_\packval}[\risk_\packval(\what{\theta})]
  - \frac{1 - \delta}{2} \radius\right)
  & \ge \half \sum_{\packval \in \{\pm 1\}}
  \E_{\statprob_\packval}\left[\indic{\sign(\what{\theta}) \neq \packval}
    \frac{\radius \delta}{2}\right]
  \ge \frac{\radius \delta}{2} \inf_\test \P(\test(\channelrv_{1:n})
  \neq \packrv)
\end{align*}
as claimed.
Continuing, we use Pinsker's inequality and
Corollary~\ref{corollary:idiot-fano}, which gives
\begin{align*}
  \tvnorm{\marginprob_1^n - \marginprob_{-1}^n}^2
  \le \frac{1}{4}\left[\dkl{\marginprob_1^n}{\marginprob_{-1}^n}
    + \dkl{\marginprob_{-1}^n}{\marginprob_1^n}\right]
  \le (e^\diffp - 1)^2 n \tvnorm{\statprob_1 - \statprob_{-1}}^2
  \le 3 n \diffp^2 \delta^2
\end{align*}
for $\diffp \le 1$. Thus, the bound~\eqref{eqn:risk-lower-bound} implies
that for any $\delta \in [0, 1]$ and any $\diffp$-private channel
$\channel$, we have
\begin{equation*}
  \minimax_n([-\radius, \radius], \risk, \channel)
  \ge \frac{\radius \delta}{4} \left(1 - 3 n \diffp^2 \delta^2\right).
\end{equation*}
Take $\delta^2 = \frac{1}{6 n \diffp^2}$ to give the result of the
corollary with $c_\ell = \frac{1}{20}$.



\section{Proof of Theorem~\ref{theorem:super-master} and related results}

In this section, we collect together the proof of
Theorem~\ref{theorem:super-master} and related corollaries.

\subsection{Proof of Theorem~\ref{theorem:super-master}}
\label{sec:proof-super-master}

Let $\channeldomain$ denote the domain of the random variable
$\channelrv$.  We begin by reducing the problem to the case when
$\channeldomain = \{1, 2, \ldots, k \}$ for an arbitrary positive
integer $k$.  Indeed, in the general setting, we let $\mc{K} =
\{K_i\}_{i=1}^k$ be any (measurable) finite partition of
$\channeldomain$, where for $z \in \channeldomain$ we let $[z]_\mc{K}
= K_i$ for the $K_i$ such that $z \in K_i$.  The KL divergence
$\dkl{\marginprob_\packval}{\meanmarginprob}$ can be defined as the
supremum of the (discrete) KL divergences between the random variables
$[\channelrv]_{\mc{K}}$ sampled according to $\marginprob_\packval$
and $\meanmarginprob$ over all partitions $\mc{K}$ of
$\channeldomain$; for instance, see~\citet[Chapter~5]{Gray90}.
Consequently, we can prove the claim for $\channeldomain = \{1, 2,
\ldots, k\}$, and then take the supremum over $k$ to recover the
general case.  Accordingly, we can work with the probability mass
functions $\margindens(\channelval \mid \packval) =
\marginprob_\packval(\channelrv = \channelval)$ and
$\meanmargindensity(\channelval) = \meanmarginprob(\channelrv =
\channelval)$, and we may write
\begin{equation}
  \label{eqn:kl-to-finite}
  \dkl{\marginprob_\packval}{\meanmarginprob} +
  \dkl{\meanmarginprob}{\marginprob_\packval} = \sum_{\channelval =
    1}^k \left(\margindens(\channelval \mid \packval) -
  \meanmargindensity(\channelval)\right) \log
  \frac{\margindens(\channelval \mid \packval)}{
    \meanmargindensity(\channelval)}.
\end{equation}
Throughout, we will also use (without loss of generality) the probability mass
functions $\channeldensity(\channelval \mid \statsample) =
\channelprob(\channelrv = \channelval \mid \statrv = \statsample)$, where we
note that $\margindens(\channelval \mid \packval) = \int
\channeldensity(\channelval \mid \statsample)
d\statprob_\packval(\statsample)$.

Now we use Lemma~\ref{lemma:log-ratio-inequality} from the proof of
Theorem~\ref{theorem:master} to complete the proof of
Theorem~\ref{theorem:super-master}.
Starting with equality~\eqref{eqn:kl-to-finite}, we have
\begin{align*}
  \frac{1}{|\packset|}
  \sum_{\packval \in \packset}
  \left[\dkl{\marginprob_\packval}{\meanmarginprob}
    + \dkl{\meanmarginprob}{\marginprob_\packval}\right]
  & \le
  \sum_{\packval \in \packset}
  \frac{1}{|\packset|}
  \sum_{\channelval = 1}^k
  \left|\margindens(\channelval \mid \packval)
  - \meanmargindensity(\channelval)\right|
  \left|\log \frac{\margindens(\channelval \mid \packval)}{
    \meanmargindensity(\channelval)}\right| \\
  & \le
  \sum_{\packval \in \packset}
  \frac{1}{|\packset|}
  \sum_{\channelval = 1}^k
  \left|\margindens(\channelval \mid \packval)
  - \meanmargindensity(\channelval)\right|
  \frac{\left|\margindens(\channelval \mid \packval)
    - \meanmargindensity(\channelval)\right|}{
    \min\left\{\meanmargindensity(\channelval),
    \margindens(\channelval \mid \packval)\right\}}.
\end{align*}
Now, we define the measure $\margincenter$ on $\channeldomain = \{1, \ldots,
k\}$ by $\margincenter(\channelval) \defeq \inf_{\statval \in \statdomain}
\channeldensity(\channelval \mid \statval)$.  It is clear that
$\min\left\{\meanmargindensity(\channelval), \margindens(\channelval \mid
\packval)\right\} \ge \margincenter(\channelval)$, whence we find
\begin{align*}
  \frac{1}{|\packset|} \sum_{\packval \in \packset}
  \left[\dkl{\marginprob_\packval}{\meanmarginprob} +
    \dkl{\meanmarginprob}{\marginprob_\packval}\right]
  & \le \sum_{\packval \in
    \packset} \frac{1}{|\packset|} \sum_{\channelval = 1}^k
  \frac{\left(\margindens(\channelval \mid \packval) -
    \meanmargindensity(\channelval)\right)^2}{
    \margincenter(\channelval)}.
\end{align*}
It remains to bound the final sum.  For any constant $c \in \R$, we
have
\begin{equation*}
  \margindens(\channelval \mid \packval) - \meanmargindensity(\channelval)
  = \int_\statdomain \left( \channeldensity(\channelval \mid \statsample) -
  c\right) \left(d\statprob_\packval(\statsample) -
  d\meanstatprob(\statsample)\right).
\end{equation*}
We define a set of functions $f : \channeldomain \times \statdomain
\rightarrow \R$ (depending implicitly on $\channeldensity$) by
\begin{equation*}
  \mc{F}_\diffp \defeq \left\{
  f \, \mid f(\channelval, \statsample) \in [1, e^\diffp]
  \margincenter(\channelval)
  ~ \mbox{for~all~} \channelval
  \in \channeldomain ~\mbox{and}~ \statsample \in \statdomain \right\}.
\end{equation*}
By the definition of differential privacy, when viewed as a joint mapping from
$\channeldomain \times \statdomain \rightarrow \R$, the conditional
p.m.f.\ $\channeldensity$ satisfies $\{(\channelval, \statsample) \mapsto
\channeldensity(\channelval \mid \statsample)\} \in \mc{F}_\diffp$.
Since constant (with respect to $\statsample$) shifts do not change the above
integral, we can modify the range of functions in $\mc{F}_\diffp$ by
subtracting $\margincenter(\channelval)$
from each, yielding the set
\begin{equation*}
  \mc{F}'_\diffp
  \defeq \left\{
  f \, \mid f(\channelval, \statsample) \in \left[0, e^\diffp - 1\right]
  \margincenter(\channelval)
  ~ \mbox{for~all~} \channelval
  \in \channeldomain ~\mbox{and}~ \statsample \in \statdomain \right\}.
\end{equation*}
As a consequence, we find that
\begin{align*}
  \sum_{\packval \in \packset} 
  \left(\margindens(\channelval \mid \packval) -
  \meanmargindensity(\channelval)
  \right)^2
  & \le \sup_{f \in \mc{F}_\diffp}
  \left\{ \sum_{\packval\in \packset} 
  \left(\int_\statdomain f(\channelval, \statsample)
  \left(d\statprob_\packval(\statsample) -
  d\meanstatprob(\statsample)\right) \right)^2\right\} \\
  & = \sup_{f \in \mc{F}'_\diffp}
  \left\{ \sum_{\packval\in \packset} 
  \left(\int_\statdomain \left( f(\channelval, \statsample) -
  \margincenter(\channelval) \right)
  \left(d\statprob_\packval(\statsample) -
  d\meanstatprob(\statsample)\right) \right)^2\right\}.
\end{align*}
By inspection, when we divide by $\margincenter(\channelval)$ and recall the
definition of the set $\linfset \subset L^\infty(\statdomain)$ in the
statement of Theorem~\ref{theorem:super-master}, we obtain
\begin{equation*}
  \sum_{\packval \in \packset}
  \left(\margindens(\channelval \mid \packval) 
  - \meanmargindensity(\channelval)\right)^2
  \le \left(\margincenter(\channelval)\right)^2
  (e^\diffp - 1)^2 \sup_{\optdens \in \linfset}
  \sum_{\packval \in \packset} \left(\int_\statdomain \optdens
  (\statsample) \left( d\statprob_\packval(\statsample) -
  d\meanstatprob(\statsample)\right) \right)^2.
\end{equation*}
Putting together our bounds, we have
\begin{align*}
  \lefteqn{\frac{1}{|\packset|} \sum_{\packval \in \packset}
    \left[\dkl{\marginprob_\packval}{\meanmarginprob} +
      \dkl{\meanmarginprob}{\marginprob_\packval}\right]} \\
  & \quad \le (e^\diffp - 1)^2
  \sum_{\channelval = 1}^k \frac{1}{|\packset|}
  \frac{\left(\margincenter(\channelval)\right)^2}{
    \margincenter(\channelval)} \sup_{\optdens \in \linfset}
  \sum_{\packval \in \packset} \left(\int_\statdomain
  \optdens(\statsample) \left( d\statprob_\packval(\statsample) -
  d\meanstatprob(\statsample)\right) \right)^2 \\
  & \quad \le (e^\diffp - 1)^2
  \frac{1}{|\packset|} \sup_{\optdens \in \linfset} \sum_{\packval \in
    \packset} \left(\int_\statdomain \optdens(\statsample) \left(
  d\statprob_\packval(\statsample) - d\meanstatprob(\statsample)\right)
  \right)^2,
\end{align*}
since $\sum_\channelval \margincenter(\channelval) \le 1$,
which is the statement of the theorem. \\


\subsection{Proof of Inequality~\eqref{EqnSuperFano}}
\label{sec:proof-super-fano}

In the non-interactive setting~\eqref{eqn:local-privacy-simple}, the
marginal distribution $\marginprob^n_\packval$ is a product measure
and $\channelrv_i$ is conditionally independent of
$\channelrv_{1:i-1}$ given $\packrv$. Thus by the chain rule for
mutual information~\cite[Chapter 5]{Gray90} and the fact (as in the
proof of Theorem~\ref{theorem:super-master}) that we may assume
w.l.o.g.\ that $\channelrv$ has finite range
\begin{align*}
  \information(\channelrv_1, \ldots, \channelrv_\numobs;
  \packrv)
  & = \sum_{i = 1}^\numobs \information(\channelrv_i; \packrv
  \mid \channelrv_{1:i-1})
  = \sum_{i = 1}^\numobs \left[H(\channelrv_i \mid \channelrv_{1:i-1})
  - H(\channelrv_i \mid \packrv, \channelrv_{1:i-1})\right].
\end{align*}
Since conditioning reduces entropy and $\channelrv_{1:i-1}$ is conditionally
independent of $\channelrv_i$ given $\packrv$, we have $H(\channelrv_i \mid
\channelrv_{1:i-1}) \le H(\channelrv_i)$ and $H(\channelrv_i \mid \packrv,
\channelrv_{1:i-1}) = H(\channelrv_i \mid \packrv)$.
In particular, we have
\begin{equation*}
  \information(\channelrv_1, \ldots, \channelrv_\numobs; \packrv)
  \le \sum_{i = 1}^\numobs \information(\channelrv_i; \packrv)
  = \sum_{i = 1}^\numobs \frac{1}{|\packset|}
  \sum_{\packval \in \packset} \dkl{\marginprob_{\packval,i}}{
    \meanmarginprob_i}.
\end{equation*}
Applying Theorem~\ref{theorem:super-master} completes the proof.



\section{Proofs of multi-dimensional mean-estimation results}
\label{sec:proofs-big-mean-estimation}

At a high level, our proofs of these results
consist of three steps, the first of
which is relatively standard, while the second two exploit specific
aspects of the local privacy setting. We outline them here:

\begin{enumerate}[(1)]
\item \label{step:standard}
  The first step is an essentially standard reduction, based on
  inequality~\eqref{eqn:user-friendly-fano}
  in Section~\ref{sec:fano}, from an
  estimation problem to a multi-way testing problem that involves
  discriminating between indices $\packval$ contained within some subset
  $\packset$ of $\R^d$.
\item \label{item:construct-packing} The second step is an appropriate
  construction of a packing set $\packset \subset \R^d$.  We require the
  existence of a well-separated set: one for which ratio of the packing set
  size $|\packset|$ to neighborhood size $\hoodbig_\delta$ is large enough
  relative to the separation $\delta$ of
  definition~\eqref{eqn:neighborhood-size}.
\item The final step is to apply Theorem~\ref{theorem:super-master} in order
  to control the mutual information associated with the testing
  problem. Doing so requires bounding the supremum in
  Theorem~\ref{theorem:super-master} (and inequality~\eqref{EqnSuperFano})
  via the operator norm of $\cov(\packrv)$ for a vector $\packrv$ drawn
  uniformly at random from $\packset$. This is made easier by the uniformity
  of the sampling scheme allowed by the
  generalization~\eqref{eqn:user-friendly-fano} of Fano's inequality we use,
  as it is possible to enforce that $\cov(\packrv)$ has relatively small
  operator norm.
\end{enumerate}

\noindent The estimation to testing reduction of Step~\ref{step:standard} is
accomplished by the reduction~\eqref{eqn:user-friendly-fano} of
Section~\ref{sec:fano}.  Accordingly, the proofs to follow are devoted to
the second and third steps in each case.


\subsection{Proof of Corollary~\ref{CorDmean}}
\label{sec:proof-d-dimensional-mean}

We provide a proof of the lower bound, as we provided the argument for
the upper bound in Section~\ref{sec:attainability-means}.

\paragraphc{Constructing a well-separated set}
Let $k$ be an arbitrary integer in $\{1, 2, \ldots, d\}$,
and let $\packset_k = \{-1, 1\}^k$ denote the $k$-dimensional hypercube.
We extend the set $\packset_k \subseteq \real^k$ to a
subset of $\R^d$ by setting $\packset = \packset_{k} \times \{0\}^{d -k}$.
For a parameter $\delta \in \openleft{0}{1/2}$ to be chosen, we define a
family of probability distributions $\{\statprob_\packval\}_{\packval \in
  \packset}$ constructively.  In particular, the random vector $X \sim
\statprob_\packval$ (a single observation) is formed by the following
procedure:
\begin{equation}
  \label{eqn:linf-sampling-scheme}
  \mbox{Choose~index~} j \in \{1, \ldots, k\} ~
  \mbox{uniformly~at~random~and~set} ~ \statrv =
  \begin{cases}
    \radius e_j & \mbox{w.p.}~ \frac{1 + \delta \packval_j}{2}
    \\ -\radius e_j & \mbox{w.p.}~ \frac{1 - \delta \packval_j}{2}.
  \end{cases}
\end{equation}
By construction, these distributions have mean vectors
\begin{align*}
  \optvar_\packval \defeq
  \E_{\statprob_\packval}[\statrv]
  = \frac{\delta \radius}{k} \packval.
\end{align*}
Consequently, given the properties of the packing $\packset$, we have
$\statrv \in \ball_1(\radius)$ with probability 1,
and fixing $t \in \R_+$, we have that the
associated separation function~\eqref{eqn:separation} satisfies
\begin{equation*}
  \delta^2(t) \ge \min\left\{\ltwo{\optvar_\packval - \optvar_\altpackval}^2
  \mid \lone{\packval - \altpackval} \ge t \right\}
  \ge \frac{\radius^2 \delta^2}{k^2}
  \min\left\{\ltwo{\packval - \altpackval}^2
  \mid \lone{\packval - \altpackval} \ge t
  \right\} \ge \frac{2\radius^2 \delta^2}{k^2} t.
\end{equation*}
We claim that
so long as $t \le k/6$ and $k \ge 3$, we have
\begin{equation}
  \label{eqn:tim-hortons}
  \log \frac{|\packset|}{\hoodbig_t} > \max\left\{\frac{k}{6}, 2\right\}.
\end{equation}
Indeed, for $t \in \N$ with $t \le k/2$, we see by the binomial
theorem that
\begin{equation*}
  \hoodbig_t = \sum_{\tau = 0}^t \binom{k}{\tau}
  \le 2 \binom{k}{t}
  \le 2 \left(\frac{k e}{t}\right)^t.
\end{equation*}
Consequently, for $t \le k/6$, the ratio $|\packset| / \hoodbig_t$ satisfies
\begin{equation*}
  \log\frac{|\packset|}{\hoodbig_t}
  \ge k \log 2 - \log 2 \binom{k}{t}
  \ge k \log 2 - \frac{k}{6}
  \log(6 e) - \log 2
  = k \log \frac{2}{2^{1/k} \sqrt[6]{6e}}
  > \max\left\{\frac{k}{6}, 2\right\}
\end{equation*}
for $k \ge 12$. The case $2 \le k < 12$ can be checked directly, yielding
claim~\eqref{eqn:tim-hortons}.

Thus we see that
the mean vectors $\{\optvar_\packval\}_{\packval \in \packset}$ provide
us with an $\radius \delta\sqrt{2t}/ k$-separated set (in $\ell_2$-norm)
with log ratio of its size at least $\max\{k/6, 2\}$.


\paragraphc{Upper bounding the mutual information}

Our next step is to bound the mutual information
$\information(\channelrv_1, \ldots, \channelrv_n; \packrv)$ when the
observations $\statrv$ come from the
distribution~\eqref{eqn:linf-sampling-scheme} and $\packrv$ is uniform
in the set $\packset$.  We have the following lemma, which applies so
long as the channel $\channelprob$ is non-interactive and
$\diffp$-locally private~\eqref{eqn:local-privacy-simple}.  See
Appendix~\ref{sec:linf-info-bound} for the proof.
\begin{lemma}
  \label{lemma:linf-info-bound}
  Fix $k \in \{1, \ldots, d\}$.
  Let $\channelrv_i$ be $\diffp$-locally differentially
  private for $\statrv_i$, and let $\statrv$
  be sampled according to the distribution~\eqref{eqn:linf-sampling-scheme}
  conditional on $\packrv = \packval$. Then
  \begin{equation*}
    \information(\channelrv_1, \ldots, \channelrv_n; \packrv) \le n
    \frac{\delta^2}{4k} (e^\diffp - 1)^2.
  \end{equation*}
\end{lemma}

\paragraphc{Applying testing inequalities}

We now show how a combination the sampling
scheme~\eqref{eqn:linf-sampling-scheme} and
Lemma~\ref{lemma:linf-info-bound} give us our desired lower bound. Fix $k
\le d$ and let $\packset = \{-1, 1\}^k \times \{0\}^{d-k}$.  Combining
Lemma~\ref{lemma:linf-info-bound} and the fact that the vectors
$\optvar_\packval$ provide a $\radius \delta \sqrt{2t}/k$-separated set of
log-cardinality at least $\max\{k/6, 2\}$, the minimax Fano
bound~\eqref{eqn:user-friendly-fano} implies that for any $k \in \{1,
\ldots, d\}$ and $t \le k/6$, we have
\begin{equation*}
  \minimax_n(\optvar(\mc{\statprob}), \ltwo{\cdot}^2, \diffp)
  \ge \frac{\radius^2 \delta^2 t}{2 k^2}
  \left(1 - \frac{n \delta^2 (e^\diffp - 1)^2 / (4k)
    + \log 2}{\max\{k/6, 2\}}\right).
\end{equation*}
Because of the one-dimensional mean-estimation lower bounds provided 
in Section~\ref{sec:location-family}, we may assume w.l.o.g.\ that 
$k \ge 12$. Setting $t = k/6$ and $\delta_{n,\diffp,k}^2
= \min\{1, k^2 / (3 n (e^\diffp - 1)^2)\}$, we obtain
\begin{equation*}
  \minimax_n(\optvar(\mc{\statprob}), \ltwo{\cdot}^2, \diffp)
  \ge \frac{\radius^2 \delta_{n, \diffp, k}^2}{12 k}
  \left(1 - \half - \frac{\log 2}{2}\right)
  \ge \frac{1}{80} \radius^2 \min\left\{\frac{1}{k},
  \frac{k}{3 n (e^\diffp - 1)^2}\right\}.
\end{equation*}
Setting $k$ in the
preceding display to be the integer in $\{1, \ldots, d\}$ nearest $\sqrt{n
  (e^\diffp - 1)^2}$ gives the lower bound
\begin{equation}
  \label{eqn:intermediate-d-mean-lower}
  \minimax_n(\theta(\pclass), \ltwo{\cdot}^2, \diffp)
  \ge c \radius^2 \min\left\{1, \frac{1}{\sqrt{n (e^\diffp - 1)^2}},
  \frac{d}{n (e^\diffp - 1)^2}\right\},
\end{equation}
where $c > 0$ is a numerical constant.  We return to this bound presently.

\subsubsection{An alternative lower bound for Corollary~\ref{CorDmean}}

We now provide an alternative lower bound, which relies on a denser sampling
strategy than the single-coordinate
construction in~\eqref{eqn:linf-sampling-scheme}.  Our proof parallels the
structure leading to the bound~\eqref{eqn:intermediate-d-mean-lower}, so we
are somewhat more terse. In this proof, we fix the power $p \in [1, 2]$ for
such $\supp(\statdomain) \subset \{\statval \in \R^d : \norm{\statval}_p \le
\radius\}$.  We also let $\radius = 1$ without loss of generality; we may
scale our results arbitrarily by this factor.

\paragraphc{Constructing a well-separated set}

As in the preceding derivation, fix $k \in \{1, \ldots, d\}$, and let
$\packset = \{-1, 1\}^k \times \{0\}^{d-k}$ be the $k$-dimensional hypercube
appended with a zero vector. As previously, we assume without loss of
generality that $k \ge 12$, as otherwise Corollary~\ref{CorLocFamily} gives
the lower bound. For a parameter $\delta \in \openleft{0}{1/2}$ to be
chosen, we define a family of probability distributions
$\{\statprob_\packval\}_{\packval \in \packset}$ constructively.  The random
vector $X \sim \statprob_\packval$ (a single observation) is supported on
the set $\statdomain \defeq \{-1 / k^{1/p}, 1 / k^{1/p}\}^k \times
\{0\}^{d-k}$, so that $\statval \in \statdomain$ satisfies
$\norm{\statval}_p = k^{1/p} / k^{1/p} = 1$.

For simplicity in notation, let us now suppose that $\packset = \{-1,
1\}^k$ and $\statdomain = \{-1 / k^{1/p}, 1/k^{1/p}\}$, suppressing
dependence on the zero vectors that we append.  We draw $\statrv$ according
to
\begin{equation}
  \label{eqn:lp-sampling-scheme}
  \mbox{for~} \statval \in \{-1, 1\}^k,
  ~~~
  P_\packval(\statrv = \statval / k^{1/p})
  = \frac{1 + \delta \packval^\top \statval}{2^k}.
\end{equation}
That is, their coordinates are independent on $\{-1 / k^{1/p}, 1 /
k^{1/p}\}$, with $P_\packval(\statrv_j = k^{-1/p}) = \frac{1 + \delta
  \packval_j}{2}$.  By construction, these distributions have our desired
support and means
\begin{equation*}
  \optvar_\packval \defeq
  \E_{\statprob_\packval}[\statrv]
  = \frac{\delta \packval}{k^{1/p}}.
\end{equation*}
Fixing $t \in \R_+$,
the
associated separation function~\eqref{eqn:separation} satisfies
\begin{equation*}
  \delta^2(t) \ge \min\left\{\ltwo{\optvar_\packval - \optvar_\altpackval}^2
  \mid \lone{\packval - \altpackval} \ge t \right\}
  \ge \frac{\radius^2 \delta^2}{k^{2/p}}
  \min\left\{\ltwo{\packval - \altpackval}^2
  \mid \lone{\packval - \altpackval} \ge t
  \right\} \ge \frac{2\radius^2 \delta^2}{k^{2/p}} t.
\end{equation*}
We again have the claim~\eqref{eqn:tim-hortons}, that is,
that $\log \frac{|\packset|}{\hoodbig_t} > \max\{\frac{k}{6}, 2\}$.  The
mean vectors $\{\optvar_\packval\}_{\packval \in \packset}$ provide us with
an $\delta\sqrt{2t}/ k^{1/p}$-separated set (in $\ell_2$-norm) with log
ratio of its cardinality at least $\max\{k/6, 2\}$.

\paragraphc{Upper bounding the mutual information}

Our next step is to bound the mutual information
$\information(\channelrv_1, \ldots, \channelrv_n; \packrv)$ when the
observations $\statrv$ come from the
distribution~\eqref{eqn:lp-sampling-scheme} and $\packrv$ is uniform
in the set $\packset$.  We have the following lemma, which applies so
long as the channel $\channelprob$ is non-interactive and
$\diffp$-locally private~\eqref{eqn:local-privacy-simple}.
\begin{lemma}
  \label{lemma:digit-info-bound}
  Let $\channelrv_i$ be $\diffp$-locally differentially
  private for $\statrv_i$, and let $\statrv$
  be sampled according to the distribution~\eqref{eqn:linf-sampling-scheme}
  conditional on $\packrv = \packval$. Then
  \begin{equation*}
    \information(\channelrv_1, \ldots, \channelrv_n; \packrv) \le n
    \delta^2 (e^\diffp - 1)^2.
  \end{equation*}
\end{lemma}
\noindent
See Appendix~\ref{sec:proof-digit-info-bound} for a proof of the lemma.

\paragraphc{Applying testing inequalities}

We now show how a combination of the sampling
scheme~\eqref{eqn:lp-sampling-scheme} and
Lemma~\ref{lemma:digit-info-bound} give us our desired lower bound.
Combining Lemma~\ref{lemma:digit-info-bound} and
the fact that the vectors
$\optvar_\packval$ provide a $\radius \delta \sqrt{2t}
/ k^{1/p}$-separated set of log-cardinality at least
$\max\{k/6, 2\}$, the minimax Fano
bound~\eqref{eqn:user-friendly-fano} implies that for any $k \in \{1,
\ldots, d\}$ and $t \le k/6$, we have
\begin{equation*}
  \minimax_n(\optvar(\mc{\statprob}), \ltwo{\cdot}^2, \diffp)
  \ge \frac{\radius^2 \delta^2 t}{k^{2/p}}
  \left(1 - \frac{n \delta^2 (e^\diffp - 1)^2
    + \log 2}{\max\{k/6, 2\}}\right).
\end{equation*}
As in the preceding argument leading to 
the bound~\eqref{eqn:intermediate-d-mean-lower},
we assume $k \ge 12$, set $t = k/6$ and
$\delta_{n,\diffp,k}^2 = \min\{1, k / (12 n(e^\diffp - 1)^2)\}$ to
obtain
\begin{equation*}
  \minimax_n(\optvar(\mc{\statprob}), \ltwo{\cdot}^2, \diffp)
  \ge \frac{\radius^2 \delta_{n,\diffp,k}^2}{6 k^\frac{2 - p}{p}}
  \left(1 - \half - \frac{6 \log 2}{k}\right)
  \ge \frac{1}{40}
  \radius^2
  \min\left\{ \frac{1}{k^\frac{2 - p}{p}},
  \frac{k^{2 \frac{p - 1}{p}}}{12 n (e^\diffp - 1)^2}\right\}.
\end{equation*}
As $k$ is arbitrary, we may take
$k = \max\{1, \min\{d, n (e^\diffp - 1)^2\}\}$ to obtain
\begin{equation*}
  \minimax_n(\optvar(\mc{\statprob}), \ltwo{\cdot}^2, \diffp)
  \ge c \radius^2
  \min\left\{1, (n (e^\diffp - 1)^2)^\frac{p - 2}{p},
  \frac{d^{2 \frac{p - 1}{p}}}{
    n (e^\diffp - 1)^2}\right\},
\end{equation*}
where $c > 0$ is a numerical constant.

Combining this minimax lower bound with our earlier
inequality~\eqref{eqn:intermediate-d-mean-lower}, and using that $(e^\diffp
- 1)^2 < 3 \diffp^2$ for $\diffp \in [0, 1]$, we obtain the 
corollary.

\subsection{Proof of Corollary~\ref{CorSparseMean}}
\label{sec:proof-mean-high-dim}

\paragraphc{Constructing a well-separated set} In this case, the packing set
is very simple: set $\packset = \{\pm e_j\}_{j=1}^d$ so that
$|\packset| = 2d$.  Fix some $\delta \in [0, 1]$, and for $\packval
\in \packset$, define a distribution $\statprob_\packval$ supported on
$\statdomain = \{-\radius, \radius\}^d$ via
\begin{align*}
  \statprob_\packval(\statrv = \statsample) = (1 +
  \delta \packval^\top \statsample / \radius) / 2^d.
\end{align*}
In words, for $\packval = e_j$, the coordinates of $\statrv$ are independent
uniform on $\{-\radius, \radius\}$ except for the coordinate $j$, for which
$\statrv_j = \radius$ with probability $1/2 + \delta \packval_j$ and
$\statrv_j = -\radius$ with probability $1/2 - \delta \packval_j$.  With
this scheme, we have $\optvar(\statprob_\packval) = \radius \delta
\packval$, which is $1$-sparse, and since $\ltwo{\delta \radius \packval -
  \delta \radius \altpackval} \ge \sqrt{2} \delta \radius$, we have
constructed a $\sqrt{2}\delta \radius$ packing in $\ell_2$-norm.
(This construction also yields a $\delta \radius$-packing in
$\ell_\infty$ norm.)

\paragraphc{Upper bounding the mutual information} Let $\packrv$ be
drawn uniformly from the packing set $\packset = \{\pm
e_j\}_{j=1}^d$. With the sampling scheme in the previous paragraph, we
may provide the following upper bound on the mutual information
$\information(\channelrv_1, \ldots, \channelrv_n; \packrv)$ for any
non-interactive private distribution~\eqref{eqn:local-privacy-simple}:
\begin{lemma}
  \label{lemma:l1-information-bound}
  For any non-interactive $\diffp$-differentially private distribution
  $\channel$, we have
  \begin{equation*}
    \information(\channelrv_1, \ldots, \channelrv_n; \packrv) \le n
    \frac{1}{d} (e^{\diffp} - 1)^2
    \delta^2.
  \end{equation*}
\end{lemma}
\noindent See Appendix~\ref{appendix:l1-information-bound} for a proof.

\paragraphc{Applying testing inequalities}

Finally, we turn to application of the testing
inequalities. Lemma~\ref{lemma:l1-information-bound}, in conjunction with
the standard testing reduction and Fano's
inequality~\eqref{eqn:user-friendly-fano} with the choice $t = 0$, implies
that
\begin{equation*}
  \minimax_n(\theta(\mc{\statprob}), \ltwo{\cdot}^2, \diffp)
  \ge \frac{\radius^2 \delta^2}{2} \left(1 - \frac{\delta^2 n
    (e^\diffp - 1)^2 / d
    + \log 2}{\log(2d)}\right).
\end{equation*}
There is no loss of generality in assuming that $d \ge 2$, in which
case the choice
\begin{align*}
  \delta^2 = \min \bigg \{1, \frac{2 d \log (2d)}{(e^\diffp -
    1)^2 n} \bigg\}
\end{align*}
yields the lower bound.

It thus remains to
provide the upper bound.  In this case, we use the sampling
strategy~\eqref{eqn:linf-sampling} of
Section~\ref{sec:attainability-means}, noting that we may take the
bound $\sbound$ on $\linf{\channelrv}$ to be $\sbound = c \sqrt{d}
\radius / \diffp$ for a constant $c$.  Let $\optvar^*$ denote the true
mean, assumed to be $\nnzs$-sparse.  Now consider estimating
$\optvar^*$ by the $\ell_1$-regularized optimization problem
\begin{equation*}
  \what{\optvar} \defeq \argmin_{\optvar \in \R^d} \left\{ \frac{1}{2
    n} \ltwobigg{ \sum_{i = 1}^n (\channelrv_i - \optvar)}^2 + \lambda
  \lone{\optvar}\right\}.
\end{equation*}
Defining the error vector $W = \optvar^* - \frac{1}{n} \sum_{i=1}^n
\channelrv_i$, we claim that
\begin{equation}
  \label{eqn:lone-solution}
  \lambda \ge 2 \linf{W}
  ~~~ \mbox{implies that} ~~~
  \ltwos{\what{\optvar} - \optvar}
  \le 3 \lambda \sqrt{\nnzs}.
\end{equation}
This result is a consequence of standard results on sparse estimation
(e.g., \citet[Theorem 1 and Corollary 1]{NegahbanRaWaYu12}).

Now we note if $W_i = \optvar^* - \channelrv_i$, then $W = \frac{1}{n}
\sum_{i=1}^n W_i$. Letting $\phi_\diffp = \frac{e^\diffp + 1}{e^\diffp - 1}$,
by construction of the sampling
mechanism~\eqref{eqn:linf-sampling} we have
$\linf{W_i} \le c \sqrt{d} \radius \phi_\diffp$ for a constant $c$.
By Hoeffding's inequality and a union bound, we thus have for some
(different) universal constant $c$ that
\begin{equation*}
  \P(\linf{W} \ge t)
  \le 2d \exp\left(-c \frac{n t^2}{\radius^2 \phi_\diffp^2 d}\right)
  ~~ \mbox{for~}t \ge 0.
\end{equation*}
By taking $t^2 = \radius^2 \phi_\diffp^2 d (\log(2d) + \epsilon^2) / (c n)$,
we find that $\linf{W}^2 \le \radius^2 \phi_\diffp^2 d (\log(2d) +
\epsilon^2) / (c n \diffp^2)$ with probability at least $1 -
\exp(-\epsilon^2)$, which gives the claimed minimax upper bound by
appropriate choice of $\lambda = c \phi_\diffp \sqrt{d \log d / n}$ in
inequality~\eqref{eqn:lone-solution}.

\subsection{Proof of inequality~\eqref{eqn:mean-laplace-sucks}}
\label{sec:mean-laplace-sucks}

We prove the bound by an argument using the private form of Fano's
inequality~\eqref{eqn:user-friendly-fano} with $t = 0$ and replacing
$\statrv$ with $\channelrv$.  The proof makes use of the classical
Varshamov-Gilbert bound (\cite[Lemma 4]{Yu97}):

\begin{lemma}[Varshamov-Gilbert]
  There is a packing $\packset$ of the $d$-dimensional hypercube $\{-1,
  1\}^d$ of size $|\packset| \ge \exp(d/8)$ such that
  \begin{align*}
    \lone{\packval - \altpackval} \ge d/2 \quad \mbox{for all distinct
      pairs $\packval, \altpackval \in \packset$.}
  \end{align*}
\end{lemma}
Now, let $\delta \in [0, 1]$ and the distribution $\statprob_\packval$
be a point mass at $\delta \packval / \sqrt{d}$. Then
$\theta(\statprob_\packval) = \delta \packval / \sqrt{d}$ and
$\ltwo{\theta(\statprob_\packval) - \theta(\statprob_\altpackval)}^2
\ge \delta^2$. In addition, a calculation implies that if
$\marginprob_1$ and $\marginprob_2$ are $d$-dimensional
$\laplace(\kappa)$ distributions with means $\theta_1$ and $\theta_2$,
respectively, then
\begin{equation*}
  \dkl{\marginprob_1}{\marginprob_2}
  = \sum_{j = 1}^d \left(\exp(-\kappa |\theta_{1,j} - \theta_{2,j}|)
  + \kappa |\theta_{1,j} - \theta_{2,j}| - 1\right)
  \le \frac{\kappa^2}{2} \ltwo{\theta_1 - \theta_2}^2.
\end{equation*}
As a consequence, we have that under our Laplacian sampling scheme for
$\channelrv$ and with $\packrv$ chosen uniformly from $\packset$,
\begin{equation*}
  \information(\channelrv_1, \ldots, \channelrv_n; \packrv)
  \le \frac{1}{|\packset|^2}
  n \sum_{\packval, \altpackval \in \packset} \dkl{\marginprob_\packval}{
    \marginprob_{\altpackval}}
  \le \frac{n \diffp^2}{2 d |\packset|^2}
  \sum_{\packval, \altpackval \in \packset}
  \ltwo{(\delta / \sqrt{d})(\packval - \altpackval)}^2
  \le \frac{2n \diffp^2 \delta^2}{d}.
\end{equation*}
Now, applying Fano's inequality~\eqref{eqn:user-friendly-fano}, we find that
\begin{equation*}
  \inf_{\what{\optvar}} \sup_{\packval \in \packset}
  \E_{\statprob_\packval}\left[\ltwos{\what{\optvar}(\channelrv_1,
      \ldots, \channelrv_n) - \optvar(\statprob_\packval)}^2\right]
  \ge \frac{\delta^2}{4}
  \left(1 - \frac{2 n \diffp^2 \delta^2 / d + \log 2}{d/8}\right).
\end{equation*}
We may assume (based on our one-dimensional results in
Corollary~\ref{CorLocFamily}) w.l.o.g.\ that $d \ge 10$.  Taking
$\delta^2 = d^2 / (48 n \diffp^2)$ then implies the
result~\eqref{eqn:mean-laplace-sucks}.


\section{Proof of Theorem~\ref{theorem:sequential-interactive}}
\label{sec:proof-sequential-interactive}

The proof of this theorem combines the techniques we used in the
proofs of Theorems~\ref{theorem:master}
and~\ref{theorem:super-master}; the first handles interactivity, while
the techniques to derive the variational bounds are reminiscent of
those used in Theorem~\ref{theorem:super-master}.  Our first step is
to note a consequence of the independence structure in
Figure~\ref{fig:interactive-channel} essential to our tensorization
steps. More precisely, we claim that for any set $S \in
\sigma(\channeldomain)$,
\begin{equation}
  \label{eqn:help-tensorize}
  \marginprob_{\pm j}(Z_i \in S \mid \channelval_{1:i-1}) = \int
  \channel(Z_i \in S \mid \channelrv_{1:i-1} = \channelval_{1:i-1},
  \statrv_i = \statval) d \statprob_{\pm j,i}(\statval).
\end{equation}
We postpone the proof of this intermediate claim to the end of this
section.

Now consider the summed KL-divergences. Let $\marginprob_{\pm j,
  i}(\cdot \mid \channelval_{1:i-1})$ denote the conditional
distribution of $\channelrv_i$ under $\statprob_{\pm j}$, conditional
on $\channelrv_{1:i-1} = \channelval_{1:i-1}$. As in the proof of
Corollary~\ref{corollary:idiot-fano}, the chain rule for
KL-divergences~\cite[e.g.,][Chapter 5]{Gray90} implies
\begin{align*}
  \dkl{\marginprob^n_{+j}}{\marginprob^n_{-j}} & = \sum_{i=1}^n
  \int_{\channeldomain^{i-1}} \dkl{\marginprob_{+j}(\cdot \mid
    \channelval_{1:i-1})}{ \marginprob_{-j}(\cdot \mid
    \channelval_{1:i-1})} d
  \marginprob^{i-1}_{+j}(\channelval_{1:i-1}).
\end{align*}
For notational convenience in the remainder of the proof, let us
recall that the symmetrized KL divergence between measures $M$ and $M'$ is
$\dklsym{M}{M'} = \dkl{M}{M'} + \dkl{M'}{M}$.

Defining $\meanstatprob \defn 2^{-d} \sum_{\packval \in \packset}
\statprob_\packval^n$, we have $2 \meanstatprob = \statprob_{+j} +
\statprob_{-j}$ for each $j$ simultaneously, We also introduce
$\meanmarginprob(S) = \int \channel(S \mid \statval_{1:n})
d\meanmarginprob(\statval_{1:n})$, and let $\E_{\pm j}$ denote the
expectation taken under the marginals $\marginprob_{\pm j}$.  We then
have
\begin{align*}
  \lefteqn{\dkl{\marginprob^n_{+j}}{\marginprob^n_{-j}} +
    \dkl{\marginprob^n_{-j}}{\marginprob^n_{+j}}} \\ 
& = \sum_{i=1}^n \Big(\E_{+j}[ \dkl{\marginprob_{+j,i}(\cdot \mid
      \channelrv_{1:i-1})}{ \marginprob_{-j,i}(\cdot \mid
      \channelrv_{1:i-1})}] + \E_{-j}[ \dkl{\marginprob_{-j,i}(\cdot
      \mid \channelrv_{1:i-1})}{ \marginprob_{+j,i}(\cdot \mid
      \channelrv_{1:i-1})}]\Big) \\ 
& \le \sum_{i=1}^n \Big(\E_{+j}[\dklsym{\marginprob_{+j,i}(\cdot \mid
      \channelrv_{1:i-1})}{\marginprob_{-j,i}( \cdot \mid
      \channelrv_{1:i-1})}] + \E_{-j}[\dklsym{\marginprob_{+j,i}(\cdot
      \mid \channelrv_{1:i-1})}{\marginprob_{-j,i}( \cdot \mid
      \channelrv_{1:i-1})}]\Big) \\ 
& = 2 \sum_{i=1}^n \int_{\channeldomain^{i-1}}
  \dklsym{\marginprob_{+j,i}(\cdot \mid \channelval_{1:i-1})}{
    \marginprob_{-j,i}(\cdot \mid \channelval_{1:i-1})} d
  \meanmarginprob^{i-1}(\channelval_{1:i-1}),
\end{align*}
where we have used the definition of $\meanmarginprob$ and that
$2\meanstatprob = \statprob_{+j} + \statprob_{-j}$ for all $j$.
Summing over $j \in [d]$ yields
\begin{align}
\sum_{j=1}^d \dklsym{\marginprob_{+j}^n}{\marginprob_{-j}^n} & \le 2
\sum_{i=1}^n \int_{\channeldomain^{i-1}} \sum_{j=1}^d \underbrace{
  \dklsym{\marginprob_{+j,i}(\cdot \mid \channelval_{1:i-1})}{
    \marginprob_{-j,i}(\cdot \mid \channelval_{1:i-1})} }_{\eqdef
  \term_{j,i}} d \meanmarginprob^{i-1}(\channelval_{1:i-1}).
  \label{eqn:initial-sum-d}
\end{align}
We bound the underlined expression in inequality~\eqref{eqn:initial-sum-d},
whose elements we denote by $\term_{j,i}$.

Without loss of generality (as in the proof of
Theorem~\ref{theorem:super-master}), we may assume $\channeldomain$ is
finite, and that $\channeldomain = \{1, 2, \ldots, k\}$ for some
positive integer $k$. Using the probability mass functions
$\margindens_{\pm j,i}$ and omitting the index $i$ when it is clear
from context, Lemma~\ref{lemma:log-ratio-inequality} implies
\begin{align*}
  \term_{j,i}
  & = \sum_{\channelval = 1}^k \left(\margindens_{+j}(\channelval
  \mid \channelval_{1:i-1})
  - \margindens_{+j}(\channelval
  \mid \channelval_{1:i-1})\right)\log \frac{\margindens_{+j}(\channelval
  \mid \channelval_{1:i-1})}{\margindens_{-j}(\channelval
  \mid \channelval_{1:i-1})} \\
  & \le \sum_{\channelval = 1}^k \left(\margindens_{+j}(\channelval
  \mid \channelval_{1:i-1})
  - \margindens_{+j}(\channelval
  \mid \channelval_{1:i-1})\right)^2
  \frac{1}{\min\{\margindens_{+j}(\channelval
  \mid \channelval_{1:i-1}), \margindens_{-j}(\channelval
  \mid \channelval_{1:i-1})\}}.
\end{align*}
For each fixed $\channelval_{1:i-1}$, define the infimal measure
$\margincenter(\channelval \mid \channelval_{1:i-1}) \defeq \inf
\limits_{\statval \in \statdomain} \channeldensity(\channelval \mid
\statrv_i = \statval, \channelval_{1:i-1})$.  By construction, we have
$\min\{\margindens_{+j}(\channelval \mid \channelval_{1:i-1}),
\margindens_{-j}(\channelval \mid \channelval_{1:i-1})\} \ge
\margincenter(\channelval \mid \channelval_{1:i-1})$, and hence
\begin{equation*}
\term_{j,i}
  \le \sum_{\channelval = 1}^k \left(\margindens_{+j}(\channelval
  \mid \channelval_{1:i-1})
  - \margindens_{+j}(\channelval
  \mid \channelval_{1:i-1})\right)^2
  \frac{1}{\margindens^0(\channelval \mid \channelval_{1:i-1})}.
\end{equation*}
Recalling equality~\eqref{eqn:help-tensorize}, we have
\begin{align*}
  \margindens_{+j}(\channelval
  \mid \channelval_{1:i-1})
  - \margindens_{+j}(\channelval
  \mid \channelval_{1:i-1})
  & = \int_{\statdomain}
  \channeldensity(\channelval \mid \statval, \channelval_{1:i-1})
  (d \statprob_{+j,i}(\statval) - d \statprob_{-j,i}(\statval)) \\
  & = \margindens^0(\channelval \mid \channelval_{1:i-1})
  \int_{\statdomain}
  \left(\frac{\channeldensity(\channelval \mid \statval, \channelval_{1:i-1})}{
    \margindens^0(\channelval \mid \channelval_{1:i-1})}
  - 1 \right)
  (d \statprob_{+j,i}(\statval) - d \statprob_{-j,i}(\statval)).
\end{align*}

From this point, the proof is similar to that of
Theorem~\ref{theorem:super-master}.  Define the collection of
functions
\begin{equation*}
  \mc{F}_\diffp \defeq \{f : \statdomain \times \channeldomain^i \to
     [0, e^\diffp - 1]\}.
\end{equation*}
Using the definition of differential privacy, we have
$\frac{\channeldensity(\channelval \mid \statval,
  \channelval_{1:i-1})}{ \margindens^0(\channelval \mid
  \channelval_{1:i-1})} \in [1, e^\diffp]$, so there exists $f \in
\mc{F}_\diffp$ such that
\begin{align*}
  \sum_{j=1}^d \term_{j,i}
  & \le \sum_{j=1}^d \sum_{\channelval = 1}^k
  \frac{\left(\margindens^0(\channelval \mid \channelval_{1:i-1})\right)^2}{
    \margindens^0(\channelval \mid \channelval_{1:i-1})}
  \bigg(\int_{\statdomain} f(\statsample, \channelval, \channelval_{1:i-1})
  (d \statprob_{+j,i}(\statval) - d \statprob_{-j,i}(\statval))\bigg)^2 \\
  & = \sum_{\channelval = 1}^k
  \margindens^0(\channelval \mid \channelval_{1:i-1})
  \sum_{j=1}^d
  \bigg(\int_{\statdomain} f(\statsample, \channelval, \channelval_{1:i-1})
  (d \statprob_{+j,i}(\statval) - d \statprob_{-j,i}(\statval))\bigg)^2.
\end{align*}
Taking a supremum over $\mc{F}_\diffp$, we find the further upper bound
\begin{align*}
  \sum_{j=1}^d \term_{j,i} \le \sum_{\channelval = 1}^k
  \margindens^0(\channelval \mid \channelval_{1:i-1})
  \sup_{f \in \mc{F}_\diffp} \sum_{j=1}^d \bigg(
  \int_{\statdomain} f(\statsample, \channelval, \channelval_{1:i-1})
  (d \statprob_{+j,i}(\statval) - d \statprob_{-j,i}(\statval))\bigg)^2.
\end{align*}
The inner supremum may be taken independently
of $\channelval$ and $\channelval_{1:i-1}$, so we
rescale by $(e^\diffp - 1)$ to obtain
our penultimate inequality
\begin{align*}
  \lefteqn{\sum_{j=1}^d  \dklsym{\marginprob_{+j,i}(\cdot
      \mid \channelval_{1:i-1})}{
      \marginprob_{-j,i}(\cdot \mid \channelval_{1:i-1})}} \\
  & \qquad ~ \le (e^\diffp - 1)^2 \sum_{\channelval = 1}^k
  \margindens^0(\channelval \mid \channelval_{1:i-1})
  \sup_{\optdens \in \linfset(\statdomain)}
  \sum_{j=1}^d \bigg( \int_{\statdomain} \optdens(\statval)
  (d \statprob_{+j,i}(\statval) - d \statprob_{-j,i}(\statval))\bigg)^2.
\end{align*}
Noting that $\margindens^0$ sums to a quantity less than or equal to one and substituting
the preceding expression in inequality~\eqref{eqn:initial-sum-d} completes the
proof.\\

\noindent Finally, we return to prove our intermediate marginalization
claim~\eqref{eqn:help-tensorize}.  We have that
\begin{align*}
  \marginprob_{\pm j}(Z_i \in S \mid \channelval_{1:i-1})
  & = \int \channel(Z_i \in S \mid \channelval_{1:i-1}, \statval_{1:n})
  d \statprob_{\pm j}(\statval_{1:n} \mid \channelval_{1:i-1}) \\
  & \stackrel{(i)}{=}
  \int \channel(Z_i \in S \mid \channelval_{1:i-1}, \statval_i)
  d \statprob_{\pm j}(\statval_{1:n} \mid \channelval_{1:i-1}) \\
  & \stackrel{(ii)}{=} \int \channel(Z_i \in S \mid
  \channelrv_{1:i-1} = \channelval_{1:i-1}, \statrv_i = \statval)
  d \statprob_{\pm j,i}(\statval),
\end{align*}
where equality~(\emph{i}) follows by the assumed conditional independence structure
of $\channel$ (recall Figure~\ref{fig:interactive-channel}) and equality~(\emph{ii})
is a consequence of the independence of $\statrv_i$ and $\channelrv_{1:i-1}$
under $\statprob_{\pm j}$. That is, we have $\statprob_{+j}(\statrv_i \in S
\mid \channelrv_{1:i-1} = \channelval_{1:i-1}) = \statprob_{+j,i}(S)$ by the
definition of $\statprob_\packval^n$ as a product and that $\statprob_{\pm j}$
are a mixture of the products $\statprob_\packval^n$.


\section{Proof of logistic regression lower bound}
\label{sec:proof-full-logistic-lower-bound}

In this section, we prove the lower bounds in
Corollary~\ref{corollary:full-logistic-lower-bound}. Before proving the
bounds, however, we outline our technique, which borrows from that in
Section~\ref{sec:proofs-big-mean-estimation}, and which we also use to prove
the lower bounds on density estimation.  The outline is as follows:
\begin{enumerate}[(1)]
\item As in step~\eqref{step:standard} of
  Section~\ref{sec:proofs-big-mean-estimation}, our first step is a standard
  reduction using the sharper version of Assouad's method
  (Lemma~\ref{lemma:sharp-assouad}) from estimation to a multiple binary
  hypothesis testing problem. Specifically, we perform a (essentially
  standard) reduction of the form~\eqref{eqn:risk-separation}.
\item Having constructed appropriately separated binary hypothesis tests, we
  use apply Theorem~\ref{theorem:sequential-interactive} via
  inequality~\eqref{eqn:sharp-assouad-kled} to control the testing error in
  the binary testing problem. Applying the theorem requires bounding certain
  suprema related to the covariance structure of randomly selected elements of
  $\packset = \{-1, 1\}^d$, as in the arguments in
  Section~\ref{sec:proofs-big-mean-estimation}. This
  is made easier by the
  symmetry of the binary hypothesis testing problems.
\end{enumerate}

With this outline in mind, we turn to the proofs of
inequality~\eqref{eqn:full-logistic-lower-bound}.  Our first step is to
provide a lower bound of the form~\eqref{eqn:risk-separation}, giving a
Hamming separation for the squared error. To that end, fix $\delta \in [0,
  1]$, and let $\packset = \{-1, 1\}^d$. Then for $\packval \in \packset$,
we set $\optvar_\packval = \delta \packval$, and define the base
distribution $\statprob_\packval$ on the pair $(\statrv, Y)$ as follows.
Under the distribution $\statprob_\packval$, we take $\statrv \in \{-1,
1\}^d$ with coordinates that are independent conditional on $Y$, and
\begin{equation*}
  \statprob_\packval(Y = y) = \half
  ~ \mbox{for~} y \in \{-1, 1\}
  ~~ \mbox{and} ~~
  \statprob_\packval(\statrv_j = \statval_j \mid y)
  = \frac{\exp(\half y \statval_j \optvar_j)}{
    \exp(-\half \statval_j \optvar_j)
    + \exp(\half \statval_j \optvar_j)}
  = \frac{e^\frac{\delta y \statval_j \packval_j}{2}}{
    e^\frac{\delta}{2} + e^\frac{-\delta}{2}}.
\end{equation*}
That is, $Y \mid \statrv = \statval$ has p.m.f.\
\begin{equation*}
  p(y \mid \statval, \theta)
  = \frac{\statprob_\packval(\statval \mid y)}{
    \statprob_\packval(\statval \mid y)
    + \statprob_\packval(\statval \mid -y)}
  = \frac{\exp(\half y \statval^\top \optvar)}{
    \exp(\half y \statval^\top \optvar)
    + \exp(-\half y \statval^\top \optvar)}
  = \frac{1}{1 + \exp(-y \statval^\top \optvar)},
\end{equation*}
the standard logistic model. Moreover, we have that with the logistic loss
$\loss(\theta; \statval, y) = \log(1 + e^{-y \theta^\top \statval})$, we
evidently have $\theta_\packval = \argmax_\theta
\E_{P_\packval}[\loss(\theta; \statrv, Y)]$. Then
for any estimator $\what{\optvar}$, by defining $\what{\packval}_j =
\sign(\what{\optvar}_j)$ for $j \in [d]$, we have
the Hamming separation
\begin{equation*}
  d \wedge \ltwos{\what{\optvar} - \optvar_\packval}^2
  \ge \delta^2 \sum_{j = 1}^d \indic{\what{\packval}_j \neq
    \packval_j}.
\end{equation*}
Thus, by the sharper variant~\eqref{eqn:sharp-assouad-kled}
of Assouad's Lemma, we obtain
\begin{align}
  \max_{\packval \in \packset} \E_{\statprob_\packval}[
    d \wedge \ltwos{\what{\optvar} - \optvar_\packval}^2]
  & \ge \frac{d \delta^2}{2}
  \left[1 - \bigg(\frac{1}{4 d} \sum_{j=1}^d
    \dkl{\marginprob_{+j}^n}{\marginprob_{-j}^n}
    + \dkl{\marginprob_{-j}^n}{\marginprob_{+j}^n}\bigg)^\half\right].
  \label{eqn:assouad-logreg}
\end{align}

We now apply
Theorem~\ref{theorem:sequential-interactive}, which requires
bounding sums of integrals
$\int \optdens (d \statprob_{+j} -
d\statprob_{-j})$, where $\statprob_{+j}$ is defined in
expression~\eqref{eqn:paired-mixtures} (that is, $P_{+j} =
\frac{1}{2^{d-1}} \sum_{\packval : \packval_j = 1} P_\packval$ and similarly
for $P_{-j}$).
We claim the following inequality:
\begin{equation}
  \label{eqn:logreg-sup}
  \sup_{\linf{\optdens} \le 1}
  \sum_{j = 1}^d \int_{\mc{X}, \mc{Y}}
  \left(\optdens(x, y) (dP_{+j}(x, y) - dP_{-j}(x, y))\right)^2
  \le 2 \left(\frac{e^\delta - 1}{e^{\delta} + 1}\right)^2
  \le \frac{\delta^2}{2}.
\end{equation}
Temporarily deferring the proof of inequality~\eqref{eqn:logreg-sup}, let us
see how it yields the bound~\eqref{eqn:full-logistic-lower-bound} we desire in
the corollary. Indeed, Theorem~\ref{theorem:sequential-interactive}
immediately gives
\begin{equation*}
  \sum_{j=1}^d
  \dkl{\marginprob_{+j}^n}{\marginprob_{-j}^n}
  + \dkl{\marginprob_{-j}^n}{\marginprob_{+j}^n}
  \le n (e^\diffp - 1)^2 \delta^2,
\end{equation*}
and applying inequality~\eqref{eqn:assouad-logreg} we find that
\begin{equation*}
  \max_{\packval \in \packset} \E_{\statprob_\packval}[
    d \wedge \ltwos{\what{\optvar} - \optvar_\packval}^2]
  \ge \frac{d \delta^2}{2}
  \left[1 - \bigg(\frac{n (e^\diffp - 1)^2 \delta^2}{4d}
    \bigg)^\half\right].
\end{equation*}
Choosing $\delta^2 = 
\min \{\frac{d}{n(e^\diffp - 1)^2}, 1\}$ yields
inequality~\eqref{eqn:full-logistic-lower-bound}, proving the corollary.

We now return to demonstrate our claim~\eqref{eqn:logreg-sup}.  Let
$\statval_{\setminus j} \in \{-1, 1\}^{d-1}$ be the vector $\statval$ with
its $j$th index removed, and similarly for $\packval_{\setminus j}$. Then we
have
\begin{align*}
  P_{+j}(x \mid y)
  = \frac{1}{2^{d-1}}
  \sum_{\packval : \packval_j = 1}
  P_\packval(x \mid y)
  & = \frac{1}{2^{d-1}}
  \sum_{\packval_{\setminus j} \in \{-1, 1\}^{d-1}}
  \frac{\exp(\half \delta \packval_{\setminus j}^\top x_{\setminus j}
    y )}{(e^{\delta/2} + e^{-\delta/2})^{d-1}}
  \frac{\exp(\half \delta y x_j \packval_j)}{e^{\delta/2} + e^{-\delta/2}} \\
  & =
  \frac{1}{2^{d-1} (e^{\delta/2} + e^{-\delta/2})} 
  \exp\left(\half \delta y x_j \packval_j\right),
\end{align*}
where $\packval_j = 1$ in this case. The result is similar for $P_{-j}$,
because $x_{\setminus j}$ is marginally uniform on $\{-1, 1\}^{d-1}$ even
conditional on $y$. Thus we obtain
\begin{align*}
  2^d \left(P_{+j}(x, y) - P_{-j}(x, y)\right)
  & =
  2^d (P_{+j}(x \mid y) - P_{-j}(x \mid y)) P(y) \\
  & =
  \frac{1}{e^{\delta/2} + e^{-\delta/2}}
  \sign(y x_j) \left[e^{\delta/2} - e^{-\delta/2}\right].
\end{align*}
Incorporating this into the variational quantity~\eqref{eqn:logreg-sup},
we have
\begin{align}
  \nonumber
  \lefteqn{\sup_{\linf{\optdens} \le 1}
    \sum_{j = 1}^d \int_{\mc{X}, \mc{Y}}
    \left(\optdens(x, y) (dP_{+j}(x, y) - dP_{-j}(x, y))\right)^2} \\
  & = \sup_{\linf{\optdens} \le 1}
  \sum_{j = 1}^d \left(
  \frac{1}{2^{d}
  (e^{\delta/2} + e^{-\delta/2})} \sum_x \left(
  \optdens(x, 1)
  \left(e^\frac{\delta x_j}{2}
  - e^\frac{-\delta x_j}{2}\right) 
  + \optdens(x, -1)
  \left(e^\frac{-\delta x_j}{2}
  - e^\frac{\delta x_j}{2}\right)\right)\right)^2 \nonumber \\
  & \le \frac{2}{4^d (e^{\delta/2} + e^{-\delta/2})^2}
  \sup_{\linf{\optdens} \le 1} \sum_{j = 1}^d
  \left(\sum_x \optdens(x) \left(e^\frac{\delta x_j}{2}
  - e^{-\frac{\delta x_j}{2}}\right)\right)^2 \nonumber \\
  & = \frac{2}{4^d}
  \left(\frac{e^\delta - 1}{e^{\delta} + 1}\right)^2
  \sup_{\linf{\optdens} \le 1} \sum_{j = 1}^d
  \left(\sum_{x \in \{-1, 1\}^d} \optdens(x) x_j\right)^2,
  \label{eqn:want-burritos}
\end{align}
where the inequality is a consequence of Jensen's inequality and symmetry.

But now we can apply standard matrix inequalities to the
quantity~\eqref{eqn:want-burritos} because of its symmetry.
Indeed, we may identify $\optdens$ with vectors
$\optdens \in \R^{2^d}$, and let vectors $w_j \in \{-1, 1\}^{2^d}$ be
indexed by $x \in \{-1, 1\}^d$ where $[w_j]_x = \sign(x_j)$.
Then $w_j^\top w_k = 0$ for $j \neq k$, and we have
\begin{align*}
  \sup_{\linf{\optdens} \le 1} \sum_{j = 1}^d
  \left(\sum_{x \in \{-1, 1\}^d} \optdens(x) x_j\right)^2
  & = \sup_{\optdens \in \R^{2^d}, \linf{\optdens} \le 1}
  \sum_{j = 1}^d \optdens^\top w_j w_j^\top \optdens \\
  & \le \sup_{\optdens \in \R^{2^d},
    \ltwo{\optdens} \le \sqrt{2^d}}
  \sum_{j = 1}^d \optdens^\top w_j w_j^\top \optdens
  = 2^d \opnorm{\sum_{j = 1}^d w_j w_j^\top}
  = 4^d,
\end{align*}
because the $w_j$ are orthogonal.
Returning to inequality~\eqref{eqn:want-burritos}, we have
\begin{equation*}
  \sup_{\linf{\optdens} \le 1}
  \sum_{j = 1}^d \int_{\mc{X}, \mc{Y}}
  \left(\optdens(x, y) (dP_{+j}(x, y) - dP_{-j}(x, y))\right)^2
  \le 2 \left(\frac{e^\delta - 1}{e^{\delta} + 1}\right)^2
  \le \frac{\delta^2}{2}.
\end{equation*}
This is the claimed inequality~\eqref{eqn:logreg-sup}.


\section{Proof of Corollary~\ref{CorDensity}}
\label{sec:proof-density-estimation}

In this section, we provide the proof of Corollary~\ref{CorDensity} on
density estimation from Section~\ref{sec:density-estimation}. We defer the
proofs of more technical results to later appendices. We provide the proof
of the minimax lower bound in Section~\ref{sec:proof-density-lower} and the
upper bound in Section~\ref{sec:proof-density-upper-bound}. Throughout all
proofs, we use $c$ to denote a constant whose value may change from line to
line.

\subsection{Proof of the lower
  bound~\eqref{eqn:private-density-estimation-rate}}
\label{sec:proof-density-lower}

As with our proof for logistic regression, the argument follows the general
outline described at the beginning of
Section~\ref{sec:proof-full-logistic-lower-bound}.  We remark that our proof
is based on an explicit construction of densities identified with corners of
the hypercube, a more classical approach than the global metric entropy
approach of~\citet{YangBa99}.  We use the local packing
approach since it is better suited to the privacy constraints and
information contractions that we have developed.  In comparison with our
proofs of previous propositions, the construction of a suitable packing of
$\densclass$ is somewhat more challenging: the identification of densities
with finite-dimensional vectors, which we require for our application of
Theorem~\ref{theorem:sequential-interactive}, is not immediately obvious.
In all cases, we guarantee that our density functions $f$ belong to the
trigonometric Sobolev space, so we may work directly with smooth density
functions $f$.

\begin{figure}
  \begin{center}
    \begin{tabular}{cc}
      \psfrag{g}{$g_1$}
      \includegraphics[height=.32\columnwidth]{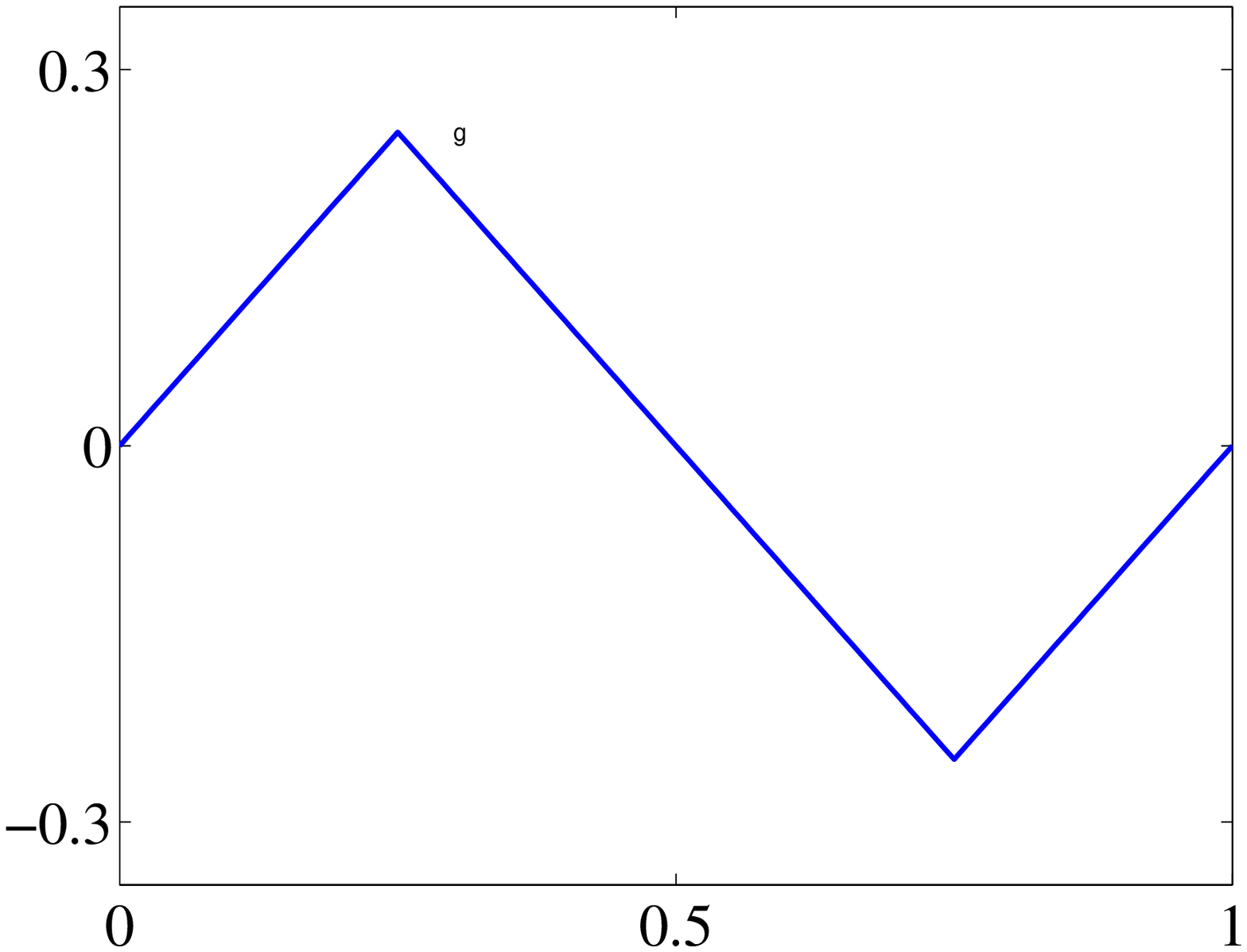} &
      \psfrag{g}{$g_2$}
      \includegraphics[height=.32\columnwidth]{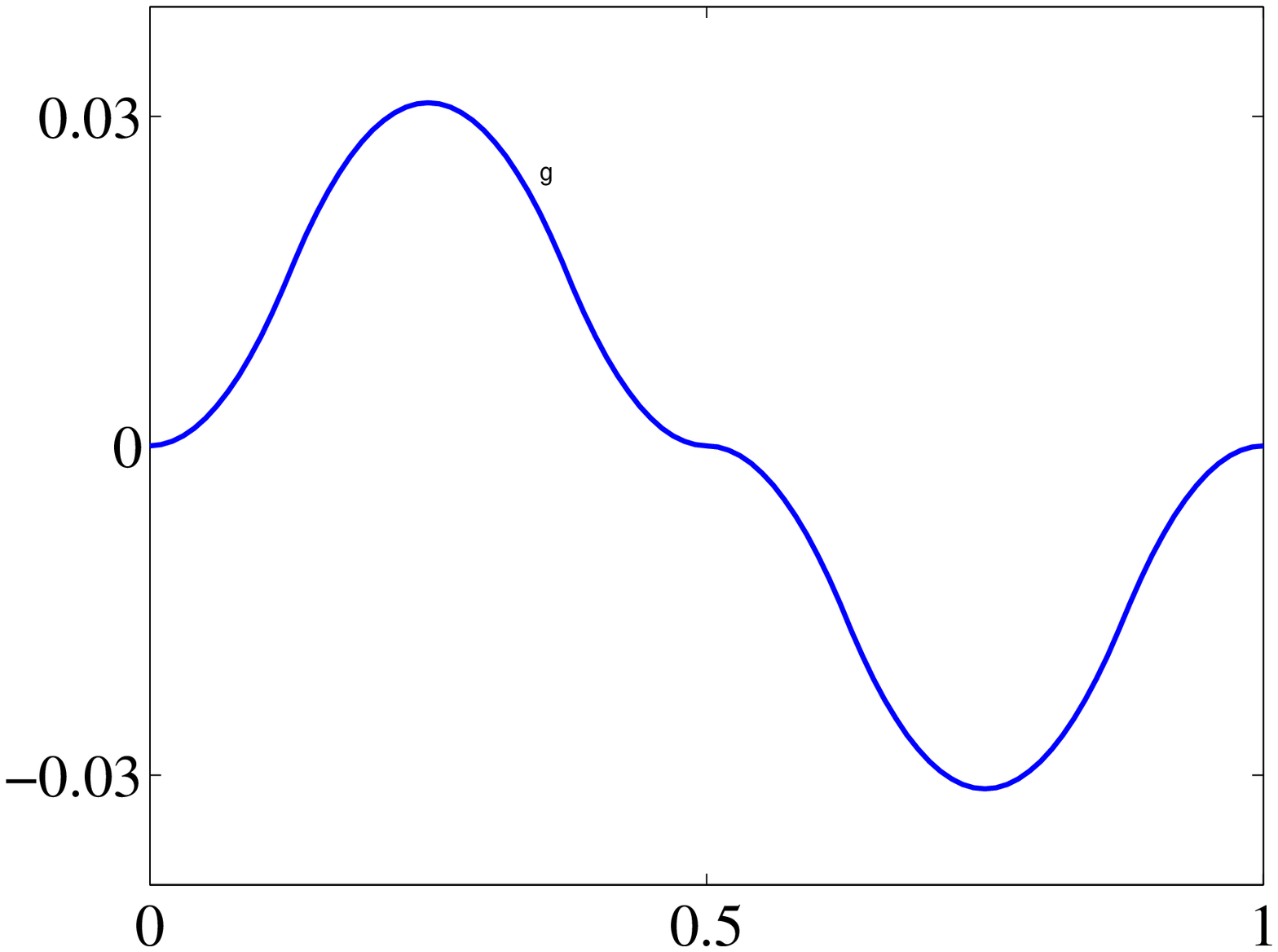} \\
        (a) & (b)
    \end{tabular}
    \caption{\label{fig:bump-func} Panel (a): illustration of
      $1$-Lipschitz continuous bump function $g_1$ used to pack
      $\densclass$ when $\numderiv = 1$.  Panel (b): bump function
      $g_2$ with $|g_2''(\statsample)| \le 1$ used to pack
      $\densclass$ when $\numderiv = 2$.
    }
  \end{center}
\end{figure}

\paragraphc{Constructing well-separated densities} 

We begin by describing a standard framework for defining local
packings of density functions.  Let $g_\numderiv: [0,1] \rightarrow
\real$ be a function satisfying the following properties:

\begin{enumerate}[(a)]
\item The function $g_\numderiv$ is $\numderiv$-times differentiable
  with
  \begin{equation*}
    0 = g_\numderiv^{(i)}(0) = g_\numderiv^{(i)}(1/2)
    = g_\numderiv^{(i)}(1)
    ~~ \mbox{for~all~} i < \numderiv.
  \end{equation*}
\item The function $g_\numderiv$ is centered with $\int_0^1
  g_\numderiv(x) dx = 0$, and there exist constants $c, c_{1/2} > 0$
  such that
\begin{equation*}
  \int_0^{1/2} g_\numderiv(x) dx = -\int_{1/2}^1 g_\numderiv(x) dx =
  c_{1/2} ~~~ \mbox{and} ~~~ \int_0^1
  \left(g^{(i)}_\numderiv(x)\right)^2 dx \ge c ~~ \mbox{for~all~} i <
  \numderiv.
\end{equation*}
\item The function $g_\numderiv$ is nonnegative on $[0, 1/2]$ and
  non-positive on $[1/2, 1]$, and Lebesgue measure is absolutely
  continuous with respect to the measures $G_j, j = 1, 2$, given by
  \begin{equation}
    \label{eqn:absolute-continuity-bumps}
    G_1(A) = \int_{A \cap [0, 1/2]} g_\numderiv(x) dx ~~~ \mbox{and} ~~~
    G_2(A) = -\int_{A \cap [1/2, 1]} g_\numderiv(x) dx.
  \end{equation}
\item Lastly, for almost every $x \in [0, 1]$, we have
  $|g^{(\numderiv)}_\numderiv(x)| \le 1$ and $|g_\numderiv(x)| \le 1$.
\end{enumerate}
\noindent As illustrated in Figure~\ref{fig:bump-func}, the functions
$g_\numderiv$ are smooth ``bump'' functions.

Fix a positive integer $\numbin$ (to be specified in the sequel).  Our first
step is to construct a family of ``well-separated'' densities for which we can
reduce the density estimation problem to one of identifying corners of a
hypercube, which allows application of
Lemma~\ref{lemma:sharp-assouad}. Specifically, we must exhibit a condition
similar to the separation condition~\eqref{eqn:risk-separation}. For each $j
\in \{1, \ldots, \numbin\}$ define the function
\begin{align*}
  g_{\numderiv, j}(x) & \defeq \frac{1}{\numbin^\numderiv} \, g_\beta
  \left(\numbin \Big( x - \frac{j - 1}{\numbin}\Big)\right) \indic{ x
    \in \left[\textstyle{\frac{j-1}{\numbin}, \frac{j}{\numbin}}\right]}.
\end{align*}
Based on this definition, we define the family of
densities
\begin{equation}
  \label{eqn:density-packer}
  \bigg\{ f_\packval \defeq 1 + \sum_{j=1}^\numbin \packval_j
  g_{\numderiv, j} ~~~ \mbox{for}~ \packval \in \packset \bigg\}
  \; \subseteq \densclass.
\end{equation}
It is a standard fact~\cite{Yu97,Tsybakov09} that for any $\packval
\in \packset$, the function $f_\packval$ is $\numderiv$-times
differentiable, satisfies $|f^{(\numderiv)}(x)| \le 1$ for all $x$.
Now, based on some density $f \in \densclass$, let us define the
sign vector $\maptocube(f) \in \{-1, 1\}^\numbin$ to have entries
\begin{equation*}
  \maptocube_j(f)
  \defeq \argmin_{s \in \{-1, 1\}}
  \int_{[\frac{j-1}{\numbin}, \frac{j}{\numbin}]}
  \left(f(\statval) - s g_{\numderiv,j}(\statval)\right)^2 d\statval.
\end{equation*}
Then by construction of the $g_\numderiv$ and $\maptocube$, we have for a
numerical constant $c$ (whose value may depend on $\numderiv$) that
\begin{equation*}
  \ltwo{f - f_\packval}^2
  \ge c \sum_{j = 1}^\numbin \indic{\maptocube_j(f) \neq \packval_j}
  \int_{[\frac{j-1}{\numbin},
      \frac{j}{\numbin}]} (g_{\numderiv,j}(x))^2 dx
  = \frac{c}{\numbin^{2 \numderiv + 1}}
  \sum_{j = 1}^\numbin \indic{\maptocube_j(f) \neq \packval_j}.
\end{equation*}
By inspection, this is the Hamming separation required in
inequality~\eqref{eqn:risk-separation},
whence the sharper version~\eqref{eqn:sharp-assouad-kled}
of Assouad's Lemma~\ref{lemma:sharp-assouad} gives the result
\begin{equation}
  \label{eqn:sharp-assouad-density}
  \minimax_n \left(\densclass[1], \ltwo{\cdot}^2, \diffp \right)
  \ge \frac{c}{\numbin^{2 \numderiv}}
  \left[1 - \bigg(\frac{1}{4 \numbin}
    \sum_{j = 1}^\numbin (\dkl{\marginprob_{+j}^n}{\marginprob_{-j}^n}
    + \dkl{\marginprob_{-j}^n}{\marginprob_{+j}^n})\bigg)^\half\right],
\end{equation}
where we have defined $\statprob_{\pm j}$ to be the probability distribution
associated with the averaged densities $f_{\pm j} = 2^{1-\numbin}
\sum_{\packval : \packval_j = \pm 1} f_\packval$.


\paragraphc{Applying divergence inequalities}

Now we must control the summed KL-divergences. To do so,
we note that by the construction~\eqref{eqn:density-packer},
symmetry implies that
\begin{equation}
  f_{+j} = 1 + g_{\numderiv,j} ~~~ \mbox{and} ~~~
  f_{-j} = 1 - g_{\numderiv,j}
  ~~~ \mbox{for each ~} j \in [\numbin].
  \label{eqn:f-plus-j}
\end{equation}
We then obtain the following result, which bounds the averaged
KL-divergences.
\begin{lemma}
  \label{lemma:density-kls}
  For any $\diffp$-locally private conditional distribution
  $\channel$, the summed KL-divergences are bounded as
  \begin{equation*}
  \sum_{j=1}^\numbin \left(\dkl{\marginprob_{+j}^n}{\marginprob_{-j}^n}
  + \dkl{\marginprob_{+j}^n}{\marginprob_{-j}^n}\right)
  \le
  4 c_{1/2}^2 \, n
  \frac{(e^\diffp - 1)^2}{\numbin^{2 \numderiv + 1}}.
  \end{equation*}
\end{lemma}
\noindent
The proof of this lemma is fairly involved, so we defer it to
Appendix~\ref{appendix:density-kls}.  We note that, for $\diffp \le 1$, we
have $(e^\diffp - 1)^2 \le 3 \diffp^2$, so we may replace the bound in
Lemma~\ref{lemma:density-kls} with the quantity $c n \diffp^2 / \numbin^{2
  \numderiv + 1}$ for a constant $c$.  We remark that standard
divergence bounds using Assouad's lemma~\cite{Yu97,Tsybakov09} provide a bound
of roughly $n / \numbin^{2 \numderiv}$; our bound is thus essentially a factor
of the ``dimension'' $\numbin$ tighter.

The remainder of the proof is an application of
inequality~\eqref{eqn:sharp-assouad-density}. In particular,
by applying Lemma~\ref{lemma:density-kls}, we find that for
any $\diffp$-locally private channel
$\channelprob$, there are constants $c_0, c_1$ (whose values
may depend on $\numderiv$) such that
\begin{equation*}
  \minimax_n\left(\densclass, \ltwo{\cdot}^2, \channelprob\right) \ge
  \frac{c_0}{\numbin^{2 \numderiv}}
  \left[
    1 - \left(\frac{c_1 n \diffp^2}{\numbin^{2 \numderiv + 2}}\right)^\half
  \right].
\end{equation*}
Choosing $\numbin_{n, \diffp, \numderiv} = \left(4 c_1 n \diffp^2
\right )^{\frac{1}{2 \numderiv + 2}}$ ensures that the quantity inside
the parentheses is at least $1/2$. Substituting for
$\numbin$ in the preceding display proves the proposition.

\subsection{Proof of the upper
  bound~\eqref{eqn:private-density-estimation-rate}}
\label{sec:proof-density-upper-bound}

We begin by fixing $\numbin \in \N$; we will optimize the choice of
$\numbin$ shortly.  Recall that, since $f \in \densclass[\lipconst]$, we
have $f = \sum_{j=1}^\infty \optvar_j \basisfunc_j$ for $\optvar_j = \int f
\basisfunc_j$. Thus we may define $\overline{\channelrv}_j = \frac{1}{n}
\sum_{i=1}^n \channelrv_{i,j}$ for each $j \in \{1, \ldots, \numbin\}$, and
we have
\begin{equation*}
  \ltwos{\what{f} - f}^2
  = \sum_{j=1}^\numbin (\optvar_j - \overline{\channelrv}_j)^2
  + \sum_{j = \numbin + 1}^\infty \optvar_j^2.
\end{equation*}
Since $f \in \densclass[\lipconst]$, we are guaranteed that
$\sum_{j=1}^\infty j^{2 \numderiv} \optvar_j^2 \leq \lipconst^2$, and
hence
\begin{equation*}
\sum_{j > \numbin} \optvar_j^2 = \sum_{j > \numbin} j^{2 \numderiv}
\frac{\optvar_j^2}{j^{2 \numderiv}} \le \frac{1}{\numbin^{2
    \numderiv}} \sum_{j > \numbin} j^{2 \numderiv} \optvar_j^2 \le
\frac{1}{\numbin^{2 \numderiv}} \lipconst^2.
\end{equation*}
For the indices $j \le \numbin$, we note that by assumption,
$\E[\channelrv_{i,j}] = \int \basisfunc_j f = \optvar_j$,
and since $|\channelrv_{i,j}| \le \sbound$, we have 
\begin{equation*}
  \E\left[(\optvar_j - \overline{Z}_j)^2\right] = \frac{1}{n}
  \var(Z_{1,j}) \le \frac{\sbound^2}{n} =
  \frac{\orthbound^2}{c_\numbin} \, \frac{\numbin}{n} \,
  \left(\frac{e^\diffp + 1}{e^\diffp - 1}\right)^2,
\end{equation*}
where $c_\numbin = \Omega(1)$ is the constant in
expression~\eqref{eqn:size-infinity-channel}.  Putting together the pieces,
the mean-squared $L^2$-error is upper bounded as
\begin{equation*}
\E_f\left[\ltwos{\what{f} - f}^2\right] \le c \left(\frac{\numbin^2}{n
  \diffp^2} + \frac{1}{\numbin^{2 \numderiv}}\right),
\end{equation*}
where $c$ is a constant depending on $\orthbound$, $c_\numbin$, and
$\lipconst$.  Choose $\numbin = (n \diffp^2)^{1 / (2 \numderiv +
  2)}$ to complete the proof.




\section{Information bounds}
\label{appendix:lemmas-mutual-information}

In this appendix, we collect the proofs of lemmas providing mutual
information and KL-divergence bounds.


\subsection{Proof of Lemma~\ref{lemma:linf-info-bound}}
\label {sec:linf-info-bound}

Our strategy is to apply Theorem~\ref{theorem:super-master} to bound the
mutual information.  Without loss of generality, we may assume that $\radius
= 1$ so the set $\statdomain = \{\pm e_j\}_{j=1}^k$, where $e_j \in
\R^d$. Thus, under the notation of Theorem~\ref{theorem:super-master}, we
may identify vectors $\optdens \in L^\infty(\statdomain)$ by vectors
$\optdens \in \R^{2k}$. Noting that $\overline{\packval} =
\frac{1}{|\packset|} \sum_{\packval \in \packset} \packval = 0$ is the mean
element of the ``packing'' set by our construction, the linear functional
$\varphi_\packval$ defined in Theorem~\ref{theorem:super-master} is
\begin{align*}
  \varphi_\packval(\optdens) & = \frac{1}{2k} \sum_{j=1}^k
  \left[\frac{\delta}{2} \optdens(e_j) \packval_j - \frac{\delta}{2}
    \optdens(-e_j) \packval_j \right] = \frac{\delta}{4k}
  \optdens^\top \left[\begin{matrix} I_{k \times k} & 0_{k \times d-k}
      \\ -I_{k \times k} & 0_{k \times d-k} \end{matrix} \right]
  \packval.
\end{align*}
Define the matrix
\begin{equation*}
  A \defeq \left[\begin{matrix} I_{k \times k} & 0_{k \times d-k}
      \\ -I_{k \times k} & 0_{k \times d-k} \end{matrix} \right]
  \in \{-1, 0, 1\}^{2k \times d}.
\end{equation*}
Then we have that
\begin{align}
  \frac{1}{|\packset|} \sum_{\packval \in \packset}
  \varphi_\packval(\optdens)^2
  & = \frac{\delta^2}{(4k)^2} \optdens^\top A
  \frac{1}{|\packset|}
  \sum_{\packval \in \packset} \packval \packval^\top
  A^\top \optdens
  = \frac{\delta^2}{(4 k)^2} \optdens^\top A \cov(\packrv) A^\top \optdens
  \nonumber \\
  & = \frac{\delta^2}{(4k)^2}
  \optdens^\top A A^\top \optdens
  = \left(\frac{\delta}{4k}\right)^2
  \optdens^\top\left[\begin{matrix} I_{k \times k} & -I_{k \times k}
      \\ -I_{k \times k} & I_{k \times k} \end{matrix} \right]
  \optdens.
  \label{eqn:spectral-linf-bound}
\end{align}
Here we have used that $A \cov(\packrv) A^\top = A I_{d \times d} A^\top$
by the fact that $\packset = \{-1, 1\}^k \times \{0\}^{d - k}$.

We complete our proof using the bound~\eqref{eqn:spectral-linf-bound}. The
operator norm of the matrix specified in~\eqref{eqn:spectral-linf-bound} is
2.  As a consequence, since we have the containment
\begin{equation*}
  \linfset = \left\{
  \optdens \in \R^{2k} : \linf{\optdens} \le 1 \right\}
  \subset \left\{\optdens \in \R^{2k}
  : \ltwo{\optdens}^2 \le 2 k \right\},
\end{equation*}
we have the inequality
\begin{equation*}
  \sup_{\optdens \in \linfset}
  \frac{1}{|\packset|}
  \sum_{\packval \in \packset} \varphi_\packval(\optdens)^2
  \le \frac{\delta^2}{16 k^2} \cdot 2 \cdot 2k
  = \frac{1}{4} \, \frac{\delta^2}{k}.
\end{equation*}
Applying Theorem~\ref{theorem:super-master} completes the proof.


\subsection{Proof of Lemma~\ref{lemma:digit-info-bound}}
\label{sec:proof-digit-info-bound}

Our strategy is to apply Theorem~\ref{theorem:super-master} to bound the
mutual information. Because the mutual information is independent of the
radius of the set $\statdomain$, we may assume without loss of generality
that $\statdomain = \{-1, 1\}^k$. Thus, under the notation of
Theorem~\ref{theorem:super-master}, we may identify vectors $\optdens \in
L^\infty(\statdomain)$ by vectors $\optdens \in \R^{2^k}$. Noting that
$\overline{\packval} = \frac{1}{|\packset|} \sum_{\packval \in \packset}
\packval = 0$ is the mean element of the packing set $\packset = \{-1,
1\}^k$ by our construction, we have $\meanstatdensity(\statval) =
\frac{1}{2^k}$ under the sampling~\eqref{eqn:lp-sampling-scheme} and the
linear functional $\varphi_\packval$ defined in
Theorem~\ref{theorem:super-master} is
\begin{align*}
  \varphi_\packval(\gamma)
  & = 
  \sum_{x \in \{-1, 1\}^k}
  \optdens(x) (\statdensity_\packval(x) - \meanstatdensity(x))
  = \frac{1}{2^k}
  \sum_{x \in \{-1, 1\}^k}
  \optdens(x)
  (1 + \delta \packval^\top \statval - 1)
  = \frac{\delta}{2^k}
  \sum_{x \in \{-1, 1\}^k}
  \optdens(x)
  \packval^\top \statval.
\end{align*}
Define the vector $u_\packval \in \Z^{2^k}$, indexed by $\statval \in
\{-1, 1\}^k$, so that $u_\packval(x) = \packval^\top x$. Identifying the
vector $\optdens \in \R^{2^k}$ with $\optdens : \statdomain \to \R$ under
the same indexing, we then have $\varphi_\packval(\optdens) =
\frac{\delta}{2^k} \optdens^\top u_\packval$ and
\begin{equation*}
  \sum_{\packval \in \packset}
  \varphi_\packval(\optdens)^2
  = \frac{\delta^2}{4^k}
  \optdens^\top \sum_{\packval \in \packset} u_\packval
  u_\packval^\top \optdens.
\end{equation*}

Let the matrix $H_k \in \{-1, 1\}^{2^k \times k}$ be a binary expansion
matrix defined as follows: the $j$th row of $H_k$ corresponds to the $k$-bit
binary expansion of the number $2^{k-1} - j + 1$ (i.e.\ the first row to
$2^{k-1}$, the second to $2^{k-1} - 1$, etc.)  where we replace 0s with -1s
in the binary expansion. Explicitly, define $H_k$ recursively by
\begin{equation*}
  H_k = \left[\begin{matrix} \ones_{2^{k-1}} \\
      -\ones_{2^{k-1}} \end{matrix}
    ~ \ones_2 \otimes H_{k-1}
    \right]
  ~~ \mbox{and} ~~
  H_1 = \left[\begin{matrix} 1 \\ -1 \end{matrix}\right],
\end{equation*}
where $\otimes$ denotes the Kronecker product.
Then
$\sum_\packval u_\packval u_\packval^\top
= (H_k H_k^\top)^2
= H_k H_k^\top H_k H_k^\top$, and
$H_k$ has orthogonal columns. Thus
$\opnorm{H_k} = \sqrt{2^k}$ and
$\opnorms{(H_k H_k^\top)^2} = 4^k$. As a consequence,
using the containment
\begin{equation*}
  \linfset = \left\{\optdens \in \R^{2^k} : \linf{\optdens} \le 1 \right\}
  \subset \left\{\optdens \in \R^{2^k} :
  \ltwo{\optdens}^2 \le 2^k \right\},
\end{equation*}
we have
\begin{equation*}
  \sup_{\optdens \in \linfset(\statdomain)}
  \sum_{\packval \in \packset}
  \varphi_\packval(\optdens)^2
  \le
  \frac{\delta^2}{4^k}
  \sup_{\ltwo{\optdens}^2 \le 2^k}
  \optdens^\top (H_k H_k^\top)^2 \optdens
  \le 2^k \delta^2.
\end{equation*}
Applying Theorem~\ref{theorem:super-master}
to divide by $2^k = \card(\packset)$ yields the result.


\subsection{Proof of Lemma~\ref{lemma:l1-information-bound}}
\label{appendix:l1-information-bound}

It is no loss of generality to assume the radius $\radius = 1$.  We use the
notation of Theorem~\ref{theorem:super-master}, recalling the linear
functionals $\varphi_\packval : L^\infty(\statdomain) \rightarrow \R$.
Because the set $\statdomain = \{-1, 1\}^d$, we can identify vectors
$\optdens \in L^\infty(\statdomain)$ with vectors $\optdens \in
\R^{2^d}$. Moreover, we have (by construction of the sampling scheme) that
$\meanstatdensity(\statval) = 1 / 2^d$, and thus
\begin{equation*}
  \varphi_\packval(\optdens)
  = \sum_{\statsample \in \{-1, 1\}^d}
  \optdens(\statsample) (\statdensity_\packval(\statsample) -
  \meanstatdensity(\statsample))
  = \frac{1}{2^d}
  \sum_{\statsample \in \statdomain} \optdens(\statsample)(1 + \delta
  \packval^\top \statsample - 1) = \frac{\delta}{2^d}
  \sum_{\statsample \in \statdomain} \optdens(\statsample) \packval^\top
  \statsample.
\end{equation*}
For each $\packval \in \packset$, we may construct a vector
$u_\packval \in \{-1, 1\}^{2^d}$, indexed by $\statsample \in \{-1,
1\}^d$, with
\begin{equation*}
  u_\packval(\statsample) = \packval^\top \statsample =
  \begin{cases}
    1 & \mbox{if~} \packval = \pm e_j ~ \mbox{and}~
    \sign(\packval_j) = \sign(\statsample_j) \\ -1 & \mbox{if~}
    \packval = \pm e_j ~ \mbox{and} ~ \sign(\packval_j) \neq
    \sign(\statsample_j).
  \end{cases}
\end{equation*}
For $\packval = e_j$, we see that $u_{e_1}, \ldots, u_{e_d}$ are the
first $d$ columns of the standard Hadamard transform matrix (and
$u_{-e_j}$ are their negatives).  Then we have that
$\sum_{\statsample \in \statdomain} \optdens(x) \packval^\top \statsample =
\optdens^\top u_\packval$, and
\begin{equation*}
  \varphi_\packval(\optdens)^2 = \frac{\delta^2}{4^d}
  \optdens^\top u_\packval u_\packval^\top
  \optdens.
\end{equation*}
Note also that $u_\packval u_\packval^\top = u_{-\packval}
u_{-\packval}^\top$, and as a consequence we have
\begin{equation}
  \label{eqn:l1-packing-quadratic}
  \sum_{\packval \in \packset} \varphi_\packval(\optdens)^2 =
  \frac{\delta^2}{4^d} \optdens^\top \sum_{\packval \in \packset}
  u_\packval u_\packval^\top \optdens = \frac{2\delta^2}{4^d} \optdens^\top
  \sum_{j=1}^d u_{e_j} u_{e_j}^\top \optdens.
\end{equation}

But now, studying the quadratic
form~\eqref{eqn:l1-packing-quadratic}, we note that the vectors
$u_{e_j}$ are orthogonal. As a consequence, the vectors (up to
scaling) $u_{e_j}$ are the only eigenvectors corresponding to
positive eigenvalues of the positive semidefinite matrix
$\sum_{j=1}^d u_{e_j} u_{e_j}^\top$.  Thus, since the set
\begin{equation*}
  \linfset = \left\{\optdens \in \R^{2^d} : \linf{\optdens} \le 1 \right\}
  \subset \left\{\optdens \in \R^{2^d} :
  \ltwo{\optdens}^2 \le 2^d \right\},
\end{equation*}
we have via an eigenvalue calculation that
\begin{align*}
  \sup_{\optdens \in \linfset} \sum_{\packval \in \packset}
  \varphi_\packval(\optdens)^2
  & \le \frac{2 \delta^2}{4^d} \sup_{\optdens :
    \ltwo{\optdens}^2 \le 2^d}
  \optdens^\top \sum_{j=1}^d u_{e_j} u_{e_j}^\top \optdens \\
  & = \frac{2 \delta^2}{4^d} \ltwo{u_{e_1}}^4 =
  2 \delta^2,
\end{align*}
since $\ltwos{u_{e_j}}^2 = 2^d$ for each $j$.  Applying
Theorem~\ref{theorem:super-master} to divide by $2d$ completes the proof.


\subsection{Proof of Lemma~\ref{lemma:density-kls}}
\label{appendix:density-kls}

This result relies on
Theorem~\ref{theorem:sequential-interactive}, along with a careful argument to
understand the extreme points of $\optdens \in L^\infty([0, 1])$ that we use
when applying the result. First, we take the packing $\packset = \{-1,
1\}^\numderiv$ and densities $f_\packval$ for $\packval \in \packset$ as in
the construction~\eqref{eqn:density-packer}.  Overall, our first step is to
show for the purposes of applying Theorem~\ref{theorem:sequential-interactive},
it is
no loss of generality to identify $\optdens \in L^\infty([0, 1])$ with vectors
$\optdens \in \R^{2 \numbin}$, where $\optdens$ is constant on intervals of
the form $[i/2\numbin, (i + 1)/2\numbin]$. With this identification complete,
we can then provide a bound on the correlation of any $\optdens \in \linfset$
with the densities $f_{\pm j}$ defined in~\eqref{eqn:f-plus-j}, which
completes the proof.

With this outline in mind, let the sets $\MySet_i$, $i \in \{1, 2,
\ldots, 2\numbin\}$, be defined as $\MySet_i =
\openright{(i-1)/2\numbin}{i/2\numbin}$ except that $\MySet_{2\numbin}
= [(2\numbin - 1)/2\numbin, 1]$, so the collection
$\{\MySet_i\}_{i=1}^{2 \numbin}$ forms a partition of the unit
interval $[0, 1]$. By construction of the densities $f_\packval$, the
sign of $f_\packval - 1$ remains constant on each
$\MySet_i$. Let us define (for shorthand) the
linear functionals $\varphi_j : L^\infty([0, 1]) \to \R$
for each $j \in \{1, \ldots, \numbin\}$ via
\begin{equation*}
  \varphi_j(\optdens)
  \defeq \int \optdens (d\statprob_{+j}
  - d\statprob_{-j})
  = \sum_{i = 1}^{2 \numbin}
  \int_{\MySet_i} \optdens(x) (f_{+j}(x) - f_{-j}(x)) dx
  = 2 \int_{\MySet_{2j - 1} \cup \MySet_{2j}}
  \!\!\!\! \optdens(x) g_{\numderiv,j}(x)
  dx,
\end{equation*}
where we recall the definitions~\eqref{eqn:f-plus-j} of the mixture
densities $f_{\pm j} = 1 \pm g_{\numderiv,j}$.
Since the set $\linfset$ from
Theorem~\ref{theorem:sequential-interactive} is compact, convex,
and Hausdorff, the Krein-Milman theorem~\cite[Proposition
  1.2]{Phelps01} guarantees that it is equal to the convex hull of its
extreme points; moreover, since the functionals $\optdens \mapsto
\varphi^2_j(\optdens)$ are convex, the supremum in
Theorem~\ref{theorem:sequential-interactive} must be attained at
the extreme points of $\linfset([0,1])$.  As a consequence, when applying the
divergence bound
\begin{equation}
  \sum_{j=1}^\numbin \left(\dkl{\marginprob_{+j}^n}{\marginprob_{-j}^n}
  + \dkl{\marginprob_{-j}^n}{\marginprob_{+j}^n}\right)
  \le 2 n (e^\diffp - 1)^2
  \sup_{\optdens \in \linfset}
  \sum_{j = 1}^\numbin \varphi_j^2(\optdens),
  \label{eqn:kl-sup-densities}
\end{equation}
we can restrict our attention to $\optdens \in \linfset$ for
which $\optdens(x) \in \{-1, 1\}$.

Now we argue that it is no loss of generality to assume that $\optdens$, when
restricted to $\MySet_i$, is a constant (apart from a measure zero set). Fix
$i \in [2\numbin]$, and assume for the sake of contradiction that there exist
sets $B_i, C_i \subset \MySet_i$ such that $\optdens(B_i) = \{1\}$ and
$\optdens(C_i) = \{-1\}$, while $\lebesgue(B_i) > 0$ and $\lebesgue(C_i) > 0$
where $\lebesgue$ denotes Lebesgue measure.\footnote{For a function $f$ and
  set $A$, the notation $f(A)$ denotes the image $f(A) = \{f(x) \mid x \in
  A\}$.}  We will construct vectors $\optdens_1$ and $\optdens_2 \in \linfset$
and a value $\lambda \in (0, 1)$ such that
\begin{equation*}
  \int_{\MySet_i} \optdens(x) g_{\numderiv,j}(x)
  d x =
  \lambda \int_{\MySet_i} \optdens_1(x) g_{\numderiv,j}(x) dx
  + (1 - \lambda) \int_{\MySet_i} \optdens_2(x) g_{\numderiv,j}(x) dx
\end{equation*}
simultaneously for all $j \in [\numbin]$, while
on $\MySet_i^c = [0, 1] \setminus \MySet_i$, we will have the equivalence
\begin{equation*}
  \left.\optdens_1\right|_{\MySet_i^c}
  \equiv \left.\optdens_2 \right|_{\MySet_i^c}
  \equiv \left.\optdens\right|_{\MySet_i^c}.
\end{equation*}
Indeed, set $\optdens_1(\MySet_i) = \{1\}$ and $\optdens_2(\MySet_i) =
\{-1\}$, otherwise setting $\optdens_1(x) = \optdens_2(x) = \optdens(x)$ for
$x \not \in \MySet_i$. For the unique index $j \in [\numbin]$ such that
$[(j-1)/\numbin, j/\numbin] \supset D_i$, we define
\begin{equation*}
  \lambda \defeq \frac{\int_{B_i} g_{\numderiv,j}(x) dx}{
    \int_{\MySet_i} g_{\numderiv,j}(x) dx}
  ~~~ \mbox{so} ~~~
  1 - \lambda = \frac{\int_{C_i} g_{\numderiv,j}(x) dx}{
    \int_{\MySet_i} g_{\numderiv,j}(x) dx}.
\end{equation*}
By the construction of the function $g_\numderiv$, the functions
$g_{\numderiv,j}$ do not change signs on $\MySet_i$, and the absolute
continuity conditions on $g_\numderiv$ specified in
equation~\eqref{eqn:absolute-continuity-bumps} guarantee $1 > \lambda
> 0$, since $\lebesgue(B_i) > 0$ and $\lebesgue(C_i) > 0$. 
We thus find that for any $j \in [\numbin]$,
\begin{align*}
  \int_{\MySet_i} \optdens(x)  g_{\numderiv,j}(x) dx
  & = \int_{B_i} \optdens_1(x) g_{\numderiv,j}(x) dx
  + \int_{C_i} \optdens_2(x) g_{\numderiv,j}(x) dx \\
  & = \int_{B_i} g_{\numderiv,j}(x)dx
  - \int_{C_i} g_{\numderiv,j}(x) dx
  = \lambda \int_{\MySet_i} g_{\numderiv,j}(x)dx
  - (1 - \lambda) \int_{\MySet_i} g_{\numderiv,j}(x)dx \\
  & = \lambda \int \optdens_1(x)
  g_{\numderiv,j}(x)dx + (1 - \lambda) \int \optdens_2(x)
  g_{\numderiv,j}(x)dx.
\end{align*}
(Notably, for $j$ such that $g_{\numderiv,j}$ is identically 0 on $\MySet_i$,
this equality is trivial.) By linearity and the strong convexity of the
function $x \mapsto x^2$, then, we find that
for sets $E_j \defeq \MySet_{2j - 1} \cup \MySet_{2j}$,
\begin{align*}
  \sum_{j=1}^\numbin \varphi_j^2(\optdens)
  & = \sum_{j=1}^\numbin
  \left(\int_{E_j}
  \optdens(x) g_{\numderiv,j}(x) dx\right)^2 \\
  & < \lambda \sum_{j = 1}^\numbin
  \left(\int_{E_j} \optdens_1(x)
  g_{\numderiv,j}(x) dx\right)^2
  + (1 - \lambda) \sum_{\packval \in
    \packset} \left(\int_{E_j} \optdens_2(x)
  g_{\numderiv,j}(x)dx \right)^2.
\end{align*}
Thus one of the densities $\optdens_1$ or $\optdens_2$ must have a
larger objective value than $\optdens$. This is our desired
contradiction, which shows that (up to
measure zero sets) any $\optdens$ attaining the supremum in the
information bound~\eqref{eqn:kl-sup-densities} must be
constant on each of the $\MySet_i$.

\newcommand{\linfsetfinite}[1]{\mc{B}_{1,#1}} 

Having shown that $\optdens$ is constant on each of the intervals
$\MySet_i$, we conclude that the
supremum~\eqref{eqn:kl-sup-densities} can be reduced to a
finite-dimensional problem over the subset
\begin{align*}
  \linfsetfinite{2\numbin} \defeq \left\{u \in \R^{2\numbin}
  \mid \linf{u} \le 1 \right\}
\end{align*}
of $\R^{2\numbin}$.
In terms of this subset, the supremum~\eqref{eqn:kl-sup-densities}
can be rewritten as the upper bound
\begin{align*}
  \sup_{\optdens \in \linfset} \sum_{j=1}^\numbin
  \varphi_j(\optdens)^2
  & \le \sup_{\optdens \in
  \linfsetfinite{2 \numbin}}
  \sum_{j=1}^\numbin
  \bigg(\optdens_{2j-1}
  \int_{\MySet_{2j - 1}} g_{\numderiv,j}(x) dx
  + \optdens_{2j}
  \int_{\MySet_2j} g_{\numderiv,j}(x) dx \bigg)^2.
\end{align*}
By construction of the function $g_\numderiv$, we have the
equality
\begin{equation*}
  \int_{\MySet_{2j-1}} g_{\numderiv,j}(x) dx
  = -\int_{\MySet_{2j}} g_{\numderiv,j}(x) dx
  = \int_0^{\frac{1}{2 \numbin}} g_{\numderiv,1}(x) dx
  = \int_0^{\frac{1}{2 \numbin}}
  \frac{1}{\numbin^\numderiv} g_\numderiv(\numbin x) dx
  = \frac{c_{1/2}}{\numbin^{\numderiv + 1}}.
\end{equation*}
This implies that
\begin{align}
  \lefteqn{\frac{1}{2e^\diffp (e^\diffp - 1)^2 n}
    \sum_{j=1}^\numbin \left(\dkl{\marginprob_{+j}^n}{\marginprob_{-j}^n}
    + \dkl{\marginprob_{+j}^n}{\marginprob_{-j}^n}\right)
    \le 
    \sup_{\optdens \in \linfset} \sum_{j=1}^\numbin
    \varphi_j(\optdens)^2 } \nonumber \\
  & \le \sup_{\optdens \in \linfsetfinite{2
      \numbin}} \sum_{j=1}^\numbin
  \left(\frac{c_{1/2}}{\numbin^{\numderiv + 1}} \optdens^\top
  (e_{2j - 1} - e_{2j})\right)^2
  = \frac{c_{1/2}^2}{\numbin^{2
      \numderiv + 2}} \sup_{\optdens \in \linfsetfinite{2 \numbin}}
  \optdens^\top \sum_{j=1}^\numbin
  (e_{2j - 1} - e_{2j})(e_{2j - 1} - e_{2j})^\top
  \optdens,
  \label{eqn:basis-density-bound}
\end{align}
where $e_j \in \R^{2 \numbin}$ denotes the $j$th standard basis vector.
Rewriting this using the Kronecker product $\otimes$, we have
\begin{equation*}
  \sum_{j=1}^\numbin (e_{2j - 1} - e_{2j})(e_{2j - 1} - e_{2j})^\top
  = I_{\numbin \times \numbin} \otimes \left[\begin{matrix} 1 & -1 \\
      -1 & 1 \end{matrix}\right]
  \preceq 2 I_{2\numbin \times 2 \numbin}.
\end{equation*}
Combining this bound with our
inequality~\eqref{eqn:basis-density-bound}, we obtain
\begin{align*}
  \sum_{j=1}^\numbin \left(\dkl{\marginprob_{+j}^n}{\marginprob_{-j}^n}
  + \dkl{\marginprob_{+j}^n}{\marginprob_{-j}^n}\right)
  \le
  4 n (e^\diffp - 1)^2
  \frac{c_{1/2}^2}{\numbin^{2 \numderiv + 2}}
  \sup_{\optdens \in
    \linfsetfinite{2 \numbin}} \ltwo{\optdens}^2
  = 4 c_{1/2}^2
  \frac{n (e^\diffp - 1)^2}{\numbin^{2 \numderiv + 1}}.
\end{align*}


\section{Technical arguments}

In this appendix, we collect proofs of technical lemmas and results
needed for completeness.

\subsection{Proof of Lemma~\ref{lemma:sharp-assouad}}
\label{sec:proof-sharp-assouad}

Fix an (arbitrary) estimator $\what{\optvar}$. By
assumption~\eqref{eqn:risk-separation}, we have
\begin{equation*}
  \Phi(\metric(\optvar, \optvar(\statprob_\packval)))
  \ge 2 \delta \sum_{j=1}^d \indic{[\maptocube(\optvar)]_j
    \neq \packval_j}.
\end{equation*}
Taking expectations, we see that
\begin{align*}
  \sup_{\statprob \in \mc{\statprob}} \E_\statprob 
  \left[\Phi(\metric(\what{\optvar}(\channelrv_1, \ldots, \channelrv_n),
    \optvar(\statprob)))\right]
  & \ge 
  \max_{\packval \in \packset} \E_{\statprob_\packval}
  \left[\Phi(\metric(\what{\optvar}(\channelrv_1, \ldots, \channelrv_n),
    \optvar_\packval))\right] \\
  & \ge \frac{1}{|\packset|} \sum_{\packval \in \packset}
  \E_{\statprob_\packval}
  \left[\Phi(\metric(\what{\optvar}(\channelrv_1, \ldots, \channelrv_n),
    \optvar_\packval))\right] \\
  & \ge \frac{1}{|\packset|} \sum_{\packval \in \packset}
  2 \delta \sum_{j=1}^d \E_{\statprob_\packval}\left[
    \indic{[\psi(\what{\optvar})]_j \neq \packval_j}\right],
\end{align*}
as the average is smaller than the maximum of a set and using the separation
assumption~\eqref{eqn:risk-separation}.  Recalling the
definition~\eqref{eqn:paired-mixtures} of the mixtures $\statprob_{\pm j}$, we
swap the summation orders to see that
\begin{align*}
  \frac{1}{|\packset|} \sum_{\packval \in \packset}
  \statprob_{\packval}\left([\maptocube(\what{\optvar})]_j
  \neq \packval_j\right)
  & = \frac{1}{|\packset|} \sum_{\packval : \packval_j = 1}
  \statprob_{\packval}\left([\maptocube(\what{\optvar})]_j
  \neq \packval_j\right)
  + \frac{1}{|\packset|} \sum_{\packval : \packval_j = -1}
  \statprob_{\packval}\left([\maptocube(\what{\optvar})]_j
  \neq \packval_j\right)
  \\
  & = \half \statprob_{+j}
  \left([\maptocube(\what{\optvar})]_j \neq \packval_j\right)
  + \half \statprob_{-j}
  \left([\maptocube(\what{\optvar})]_j \neq \packval_j\right).
\end{align*}
This gives the statement claimed in the lemma, while taking an infimum over
all testing procedures $\test : \channeldomain^n \to \{-1, +1\}$ gives the
claim~\eqref{eqn:sharp-assouad}.

\subsection{Proof of unbiasedness for sampling
  strategy~\eqref{eqn:ltwo-sampling}}
\label{appendix:ltwo-sampling}

We compute the expectation of a random variable $\channelrv$ sampled
according to the strategy~\eqref{eqn:ltwo-sampling}; i.e., we compute
$\E[\channelrv \mid v]$ for a vector $v \in \R^d$. By scaling, it is no loss
of generality to assume that $\ltwo{v} = 1$, and using the rotational
symmetry of the $\ell_2$-ball, we see it is no loss of generality to assume
that $v = e_1$, the first standard basis vector.
  
Let the function $s_d$ denote the surface area of the sphere in
$\R^d$, so that
\begin{equation*}
  s_d(r) = \frac{d \pi^{d/2}}{\Gamma(d/2 + 1)} r^{d-1}
\end{equation*}
is the surface area of the sphere of radius $r$. (We use $s_d$ as a
shorthand for $s_d(1)$ when convenient.) Then for a random variable
$W$ sampled uniformly from the half of the $\ell_2$-ball with first
coordinate $W_1 \ge 0$, symmetry implies that by integrating over the
radii of the ball,
\begin{equation*}
  \E[W] = e_1 \frac{2}{s_d} \int_0^1 s_{d-1}\left(\sqrt{1 -
    r^2}\right) r dr.
\end{equation*}
Making the change of variables to spherical coordinates (we use $\phi$
as the angle), we have
\begin{equation*}
  \frac{2}{s_d} \int_0^1 s_{d-1}\left(\sqrt{1 - r^2}\right) r dr =
  \frac{2}{s_{d}} \int_0^{\pi/2} s_{d-1}\left(\cos \phi \right) \sin
  \phi\, d\phi = \frac{2 s_{d-1}}{s_d} \int_0^{\pi/2} \cos^{d - 2}(\phi)
  \sin(\phi) \, d\phi.
\end{equation*}
Noting that $\frac{d}{d \phi} \cos^{d - 1}(\phi) = -(d - 1) \cos^{d -
  2}(\phi) \sin(\phi)$, we obtain
\begin{equation*}
  \int_0^{\pi/2} \cos^{d - 2}(\phi) \sin(\phi)
  \,d\phi = \left.-\frac{\cos^{d - 1}(\phi)}{d - 1}\right|_0^{\pi/2} =
  \frac{1}{d - 1}.
\end{equation*}
Thus
\begin{equation}
  \E[W] = 2 e_1 \frac{(d - 1) \pi^{\frac{d - 1}{2}} \Gamma(\frac{d}{2} +
    1)}{ d \pi^{\frac{d}{2}} \Gamma(\frac{d - 1}{2} + 1)} \frac{1}{d -
    1} = e_1 \underbrace{ \frac{2 \Gamma(\frac{d}{2} + 1)}{\sqrt{\pi} d
      \Gamma(\frac{d - 1}{2} + 1)}}_{\eqdef c_d},
  \label{eqn:ltwo-halfspace-expectation}
\end{equation}
where we define the constant $c_d$ to be the final ratio.

Allowing again $\ltwo{v} \le \radius$, with the
expression~\eqref{eqn:ltwo-halfspace-expectation}, we see that for our
sampling strategy for $\channelrv$, we have
\begin{equation*}
  \E[\channelrv \mid v] = v \frac{\sbound}{\radius} c_d \left(
  \frac{e^\diffp}{e^\diffp + 1} - \frac{1}{e^\diffp + 1} \right) =
  \frac{\sbound}{\radius} c_d \frac{e^\diffp - 1}{e^\diffp + 1}.
\end{equation*}
Consequently, the choice
\begin{equation*}
  \sbound = \frac{e^\diffp + 1}{e^\diffp - 1} \frac{\radius}{c_d} =
  \frac{e^\diffp + 1}{e^\diffp - 1} \frac{\radius \sqrt{\pi} d
    \Gamma(\frac{d - 1}{2} + 1)}{2 \Gamma(\frac{d}{2} + 1)}
\end{equation*}
yields $\E[\channelrv \mid v] = v$. Moreover, we have
\begin{equation*}
  \ltwo{\channelrv} = \sbound \le \radius \frac{e^\diffp + 1}{e^\diffp - 1}
  \frac{3 \sqrt{\pi} \sqrt{d}}{4}
\end{equation*}
by Stirling's approximation to the $\Gamma$-function. By noting that
$(e^\diffp + 1) / (e^\diffp - 1) \le 3 / \diffp$ for $\diffp \le 1$,
we see that $\ltwo{\channelrv} \le 4 \radius \sqrt{d} / \diffp$.

\subsection{Proof of unbiasedness for
  sampling strategy~\eqref{eqn:linf-sampling}}
\label{appendix:linf-sampling}

We compute conditional expectations for each of the uniform quantities in
the sampling scheme~\eqref{eqn:linf-sampling}. This argument is based on
Corollary 3.7 of \citet{DuchiJoWa14}, but we present it here for clarity and
completeness. In each summation to follow, we implicitly assume that $z \in
\{-1, 1\}^d$ and that $x \in \{-1, 1\}^d$, as the general result follows by
scaling of these quantities. We consider first the case that $d$ is odd. In
this case, we have by symmetry that
\begin{align*}
  \sum_{z : \<z, x\> \ge 0} z
  & = \sum_{z : \<z, x\> = 1} z
  + \sum_{z : \<z, x\> = 3} z
  + \cdots + \sum_{z : \<z, x\> = d} z \\
  & = \left[\binom{d - 1}{\frac{d - 1}{2}}
    - \binom{d - 1}{\frac{d + 1}{2}}\right] x
  + \left[\binom{d - 1}{\frac{d + 1}{2}}
    - \binom{d - 1}{\frac{d + 3}{2}}\right] x
  + \cdots + \binom{d - 1}{d - 1} x
  = \binom{d - 1}{\frac{d - 1}{2}} x,
\end{align*}
and as $|\{z \in \{-1, 1\}^d \mid \<z, x\> > 0\}| = 2^{d - 1}$, we obtain
that for $Z \sim \uniform(\{-1, 1\}^d)$ we have
\begin{equation*}
  \E[Z \mid \<Z, x\> \ge 0]
  = \frac{1}{2^{d-1}} \binom{d - 1}{\frac{d - 1}{2}} x
  ~~ \mbox{and} ~~
  \E[Z \mid \<Z, x\> \le 0]
  = -\frac{1}{2^{d-1}} \binom{d - 1}{\frac{d - 1}{2}} x.
\end{equation*}
Similarly, for $d$ even, we find that
\begin{align*}
  \sum_{z : \<z, x\> \ge 0} z
  & = \sum_{z : \<z, x\> = 2} z
  + \sum_{z : \<z, x\> = 4} z
  + \cdots + \sum_{z : \<z, x\> = d} z \\
  & = \left[\binom{d-1}{\frac{d}{2}} - \binom{d - 1}{\frac{d}{2} + 1}\right] x
  + \left[\binom{d-1}{\frac{d}{2} + 1} - \binom{d - 1}{\frac{d}{2} + 2}\right] x
  + \cdots + \binom{d - 1}{d-1} x
  = \binom{d - 1}{\frac{d}{2}} x.
\end{align*}
Noting that the set
$\{z \in \{-1, 1\}^d \mid \<z, x\> \ge 0\}$ has cardinality
$2^{d - 1} + \half \binom{d}{d/2}$ (because we must consider vectors 
$z$ such that $\<z, x\> = 0$),
we find that for $d$ even we have
\begin{equation*}
  \E[Z \mid Z^\top x \ge 0]
  = \frac{1}{2^{d - 1} + \half \binom{d}{d/2}}
  \binom{d - 1}{\frac{d}{2}} x
  ~~ \mbox{and} ~~
  \E[Z \mid Z^\top x \le 0]
  = -\frac{1}{2^{d - 1} + \half \binom{d}{d/2}}
  \binom{d - 1}{\frac{d}{2}} x.
\end{equation*}
Inverting the constant multipliers on the vectors $x$ in the
preceding equations shows that the strategy~\eqref{eqn:linf-sampling}
is unbiased.




\end{document}